\documentclass{article}

\title{Lie rings in finite-dimensional theories}
\author{Moreno Invitti }
\date{}
\usepackage{graphicx} 
\usepackage{amsfonts}
\usepackage[utf8]{inputenc}
\usepackage{xcolor}

\usepackage{amssymb}
\usepackage{tikz-cd}
 \usepackage{amscd,amsmath}
\usepackage{amssymb}
\usepackage{amsthm}
\def\Ind#1#2{#1\setbox0=\hbox{$#1x$}\kern\wd0\hbox to 0pt{\hss$#1\mid$\hss}
\lower.9\ht0\hbox to 0pt{\hss$#1\smile$\hss}\kern\wd0}

\def\notind#1#2{#1\setbox0=\hbox{$#1x$}\kern\wd0
\hbox to 0pt{\mathchardef\nn=12854\hss$#1\nn$\kern1.4\wd0\hss}
\hbox to 0pt{\hss$#1\mid$\hss}\lower.9\ht0 \hbox to 0pt{\hss$#1\smile$\hss}\kern\wd0}

\usepackage{amsmath} 
\usepackage{amssymb}
\usepackage{wasysym}
\usepackage{tikz-cd}
\usepackage{amsmath}
\usepackage{subfigure}
\newtheorem{theorem}{Theorem}[section]
\newtheorem{corollary}[theorem]{Corollary}

\newtheorem*{claim}{Claim}
\newtheorem*{conjecture}{Conjecture}
\newtheorem*{observation}{Observation}
\newtheorem{lemma}[theorem]{Lemma}
\newtheorem{fact}[theorem]{Fact}

\newtheorem*{defn}{Definition}

\newtheorem*{Char0}{Theorem}

\newtheorem*{thmB,2}{Theorem B, Second version}
\usepackage{amssymb}
\usepackage{tikz-cd}
\usepackage{mathptmx} 
\usepackage{graphicx} 
\usepackage{refcheck}
\usepackage{hyperref}
 \newenvironment{claimproof}[1]{\par\noindent\textit{Proof:}\space#1}{\hfill $\blacksquare$}

\DeclareMathOperator{\ad}{ad}

\usepackage{amssymb}
\usepackage{tikz-cd}
\usepackage{mathptmx} 
\usepackage{graphicx} 
\usepackage{refcheck}
\usepackage{hyperref}
\title{Lie rings in finite-dimensional theories}
\author{Moreno Invitti}
\begin{document}

\maketitle
\begin{abstract}
    We study Lie rings definable in a finite-dimensional theory, extending the results for the finite Morley rank case. In particular, we prove a classification of Lie rings of dimension up to four in the NIP or connected case. In characteristic $0$, we verify a version of the Cherlin-Zilber Conjecture. Moreover, we characterize the actions of some classes, namely abelian, nilpotent and soluble, of Lie rings of finite dimension. Finally, we show the existence of definable envelopes for nilpotent and soluble Lie rings. These results are used to verify that the Fitting and the Radical ideal of a Lie ring of finite dimension are both definable and respectively nilpotent and soluble.
\end{abstract}
\section*{Introduction}
Lie algebras are a fundamental field of study in algebra. Lots of research has been focused on Lie algebras over algebraically closed fields in characteristic $0$, in particular over $\mathbb{C}$. The main result in this direction is a complete classification of simple Lie algebras of finite dimension over $\mathbb{C}$, and, in general, over algebraically closed fields of characteristic $0$ (see \cite{serre1966algebres}). Another very important result derived from the previous one has been obtained by Cartan, which classifies the simple Lie algebras of finite dimension over $\mathbb{R}$ (see \cite{helgason1979differential} for the classification and \cite{knapp1996quick} for a short proof).\\
We have to wait until 2006 until Block, Strade, Premet, and Wilson (see \cite{block1988classification} and \cite{premet2006classification})  complete the classification of Lie algebras over an algebraically closed field of characteristic $p>3$, with a way more complicated proof.\\
A natural extension of Lie algebras are Lie rings. These are, intuitively, Lie algebras without any underlying vector space. Their role in group theory is fundamental. One of the main result is the famous Lazard correspondence between groups with an $N$-series and Lie rings \cite{lazard1954groupes}. This correspondence was used by Zelmanov in the solution of the restricted Burnside's problem (see \cite{zel1991solution} and \cite{zel1992solution}). Other applications of the Lazard correspondence can be found in \cite{shumyatsky2019applications}.
A infinite Lie rings are unclassifiable, unless we add some notion of model theoretic "tameness". The case in which "tame" is substituted by finite Morley rank has been studied by different authors, as Rosengarten  in \cite{rosengarten1991aleph}, Deloro and Ntisiri in \cite{deloro2023simple} and \cite{deloro2025soluble}. The line of research in this context is the study of a version of Cherlin-Zilber conjecture for Lie rings.
\begin{conjecture}
    Let $\mathfrak{g}$ be a definable simple Lie ring of finite Morley rank. Then, if the characteristic of $\mathfrak{g}$ is enough large, $\mathfrak{g}$ is a Lie algebra of finite dimension over an algebraically closed field.
\end{conjecture}
This conjecture has been verified in characteristic $0$ (see \cite[Theorem 3.0]{deloro2023simple}). In characteristic $p\not=0$, we are way far from proving (or confuting) this conjecture. A partial answer is given by the classification on connected Lie rings of Morley rank up to $4$, given by Rosengarten in \cite{rosengarten1991aleph} for the cases of Morley rank up to $3$ and Deloro, Ntisiri in the case of Morley rank $4$.
\begin{theorem}
  Let $\mathfrak{g}$ be a connected definable Lie ring of finite Morley rank. Then, 
  \begin{itemize}
      \item If $\operatorname{RM}(\mathfrak{g})=1$, $\mathfrak{g}$ is abelian \cite[Corollary 4.1.1]{rosengarten1991aleph};
      \item If $\operatorname{RM}(\mathfrak{g})=2$, then $\mathfrak{g}$ is soluble. Moreover, if it is not nilpotent, $\mathfrak{g}/Z(\mathfrak{g})\simeq \mathfrak{g}\mathfrak{a}_1(K)$ for a algebraically closed field $K$ \cite[Theorem 4.2.1]{rosengarten1991aleph};
      \item If $\operatorname{RM}(\mathfrak{g})=3$ and $\mathfrak{g}$ is not soluble, then $\mathfrak{g}\simeq \operatorname{sl}_2(K)$ for an algebraically closed field $K$ \cite[Theorem 4.4.1]{rosengarten1991aleph};
      \item If $\operatorname{RM}(\mathfrak{g})=4$, then $\mathfrak{g}$ cannot be simple \cite[Theorem 4]{deloro2023simple}.
  \end{itemize}
\end{theorem}
Another family of results are the linearisation theorems for Lie rings of finite Morley rank. In particular, Deloro and Ntisiri have verified the following result (\cite[Theorem A]{deloro2025soluble}) that they used to proved a version of the "Lie-Malcev" theorem.
\begin{theorem}
  Let $\mathfrak{g}$ be a connected definable Lie ring of finite Morley rank and $V$ a minimal connected $\mathfrak{g}$-module. Assume that $\mathfrak{g}$ has a definable abelian ideal $\mathfrak{a}$ that does not act trivially on $V$. Then, $\mathfrak{a}\leq Z(\mathfrak{g})$ and there exists an algebraically closed field $K$ such that $V$ is a $K$-vector space of finite dimension, $\mathfrak{a}$ acts $K$-scalarly and $\mathfrak{g}$ acts $K$-linearly on $V$.
\end{theorem}
The first result in a broader context has been obtained by Zamour in \cite{zamour2025quelques}, where he proves that for stable Lie rings the Engel's theorem is true (this has been proved with the $\mathfrak{M}_c$-assumption, which generalises stability) and the existence of nilpotent and soluble envelopes.\\
The article's aim is to generalize these results in the context of finite-dimensional Lie rings. The notion of dimensional theory has been introduced by Frank Wagner in \cite{wagner2020dimensional}. We are interested in a particular subclass of dimensional theories, called finite-dimensional theories.
\begin{defn}
    A theory $T$ is \emph{finite-dimensional} if there exists a function $\operatorname{dim}$ from the class of all interpretable subsets in any model $\mathcal{M}$ of $T$ into $\omega\cup\{-\infty\}$ such that for any $\phi(x,y)$ formula, $X, Y$ interpretable sets in $T$ and $f$ an interpretable function from $X$ to $Y$, then:
\begin{itemize}
    \item If $a,a'$ have same type over $\emptyset$, $\operatorname{dim}(\phi(x,a))=\operatorname{dim}(\phi(x,a'))$ 
    \item $\operatorname{dim}(\emptyset)=-\infty$ and $\operatorname{dim}(X)=0$ if and only if $X$ is finite;
    \item $\operatorname{dim}(X\cup Y)=\max\{\operatorname{dim}(X),\operatorname{dim}(Y)\}$;
    \item If $\operatorname{dim}(f^{-1}(y))\geq k$ if for any $y\in Y$, then $\operatorname{dim}(X)\geq \operatorname{dim}(Y)+k$;
    \item If $\operatorname{dim}(f^{-1}(y))\leq k$ for any $y\in Y$, then $\operatorname{dim}(X)\leq \operatorname{dim}(Y)+k$.
\end{itemize}
\end{defn}
Examples of finite-dimensional theories are $o$-minimal, superstable, and supersimple theories of finite $(S)U$-rank. Since theories of finite Morley rank are in particular superstable of finite $SU$-rank, any Lie ring of finite Morley rank is finite-dimensional (even if the Morley rank is not a dimension in the previous definition \cite{lachlan1980singular}). On the other hand, the class of Lie rings that are finite-dimensional is way bigger than finite Morley rank, since it contains all Lie rings definable in an $o$-minimal theory. In particular, Lie algebras over $\mathbb{R}$ or Lie algebras over pseudofinite fields are examples of finite-dimensional Lie ring that are not of finite Morley rank. We can still hope to obtain a generalisation of the "log-CZ" conjecture. The notion of absolute simplicity will be defined in section $4$.
\begin{conjecture}
    Let $\mathfrak{g}$ be an absolutely simple definable Lie ring of finite dimension and enough large characteristic. Then, there is a definable Lie subring $\mathfrak{h}$ of finite index in $\mathfrak{g}$ and a finite Lie subring $\mathfrak{h}_0$ such that $\mathfrak{h}/\mathfrak{h}_0$ is a Lie algebra of finite dimension over a definable perfect field $K$. 
\end{conjecture}
The assumption about the perfection of the field follows easily from the fact that a field of finite dimension is always perfect.\\
We prove that the conjecture holds when the characteristic is $0$. 
\begin{Char0}\label{Char0}
  Let $(\mathfrak{g},V)$ be definable module of finite dimension with $char(\mathfrak{g})=0$. Assume that $V$ is minimal and that the action is not almost trivial. Then, $\mathfrak{g}/\widetilde{C}_{\mathfrak{g}}(V)$ embeds in $\operatorname{gl}_n(K)$ for a certain $n\in \omega$ and a definable field $K$ of characteristic $0$. Moreover, it is virtually connected and $(\mathfrak{g}/\widetilde{C}_{\mathfrak{g}}(V))^0$ is a Lie algebra over a definable finite-dimensional field of characteristic $0$.\\
  In particular, if $\mathfrak{g}$ is an absolutely simple Lie ring, then there exists a definable Lie subring of finite index such that $\mathfrak{h}$ is a simple Lie algebra of finite dimension over a definable field of characteristic $0$.
\end{Char0}
In finite characteristic (even greater than $3$) the conjecture is open also in dimension $3$, since we have the so-called bad Lie rings.\\
This is not the only problem: also the structure of Lie rings of dimension 1 is not clear, and the technique used for finite Morley rank cannot be applied. On the other hand, assuming the additional hypothesis of connectedness or of NIP, the instrument used in \cite{deloro2023simple} lead to similar results. In particular, the second one is interesting since no result (without assuming the f.s.g. condition, see \cite{ealy2008superrosy}) has been proved for NIP groups of finite dimension.\\
In the first section, we introduce some preliminary notions necessary for this article. In the following section, we verify that a simple non abelian finite-dimensional connected Lie ring has DCC. More precisely, we will prove the following theorem.
\begin{theorem}
   Let $\mathfrak{g}$ a connected Lie ring definable in a finite-dimensional theory. Then, there exists a series
    $$\mathfrak{g}=\mathfrak{h}_0\geq\mathfrak{h}_1\geq...\geq \mathfrak{h}_n=\{0\}$$
    of definable ideals such that $\mathfrak{h_i}/\mathfrak{h_{i+1}}$ is either abelian or it is connected, $\mathfrak{g}$-minimal and it has the DCC.
\end{theorem}
This result is used in the third section to derive the following characterization of connected Lie rings of dimension $\leq 4$ and characteristic $>3$.
\begin{theorem}
  Let $\mathfrak{g}$ be a connected definable Lie ring of finite dimension and of characteristic $>3$. Then,
  \begin{itemize}
      \item if $\dim(\mathfrak{g})=1$, $\mathfrak{g}$ is abelian;
      \item if $\dim(\mathfrak{g})=2$, $\mathfrak{g}$ is soluble. If it is not nilpotent, $\mathfrak{g}/Z(\mathfrak{g})\simeq \mathfrak{g}\mathfrak{a}_1(K)$ for a perfect definable field $K$;
      \item if $\dim(\mathfrak{g})=3$ and $\mathfrak{g}$ is not soluble, $\mathfrak{g}$ is bad;
      \item if $\dim(\mathfrak{g})=4$ and $\mathfrak{g}$ is simple, $\mathfrak{g}$ is bad.
   \end{itemize}
\end{theorem}
In the fourth section, we introduce some notation and easy results on "almost Lie ring theory", defining concepts as the almost centraliser and the almost center of a Lie ring. Moreover, we study the properties of absolutely simple Lie rings. Finally, we use a result of \cite{InvittiDim} to derive a linearization theorem for abelian Lie rings of finite dimension.\\
These results are used prove the "log-CZ"-conjecture in characteristic $0$ in the fifth section.\\
Another consequence of the linearization is the classification of NIP Lie rings of small dimension and characteristic $>3$. 
\begin{theorem}
  Let $\mathfrak{g}$ be a NIP definable Lie ring of finite dimension. Then,
  \begin{itemize}
      \item If $\dim(\mathfrak{g})=1$, $\mathfrak{g}$ is virtually abelian;
      \item if $\dim(\mathfrak{g})=2$, $\mathfrak{g}$ is virtually soluble. If it is not virtually nilpotent, $\mathfrak{g}$ is virtually connected;
      \item if $\dim(\mathfrak{g})=3$ and $\mathfrak{g}$ is not virtually soluble, it is virtually connected;
      \item if $\dim(\mathfrak{g}=4$ and $\mathfrak{g}$ is absolutely simple, $\mathfrak{g}$ is virtually connected.
  \end{itemize}
\end{theorem}
Finally, we apply this result and the proof of \cite{deloro2023simple} to prove a complete classification of stable Lie algebras of finite dimension.\\
In the sixth section, we prove a linearization theorem for Lie rings in finite-dimensional theories. 
\begin{theorem}\label{Lie}
    Let $\mathfrak{g}$ be a Lie ring acting on $V$ a definable absolutely minimal $\mathfrak{g}$-module. Suppose that $\mathfrak{g}$ has an abelian ideal $\mathfrak{a}$ such that the action of $\mathfrak{a}$ is not almost trivial. Then, there exists $\mathfrak{g}_1$ a definable ideal in $\mathfrak{g}$ of finite index in $\mathfrak{a}$ and $V_1$ a finite $\mathfrak{g}$-module in $V$ such that $\mathfrak{a}/\mathfrak{g}_1\leq Z(\mathfrak{g}/\mathfrak{g}_1)$ and there exists a definable field $K$ such that $\mathfrak{a}/\mathfrak{g}_1$ embeds in $K\cdot Id$, $V/V_1$ is a $K$-vector field of finite dimension and $\mathfrak{g}/\mathfrak{g}_1$ acts $K$-linearly. 
\end{theorem}
Moreover, we also characterize the action of nilpotent Lie rings.\\
In the seventh section, we prove the existence of definable envelopes for soluble and nilpotent Lie sub rings in $\widetilde{\mathfrak{M}}_c$-Lie rings. This class of Lie rings, that has the same role of $\widetilde{\mathfrak{M}}_c$-groups, contains, for example, all Lie rings definable in a finite-dimensional theory or in a simple theory. This result extend the existence of definable envelopes for stable Lie rings proved in \cite{zamour2025quelques}.
\begin{theorem}
  Let $\mathfrak{g}$ be an hereditarily $\widetilde{\mathfrak{M}}_c$-Lie ring and $\mathfrak{h}$ a Lie subring. Then,
  \begin{itemize}
      \item if $\mathfrak{h}$ is soluble, it is contained in a definable soluble Lie subring $\mathfrak{h}_1$ such that $N_{\mathfrak{g}}(\mathfrak{h})\leq N_{\mathfrak{g}}(\mathfrak{h}_1)$;
      \item if $\mathfrak{h}$ is nilpotent, it is contained in a definable nilpotent Lie subring $\mathfrak{h}_1$ such that $N_{\mathfrak{g}}(\mathfrak{h})\leq N_{\mathfrak{g}}(\mathfrak{h}_1)$.
  \end{itemize}
\end{theorem}
A consequence of this result is the definability of the Radical and of the Fitting ideal, that are respectively soluble and nilpotent, for a finite-dimensional Lie ring.
\section{Preliminaries}
In this section, we introduce the basic notion necessary for this article.
\subsection{Finite-dimensional groups}
We start defining some chain conditions on definable subgroups.
\begin{defn}
\begin{itemize}
    \item A definable group $G$ is said to have the \emph{DCC (descending chain condition on definable subgroups)} if it does not admit a strictly descending infinite chain of definable subgroups.
    \item A definable group $G$ is said to have the \emph{ucc (uniform descending chain condition)} if it does not admit an infinite strictly descending chain of uniformly definable subgroups.
    \item  A definable group $G$ is said to have the \emph{$\omega$-DCC (infinite descending chain condition)} if does not admit an infinite strictly descending chain of definable subgroups $\{G_i\}_{i<\omega}$ such that $|G_i: G_{i+1}|\geq \omega$.
\end{itemize}
\end{defn}
Groups definable in a finite-dimensional theory respect the $\omega$-DCC (Corollary 2.4 of \cite{wagner2020dimensional}) and ucc. The latter is a consequence of the following Lemma.
\begin{lemma}\label{boundedind}
    Let $G$ be a definable group of finite dimension and $\{H_i\}_{i\in I}$ a family of uniformly definable subgroups. Then, there exists $n,d<\omega$ such that there is no $J=\{j_1,...,j_n\}\subseteq I$ of cardinality $n$ such that $|\bigcap_{i=1}^kH_{j_i}:\bigcap_{i=1}^{k+1} H_{j_i}|$ has index greater than $d$ for any $k\leq n-1$.
\end{lemma}
\begin{proof}
Let $n=\dim(G)$. We may work in a sufficiently saturated structure $\mathfrak{M}$ in which $G$ is definable and assume that $\phi(x,y)$ is the formula defining the family. Assume, for a contradiction, that the conclusion is false. Then, for $N=n+2$ and for any $k\in \omega$, there exist $g_1,...,g_N$ such that $|\bigcap_{j\leq k-1} \phi(\mathfrak{M},g_j)/\bigcap_{j\leq k} \phi(\mathfrak{M},g_j)|\geq k$. Therefore, the partial type given by the formulas 
$$\exists a^1_1,...,a^1_k,...,a^N_1,...,a_k^N: a^i_j\in \bigcap_{j=1}^i \phi(\mathfrak{M},x_j)\wedge a^i_j{a^i_k}^{-1}\not\in \phi(\mathfrak{M},x_j)\wedge \forall_{i\leq N}\ \phi(\mathfrak{M},x_i)\leq G(\mathfrak{M})$$
for all $k<\omega$, is finitely satisfable. By compactness and saturation, there exist subgroups $\{H_i\}_{i\leq N}$ such that $|\bigcap_{i<j}H_i/\bigcap_{i\leq j}H_i|$ is infinite for every $i$. This would imply that the dimension of $G$ is strictly greater than $n$, a contradiction.
\end{proof}

If we apply these results to the centralisers of elements in $G$, we obtain that a finite-dimensional definable group has the hereditarily $\widetilde{\mathfrak{M}}_c$-property.
\begin{defn}
    A group $G$ is \emph{hereditarily-$\widetilde{\mathfrak{M}}_c$} if for any definable subgroups $H,N$, such that $N$ is normalised by $H$, there exist natural numbers $n_{HN}$ and $d_{HN}$ such that any sequence of centralisers 
    $$C_H(a_0/N)\geq C_H(a_0,a_1/N)\geq...\geq C_H(a_0,a_1,...,a_n/N)\geq...$$
    with each centraliser of index at least $d_{HN}$ in the previous one, has length at most $n_{HN}$.
\end{defn}
The notions of almost containment and commensurability are fundamental in the study of groups definable in finite-dimensional theories.
\begin{defn}
\begin{itemize}
    \item  Let $G$ be a group and $A, B\leq G$ subgroups. $A$ \emph{almost contains} $B$, denoted $A\apprge B$, if $|B:B\cap A|$ is finite.
    \item  If $A\apprle B$ and $B\apprle A$, $A$ and $B$ are \emph{commensurable} and we denote this with $A\sim B$. $\sim$ is clearly an equivalence relation on the class of definable subgroups of $G$.
    \item Given a family of subgroups $\{G_i\}_{i\in I}$ in $G$, we say that the family $\{G_i\}_{i\in I}$ is \emph{uniformly commensurable} if there exists $n<\omega$ such that $|G_i/G_i\cap G_j|\leq n$ for every $i,j\in I$.
\end{itemize}
  \end{defn}
 An important theorem about uniformly commensurable groups is Schlichting's theorem \cite[Theorem 4.2.4]{wagner2000simple}.
\begin{theorem}
    Let $G$ be a group and $\{G_i\}_{i\in I}$ a uniformly commensurable family of subgroups in $G$. Then, there exists a subgroup $H$ in $G$ commensurable with every $G_i$ and invariant under all the automorphisms that fix set-wise the family $\{G_i\}_{i\in I}$. \\
    Moreover, if $G$ is definable and $\{G_i\}_{i\in I}$ is an uniformly commensurable family of definable subgroup in $G$, also $H$ is definable.
\end{theorem} 
  The notion of almost direct sum is central in the study of connected Lie rings.
 \begin{defn}
   Let $G$ be a group and $A,B$ two subgroups such that $A\leq N_G(B)$. The sum $A+B$ is \emph{almost direct} if $A\cap B$ is finite. We denote $A\widetilde{\oplus}B$ the almost direct sum of $A$ and $B$.
  \end{defn}
  We define the notion of isogenicity.
  \begin{defn}
    Given two groups $G,H$, they are \emph{isogenous} if there exists an isomorphism
    $$\phi:A/B\mapsto A_1/B_1$$
    with $A,A_1$ subgroups of finite index in $G,H$ respectively and $B,B_1$ finite subgroups in $G,H$ respectively. An \emph{isogeny} $\phi$ from $G$ to $H$ is an homomorphism with finite kernel and image of finite index. 
\end{defn}
An easy result is the following, which we will use without reference.
\begin{lemma}
    Let $G,H$ be two definable isogenous groups. Then, $G$ is virtually connected iff $H$ is virtually connected.
\end{lemma}
\begin{proof}
    Assume, for a contradiction, that $G$ is non virtually connected, while $H$ is. We can assume $H$ is connected. Let $B_1$ the finite subgroup in $H$, $A$ and $B$ definable subgroups in $G$ respectively of finite index and finite such that
    $$\phi:A/B\mapsto H/B_1$$
    is an isomorphism and let $N=|B|$. Since $G$ is not virtually connected, there exists a definable subgroup $G_1$ of finite index $M>N$. Then $|A:G_1\cap A|=M$ and so $G_1\cap A/B$ is a proper definable subgroup of finite index. Since $A/B$ is isomorphic to $H/B_1$, which is connected, also $A/B$ is connected, a contradiction.
\end{proof}
Finally, we proves that a field definable in a finite-dimensional theory is perfect.
\begin{lemma}\label{perfection}
   Let $K$ be a field definable in a finite-dimensional theory. Then, $K$ is perfect.
\end{lemma}
\begin{proof}
   In characteristic $0$, we have nothing to prove. Assume $\operatorname{char}(K)=p$. Then, $K^p$ is a definable subfield of $K$. By \cite[Proposition 3.3]{wagner2020dimensional}, $\dim(K)=\dim(K^p)\operatorname{lin.dim}_{K^p}(K)$. The homomorphism $^p:K\mapsto K$ that sends $k$ in $k^p$ has finite kernel and the image is $K^p$. By fibration, $\dim(K)=\dim(K^p)$. This implies that $\operatorname{lin.dim}_{K^p}(K)=1$ \hbox{i.e.} $K=K^p$.
\end{proof}
\subsection{Lie rings}
We define Lie rings and introduce their basic properties.
\begin{defn}
    A \emph{Lie ring} $(\mathfrak{g},+,[,])$ is an abelian group $(\mathfrak{g},+)$ with a map 
    $$[\_,\_]:\mathfrak{g}\times \mathfrak{g}\mapsto \mathfrak{g}$$
    such that $[\_,\_]$ is bilinear, antisymmetric and $[g_1,[g_2,g_3]]+[g_3,[g_1,g_2]]+[g_2,[g_3,g_1]]=0$ for all $g_1,g_2,g_3\in \mathfrak{g}$. $[\_,\_]$ is called the \emph{(Lie) bracket} on $\mathfrak{g}$.\\
    A subgroup $\mathfrak{c}$ is a \emph{Lie subring}, denoted $\mathfrak{c}\leq \mathfrak{g}$, if $\mathfrak{c}$ is closed for $[\_,\_]$. A Lie subring $\mathfrak{c}$ is an \emph{ideal} in $\mathfrak{g}$ if, for all $g\in \mathfrak{g}$, $[g,\mathfrak{c}]$ is contained in $\mathfrak{c}$.\\
    A Lie ring $\mathfrak{g}$ is \emph{simple} if $\mathfrak{g}$ has no non trivial ideal (\hbox{i.e.} different from $\{0\}$ and $\mathfrak{g}$). A definable Lie ring $\mathfrak{g}$ is \emph{definably simple} if it has no definable proper ideals.
\end{defn}
It is easy to verify that the set $g[p]=\{g\in \mathfrak{g}:\ g^p=0\}$ for $p$ a prime is an ideal. Therefore, a simple Lie ring has elements all of the same order (equal to a prime or infinite). This is called the characteristic of $\mathfrak{g}$, denoted by $\operatorname{char}(\mathfrak{g})$.
\begin{defn}
    Let $\mathfrak{g}$ be a Lie ring. If $\mathfrak{g}=\mathfrak{g}[p]$, we say that $\mathfrak{g}$ has \emph{characteristic} $p$. If $(\mathfrak{g},+)$ is torsion-free, $\mathfrak{g}$ is said to have \emph{characteristic} $0$.
\end{defn}
The following lemma is fundamental in the study of Lie rings of dimension $1$ (and more in general of definably minimal Lie rings) and the proof can be found in \cite{rosengarten1991aleph}.
\begin{lemma}\label{FinSubAlg}
    Let $\mathfrak{g}$ be a Lie algebra over $\mathbb{F}_p$ of finite dimension with $p>3$ a prime. Assume that the $\mathbb{F}_p$-dimensions of centralisers of an element in $\mathfrak{g}$ are bounded by $d$. Then, $\operatorname{dim}_{\mathbb{F}_p}(\mathfrak{g})<p^{d+1}d+d$ 
\end{lemma}
\section{Connected simple Lie rings}
In this section, we prove that, given $\mathfrak{g}$ a simple definable non abelian connected Lie ring, then $(\mathfrak{g},+)$ has the DCC.\\
We need some preliminary lemmas.
\begin{lemma}\label{virtconn}
    Let $G$ be a connected definable abelian group, $\{H_i\}_{i\leq n}$ definable subgroups such that $G=\widetilde{\oplus}_{i=1}^n H_i$. Then, $H_i$ is virtually connected for every $i\leq n$.
\end{lemma}
\begin{proof}
    By induction on $n$, it is sufficient to verify the statement for $n=2$. Therefore, let $H_1,H_2$ be definable subgroup in $G$ such that $G=H_1\widetilde{\oplus} H_2$.\\
    Suppose, for a contradiction, that $H_1$ is not virtually connected and let $m=|H_1\cap H_2|$. By non virtual connectivity of $H_1$, there exists a definable subgroup $H$ in $H_1$ of index $M>m$. We verify that $H+H_2$ is strictly contained in $G$. This contradicts the connectivity of $G$. Let $\{a_i\}_{i=1}^M$ be a system of representatives of $H$ in $H_1$ then $a_i-a_j\in H_1-H$. Since $G=H+H_2$ then $a_i=h_i+k_i$ with $h_i\in H_2$ and $k_i\in H$. Let $i\not=j\leq M$, then $a_i-a_j=(h_i-h_j)+k_i-k_j$. Therefore, $a_i-a_j-(k_i-k_j)=h_i-h_j\in H_2\cap H_1$. So $a_i-a_j\in H+H_1\cap H_2$ for $i,j\leq M$. This is a contradiction since it implies that $M=|H_1/H|=|H_1\cap H_2/H\cap H_2|<m<M$.
\end{proof}
The non connected version of the following Lemma (under the stronger assumption that both the Lie subrings are definable) will be proved in the fourth chapter.
\begin{lemma}\label{Definabilityder}
    Let $\mathfrak{g}$ be a definable Lie ring in a finite-dimensional theory and $\mathfrak{h},\mathfrak{k}$ two Lie subrings. Assume that they normalise each other and that $\mathfrak{h}$ is connected and definable. Then, $[\mathfrak{k},\mathfrak{h}]$ (the subgroup generated by $[h,k]$ for $h\in\mathfrak{h}$ and $k\in \mathfrak{k}$) is a definable, connected Lie subring normalised by $\mathfrak{h}$ and $\mathfrak{k}$.
\end{lemma}
\begin{proof}
    We verify that $[\mathfrak{h},\mathfrak{k}]$ is closed by $[\_,\_]$. By bilinearity, it is sufficient to verify that 
    $[[h_1,k_1],[h_2,k_2]]\in [\mathfrak{h},\mathfrak{k}]$. By the Jacobi identity, 
    $$[[h_1,k_1],[h_2,k_2]]=-[k_2,[[h_1,k_1],h_2]]-[h_2,[k_2,[h_1,k_1]]].$$
    $[k_2,[h_1,k_1]]\in \mathfrak{k}$ and $[h_2,[h_1,k_1]]\in \mathfrak{h}$ since $\mathfrak{h}$ and $\mathfrak{k}$ are Lie subrings that normalise each other. This proves that $[\mathfrak{g},\mathfrak{k}]$ is a Lie subring.\\
    We take a sum $L=\sum_{i=1}^n [\mathfrak{h},k_i]$ of maximal dimension between the finite sums of $[\mathfrak{h},k_i]$ with $k_i\in \mathfrak{k}$. This is clearly a definable subgroup, and it is connected since it is the sum of finitely many connected subgroups.
    It is obviously contained in $[\mathfrak{h},\mathfrak{k}]$. We verify the inverse inclusion. Being a subgroup, by maximality of the dimension of $L$, $L+[\mathfrak{h},g]$ has dimension equal to $\operatorname{dim}(L)$. Therefore, $\operatorname{dim}(L)\cap \operatorname{dim}([\mathfrak{h},g])$
    is equal to $\operatorname{dim}([\mathfrak{h},k])$. By connectivity, $[\mathfrak{h},k]\leq L$ and so $L$ coincides with $[\mathfrak{h},\mathfrak{k}]$. Since $L$ is connected and definable, the same holds for $[\mathfrak{h},\mathfrak{k}]$. The last point follows easily.
\end{proof}
A corollary of Lemma \ref{Definabilityder} is that in a connected Lie ring of finite dimension, both the central descending series and the abelian descending series are definable and connected. We derive another easy consequence.
\begin{lemma}\label{DefSimImpSim}
    Let $\mathfrak{g}$ be a connected non abelian definably simple Lie ring. Then, $\mathfrak{g}$ is simple.
\end{lemma}
\begin{proof}
    Suppose, by way of contradiction, that $\mathfrak{g}$ is not simple and let $J$ be a non trivial ideal in $\mathfrak{g}$. By Lemma \ref{Definabilityder}, $[J,\mathfrak{g}]$ is a definable connected ideal. By definable simplicity, this ideal must be $0$, since it is contained in $J$, which is a proper ideal in $\mathfrak{g}$. Therefore, $J\subseteq Z(\mathfrak{g})$ that consequently is non trivial. Since $Z(\mathfrak{g})$ is a definable non trivial ideal, it coincides with $\mathfrak{g}$ \hbox{i.e.} $\mathfrak{g}$ is abelian, a contradiction.
\end{proof}
The following results is fundamental in the analysis of connected Lie rings of finite dimension.
\begin{lemma}\label{DecGen}
    Let $G$ be a group and $A$ an abelian connected group, both definable in a finite-dimensional theory. Then, in the following two cases, $A$ has the DCC:
    \begin{itemize}
        \item $G$ is connected, $G$ embeds in $\operatorname{DefEnd}(A)$ and $A$ is $G$-minimal;
        \item $G$ embeds in $\operatorname{DefAut}(A)$ and $A$ is $G$-minimal.
    \end{itemize}
\end{lemma}
\begin{proof}
    We take $A_1$ an infinite definable subgroup of minimal dimension in $A$. Observe that $A_1$ need not be virtually connected.\\
    Take $B=\sum_{i=1}^n (g_{j_1}\cdot\cdot \cdot g_{j_i}A_1:=A_i)$ of maximal dimension between the subgroups of this form.\\
    Clearly, $B$ is definable. Every $A_i$ has dimension equal to the dimension of $A$ or to $0$, by the minimality of the dimension. The finite $A_i$ can be excluded from this sum without changing the dimension, and so we may assume that each $A_i$ has dimension equal to $\dim(A_1)$. Moreover, this sum is almost direct. Assume not, then $A_i\cap \sum_{j\not=i} A_j$ is an infinite definable subgroup of $A_i$. By minimality, it is of finite index, and $A_i$ can be excluded from the sum without changing the dimension.\\
    By maximality of the dimension, $g\cdot B$ is almost contained in $B$ for every $g\in G$. Indeed, $gB+B$ is a sum of the previous form and so it must have the same dimension as $B$.\\
    If $G$ embeds in $\operatorname{Aut}(A)$, then $\{g\cdot B\}_{g\in G}$ is a uniformly commensurable family of definable subgroups. Indeed, it is a uniformly definable family of commensurable subgroups and the uniform commensurability follows by compactness. Therefore, by Schilisting's theorem, there exists a definable $G$-invariant subgroup $B_1\sim B$. Since $A$ is $G$-minimal and $B_1$ is infinite, $\dim(B)=\dim(B_1)=\dim(A)$. Then, $A$ is connected and equal to the almost direct sum of the subgroups $A_i$ for $i\leq n$. By Lemma \ref{virtconn}, each $A_i$ is virtually connected. Up to take the connected component, we may assume that any $A_i$ is connected. Then, the DCC follows by the following claim.
    \begin{claim}
    Let $A$ be a finite-dimensional group. Assume that $A$ is equal to the almost direct sum $\widetilde{\oplus}_{i=1}^n A_i$ for $A_i$ definably minimal connected subgroups. Then, $A$ has the DDC.
\end{claim}
\begin{proof}
    It is sufficient to prove that every infinite definable subgroup is virtually connected. The conclusion follows by $\omega$-DCC. Let $B$ be a definable subgroup of $A$ and take $D=\sum_{i\in J} A_i$ such that $D$ is a finite sum of $A_i$ of maximal dimension between the sums of this form with finite intersection with $B$. We prove that $B+D$ contains every $A_i$ and, therefore, it is equal to $A$. This implies that $B$ is virtually connected by Lemma \ref{virtconn}.\\
    Let $A_i$ such that $A_i\cap (B+D)$ is not equal to $A_i$. By minimality of $A_i$, this implies that the intersection is finite. It follows that $A_i+D$ has finite intersection with $B$ and it has dimension strictly greater than $D$, a contradiction. This implies, by connectivity, that $A=B+D$. 
\end{proof}
We verify the Lemma in the case that $G$ is connected and embeds in $\operatorname{DefEnd}(A)$. By the claim, it is sufficient to prove that there exists an almost direct sum of the form $\widetilde{\oplus}_{i=1}^n A_i$ with $A_i$ definably minimal subgroups.\\
    Let 
    $$X=\{(g,b)\in G\times B:\ g\cdot b\in B\}.$$
    This is a definable subset of $G\times B$. Take the first projection: its image is equal to $B$ and every fiber is $\{b\in B:g\cdot b\in B\}$. Since $g\cdot B\apprle B$, any fiber has dimension equal to $\operatorname{dim}(B)$. By fibration, $\operatorname{dim}(X)=\operatorname{dim}(G)+\operatorname{dim}(B)$.\\
    The image of the projection to the second component is $B$ and the fibers are $C_{G}(b/B)=\{g\in \mathfrak{g}:g\cdot b\in B\}$. Since $G$ is connected, the subgroup $C_{G}(b/B)$ is of finite index iff it coincides with $G$. Let $B_1$ be the definable subgroup $C_{B}(G/B)=\{b\in B:G\cdot b\leq B\}$. Define the definable subgroup $Y$ as
    $$Y=\{(g,h):\ g\in G,h\in B_1\}.$$
   By union, $\operatorname{dim}(X)=\max\{\operatorname{dim}(Y),\operatorname{dim}(Y^c)\}$ and by fibration $\operatorname{dim}(Y)=\operatorname{dim}(G)+\operatorname{dim}(B_1)$. The fiber of an element $h\in B_1^c$ has dimension strictly less than $\operatorname{dim}(G)$. Consequently, $\operatorname{dim}(Y^c)\leq \operatorname{dim}(B)+\operatorname{dim}(G)-1$ by lower fibration. This implies that $B_1$ is a definable subgroup of $B$ of finite index.\\
    Moreover, $B_1$ is $G$-invariant: given $b\in B_1$, $G\cdot b\leq B$ and so $G\cdot b$ is almost contained in $B_1$. Since $G\cdot b$ is connected, $G\cdot b\leq B_1$.\\
    By $G$-minimality of $A$, $B$ is equal to $A$. Then, the conclusion follows by the claim.
\end{proof}
From Lemma \ref{DecGen}, we can derive the following corollary.
\begin{corollary}\label{DCC}
    Let $\mathfrak{g}$ be a definable connected Lie ring of finite dimension and $\mathfrak{h}$ a definably minimal ideal. Then, either $\mathfrak{h}$ has the DCC and $Z(\mathfrak{h})=Z(\mathfrak{g})\cap \mathfrak{h}$ is finite or $\mathfrak{h}\apprle Z(\mathfrak{g})$.
\end{corollary}
\begin{proof}
We may assume $\mathfrak{h}$ connected. Indeed, $[\mathfrak{g},\mathfrak{h}]$ is a definable connected ideal by Lemma \ref{Definabilityder}. Then, either it is $\{0\}$ or $[\mathfrak{g},\mathfrak{h}]$ is a connected infinite ideal contained in $\mathfrak{h}$. In the first case, $\mathfrak{h}\cap Z(\mathfrak{g})$ is infinite and, by minimality, $\mathfrak{h}\apprle Z(\mathfrak{g})$. Consequently, we may assume that $[\mathfrak{g},\mathfrak{h}]$ is of finite index in $\mathfrak{h}$, by minimality. Being $[\mathfrak{g},\mathfrak{h}]$ connected, $\mathfrak{h}$ is virtually connected. We may assume $\mathfrak{h}$ connected.
The map
    $$\operatorname{ad}:\mathfrak{g}\to \operatorname{DefEnd}(\mathfrak{h},+)$$
    that sends $g$ in $\operatorname{ad}(g)=[g,\_]$ is an homorphism of groups. Moreover, $\mathfrak{h}$ is $\mathfrak{g}$-minimal, being a minimal ideal. By Lemma \ref{DecGen}, $\mathfrak{h}$ has the DCC and the conclusion follows. 
\end{proof}
A consequence of Corollary \ref{DCC} is the possibility of constructing a series of ideal $\{\mathfrak{h}_i\}_{i\leq n}$ of a finite-dimensional connected Lie ring $\mathfrak{h}$ such that, for any $i\leq n$, either $\mathfrak{h}_i/\mathfrak{h}_{i+1}\leq Z(\mathfrak{g}/\mathfrak{h}_{i+1})$ or $\mathfrak{h}_i/\mathfrak{h}_{i+1}$ is $\mathfrak{g}$-minimal, connected and has the DCC.
\begin{theorem}\label{DCCseries}
    Let $\mathfrak{g}$ be a connected Lie ring definable in a finite-dimensional theory. Then, there exists a series
    $$\mathfrak{g}=\mathfrak{h}_0\geq\mathfrak{h}_1\geq...\geq \mathfrak{h}_n=\{0\}$$
    of definable ideals such that either $\mathfrak{h_i}/\mathfrak{h_{i+1}}\leq Z(\mathfrak{g}/\mathfrak{h}_i)$ or $\mathfrak{h_i}/\mathfrak{h_{i+1}}$ is connected, $\mathfrak{g}$-minimal and it has the DCC.
\end{theorem}
\begin{proof}
    We prove the Theorem by induction on $\dim(\mathfrak{g})$. If the center of $\mathfrak{g}$ is infinite, $\mathfrak{g}/Z(\mathfrak{g})$ has dimension strictly smaller than $\dim(\mathfrak{g})$, and the Theorem follows by induction. So we may assume $Z(\mathfrak{g})=\{0\}$.\\
    Let $\mathfrak{h}$ be a definably minimal ideal of $\mathfrak{g}$. Then, $[\mathfrak{g},\mathfrak{h}]$ is a definable connected ideal contained in $\mathfrak{h}$ by Lemma \ref{Definabilityder}. If it is $\{0\}$, then $\mathfrak{h}\leq Z(\mathfrak{g})=\{0\}$, a contradiction. By Corollary \ref{DCC}, $\mathfrak{h}$ has the DCC. Taking the connected component of $\mathfrak{h}$, the conclusion follows.
\end{proof}
A second corollary of Lemma \ref{DecGen} is that, if the additive group of a field of finite dimension is connected, then it has the DCC.
\begin{corollary}\label{DCCfield}
    Let $K$ be a finite-dimensional definable field with $K^+$ connected. Then $K^+$ has the DCC. In particular, for a finite-dimensional field $K$ of characteristic $0$, $K^+$ has the DCC.
\end{corollary}
\begin{proof}
    Denote $A=(K,+)$ and $G=(K,\cdot)$. Then $G$ acts on $A$ by automorphisms, and $A$ is $G$-minimal since a field has no proper ideals. From Lemma \ref{DecGen}, $(K,+)$ has DCC.
\end{proof}
We extend the previous result to direct products of $K^+$. We recall the Goursat's Lemma.
\begin{lemma}\label{Goursat}
   Let $G=H_1\times H_2$ for $H,K$ two groups. Then, there is a bijective correspondence between subgroups $K$ of $G$ and quintuples $(K_1,\overline{K}_1,K_2,\overline{K}_2,\theta)$ where $K_1\unlhd \overline{K}_1\leq H_1$ and $K_2\unlhd \overline{K}_2\leq H_2$ and $\theta$ an isomorphism of groups from $\overline{K}_1/K_1$ and $\overline{K}_2/K_2$. In particular, $\overline{K}_1=\pi_1(K)$, $K_1=\{h\in H_1:\ (h,1)\in K\}$ and similarly for $\overline{K}_2$ and $K_2$. Finally $\theta$ is such that $\theta([h_1])=[h_2]$ iff $(h_1,h_2)\in K$.
\end{lemma}
\begin{lemma}\label{DCCDirPro}
   Let $H_1,H_2$ be two definable finite-dimensional groups with the DCC. Then, $G=H_1\times H_2$ has the DCC.
\end{lemma}
\begin{proof}
   It is sufficient to prove that any definable subgroup has a connected component. Let $K$ be a definable subgroup of $G$. Then, let $(K_1,\overline{K}_1,K_2,\overline{K}_2,\theta)$ as in Lemma \ref{Goursat}. Clearly, all the subgroups are definable, as the isomorphism $\theta$. Moreover, by fibration, $\dim(K)=\dim(K_1)+\dim(\overline{K}_2)=\dim(K_2)+\dim(\overline{K}_1)$. Given an infinite strictly descending family $\{K_i\}_{i<\omega}$ of definable subgroups of finite index in $K$, we obtain a descending family of subgroups $(K^i_1,\overline{K}^i_1,K^i_2,\overline{K}^i_2,\theta^i)$ all of finite index (by dimensionality). Since $K_1,\overline{K}_1,K_2,\overline{K}_2$ are virtually connected, all these stabilize after finitely many steps, say $N$. Then, $K_1^N=K_1^{N+1}$ and similarly for all the other subgroups. Consequently, $\theta^{N+1},\theta^N$ must be isomorphisms between $\overline{K}_1^N/K_1^N$ and $\overline{K}^N_2/K^N_2$ such that, if $\theta^{N+1}([a])=[b]$, then $\theta^{N}([a])=[b]$. Since they are isomorphisms, we conclude that they are the same. Since the correspondence is bijective, we conclude that $K_N=K_{N+1}$, a contradiction. Therefore, $K$ has a connected component, and the proof follows.
\end{proof}
In particular, we can apply the Lemma \ref{DCCDirPro} to the direct product of connected additive group of a field.
\begin{corollary}\label{DCCK^n}
    Let $K$ be a definable finite-dimensional field with $K^+$ connected. Then, for any $n\in \omega$, $K^n$ has the DCC.
\end{corollary}
Finally, we prove some useful lemmas for Lie rings with the DCC. These are the same results obtained by Deloro and Ntsiri in \cite{deloro2023simple}.
\begin{lemma}
    Let $\mathfrak{g}$ be a simple definable Lie ring with DCC and $\mathfrak{h}$ a definable connected proper subring. Then $I=C^{\circ}_{\mathfrak{h}}(\mathfrak{g}/\mathfrak{h})$ is nilpotent.
\end{lemma}
\begin{proof}
    $I$ exists by DCC and it is an ideal of $\mathfrak{h}$, and so also every $I^{[n]}$ is an ideal of $\mathfrak{h}$. We claim that, for $n\geq 0$, $[I^{[n+1]},\mathfrak{g}]\leq I^{[n]}$. For $n=0$ we have that $[I^{[1]},\mathfrak{g}]\leq [[I,I],\mathfrak{g}]\leq [\mathfrak{h},I]\leq I$ and same for $n+1$. Therefore, the chain of $I^{[n]}$ stabilises in an ideal of $\mathfrak{g}$ that, by simplicity, can only be $\mathfrak{g}$ or $0$. It cannot be $\mathfrak{g}$ since $I$ is contained in $\mathfrak{h}<\mathfrak{g}$ and so it is $0$.
\end{proof}
We obtain an important self-normalising result.
\begin{lemma}
    Let $\mathfrak{g}$ be a definable, connected Lie ring and $\mathfrak{h}\subseteq \mathfrak{g}$ be a definable connected subring of codimension $1$. Then either $\mathfrak{h}$ is an ideal of $\mathfrak{g}$ or $N_{\mathfrak{g}}(\mathfrak{h})=\mathfrak{h}$.
\end{lemma}
\begin{proof}
    Let $I=C_{\mathfrak{g}}(\mathfrak{g}/\mathfrak{h})$. If $I=\mathfrak{h}$, the proof is completed. Therefore, suppose the action is not trivial, then it is free. This implies that, given $g\in \mathfrak{g},h\in \mathfrak{h}$, $gh\in \mathfrak{h}$ iff $g\mathfrak{h}$ or $h=0$. Since $\mathfrak{h}\not=0$ then it implies that $N_{\mathfrak{g}}(\mathfrak{h})=\mathfrak{h}$.
\end{proof}
We finish by proving the triviality of definable derivations on finite-dimensional fields.
\begin{defn}
    Let $K$ be a field. An additive homomorphism $\sigma$ of $K$ is a \emph{derivation} of $K$ if, for any $k,k'\in K$,
    $$\sigma(kk')=k\sigma(k')+k'\sigma(k).$$
\end{defn}
\begin{lemma}\label{DerTri}
    A finite-dimensional field $K$ has no non trivial definable derivations.
\end{lemma}
\begin{proof}
Let $\delta$ a definable derivation of $K$.\\
    In characteristic $p$, given $x\in K$, there exists $y\in K$ such that $y^p=x$ since $K$ is perfect by Lemma \ref{perfection}. Therefore, $\delta(x)=\delta(y^p)=py^{p-1}\delta(y)=0$. This implies that a dimensional field of prime characteristic has no non-trivial derivation.\\
    In characteristic $0$, we verify that the field of constant $C$ of a derivation $\delta$ is definable and relatively algebraically closed. This would imply that it is an infinite algebraically closed subfield of a finite-dimensional field. This is a contradiction to \cite[Proposition 3.3]{wagner2020dimensional}.\\
    It is sufficient to verify that, for any $y$ algebraic element over $C$, then $y\in C$.\\
    Let $f(x)$ be the minimal polynomial of $y$ over $C$. Then, $f(y)=\sum_{i=0}^n a_iy^i=0$ with $a_i\in C$. This implies that $0=\delta(0)=\delta(f(y))=(\sum_{i=1}^n a_i\cdot i\cdot y^{i-1})\delta(y)$. Called $g(x)=\sum_{i=1}^n a_i\cdot i\cdot x^{i-1}$, it is a polynomial of degree strictly less than $\mathrm{deg}(f(x))$ and so $f(x)$ cannot be $0$. This implies that $\delta(x)=0$ and so $C$ is relatively algebraically closed.
\end{proof}
\section{Characterization of Lie rings of small dimension}
The Corollary \ref{DCC} is used to classify connected definable Lie rings of finite dimension. We recall the Theorem we will prove.
\begin{theorem}
  Let $\mathfrak{g}$ be a connected definable Lie ring of finite dimension and of characteristic $>3$. Then,
  \begin{itemize}
      \item if $\dim(\mathfrak{g})=1$, $\mathfrak{g}$ is abelian;
      \item if $\dim(\mathfrak{g})=2$, $\mathfrak{g}$ is soluble. If it is not nilpotent, $\mathfrak{g}/Z(\mathfrak{g})\simeq \mathfrak{g}\mathfrak{a}_1(K)$ for a perfect definable field $K$;
      \item if $\dim(\mathfrak{g})=3$ and $\mathfrak{g}$ is not soluble, $\mathfrak{g}$ is bad;
      \item if $\dim(\mathfrak{g})=4$ and $\mathfrak{g}$ is simple, $\mathfrak{g}$ is bad.
   \end{itemize}
\end{theorem}
In this case, the definitions of bad and Borel are the corresponding notions, in finite-dimensional theories, of the bad Lie rings and Borel of finite Morley rank (see \cite{deloro2023simple}).
\begin{defn}
    Let $\mathfrak{g}$ be a simple connected Lie ring of finite dimension. A definable connected proper soluble Lie subring $\mathfrak{h}$ of $\mathfrak{g}$ is a \emph{Borel} of $\mathfrak{g}$ if it is not contained in any proper connected definable soluble Lie subring of $\mathfrak{g}$. $\mathfrak{g}$ is \emph{bad} if it has a nilpotent Borel.
\end{defn}
In the NIP context, the two notions will be modified, since we do not always have connected components.\\
The non-existence of bad Lie rings in the finite Morley rank context follows from Macintyre's theorem \cite[Theorem 3.1]{poizat2001stable}, which clearly does not hold in the finite-dimensional context. In characteristic $0$, these surely exist, for example, the Lie algebra $\mathfrak{so}_3(R)$. In characteristic $p$, it is still unclear whether they exist. Moreover, they do not necessarily lead to a confutation of the finite-dimensional "log-CZ"-conjecture. A deep open question is whether bad Lie rings of dimension $3$ and $4$ and characteristic $\geq 5$ are Lie algebras over a definable perfect field. The only partial answer is that the additive group has to be a vector space over a perfect field of dimension $1$. 
\subsection{Dimension 1}
We start from the case $n=1$. We prove that a definably minimal connected Lie ring of finite dimension is abelian. The proof is very similar to the proof of Rosengarten \cite{rosengarten1991aleph}.
\begin{defn}
    A Lie ring $\mathfrak{g}$ is \emph{definably minimal} if it has no infinite definable Lie subring of infinite index.
\end{defn}
\begin{theorem}\label{Dim1}
    Let $\mathfrak{g}$ be a definably minimal connected Lie ring with characteristic different from $2,3$. Then, $\mathfrak{g}$ is abelian.
\end{theorem}
\begin{proof}
    The center of $\mathfrak{g}$ is a definable ideal of $\mathfrak{g}$. By minimality, it can only be finite or equal to $\mathfrak{g}$. In the latter case, the conclusion is achieved. Assume the first case. Since $\mathfrak{g}$ is connected, $\mathfrak{g}/Z(\mathfrak{g})$ is a centerless connected definably minimal Lie ring of finite dimension (the proof is the same as \cite[Lemma 6.1]{borovik1994groups}). We verify that this structure cannot exist. Since it is connected and centerless, the centraliser of every element different from $0$ is finite. By compactness, there is a bound $d$ on the cardinality of the centralisers. This implies that the characteristic cannot be $0$ since, for any $g\in \mathfrak{g}$, $\langle g\rangle\leq C_{\mathfrak{g}}(g)$.\\
    By Lemma \ref{FinSubAlg}, there is a bound also on the cardinality of every finite subring. Therefore, to find a contradiction, it is sufficient to construct a strictly ascending sequence of finite Lie subrings of $\mathfrak{g}$.\\
    Take $g\in \mathfrak{g}-{0}$. By assumptions, $C_{\mathfrak{g}}(g)$ is a finite proper Lie subring of $\mathfrak{g}$. Since $\mathfrak{g}$ is connected, the subgroup of finite index $[g,\mathfrak{g}]$ is equal to $\mathfrak{g}$ and so $C_{\mathfrak{g}}^2(g)>C_{\mathfrak{g}}(g)$. Iterating the process, we obtain that $C^i_{\mathfrak{g}}(g)<C^{i+1}_{\mathfrak{g}}(g)$. Consequently, it is sufficient to verify that $C^{p^j}_{\mathfrak{g}}(g)$ is a Lie subring for every $j<\omega$.\\
   Let $h,k\in C^{p^{i}}_{\mathfrak{g}}(g)$. Then, 
    $$ad_g^{p^{i}}[h,k]=\sum_{j=0}^{i} \binom{p^{i}}{p^j} [ad_g^{p^j}(h),ad_g^{p^{i-j}}(k)]=0.$$
    This concludes the proof.
\end{proof}
\subsection{Dimension 2}
The proof in the case $\operatorname{dim}(\mathfrak{g})=2$ uses a linearization theorem of the action of a Lie ring $\mathfrak{g}$ on a $\mathfrak{g}$-module, both of finite dimension. Since, for the moment, all the Lie rings are connected, it is sufficient to recall the Fact at the end of section 2.2 of \cite{deloro2023simple}. 
\begin{fact}\label{LinDimCon}
    Let $(\mathfrak{g},V)$ be a definable, faithful, irreducible $\mathfrak{g}$-module with $V$ connected. Assume that both $\mathfrak{g}$ and $C_{\operatorname{End}(V)}(\mathfrak{g})$ are infinite. Then, there exists a definable field $K$ such that $A$ is a finite-dimensional $K$-vector space and $\mathfrak{g}$ acts $K$-linearly on $V$.
\end{fact}
On the other hand, when we will analyze NIP finite-dimensional Lie rings, we will not assume connectivity. Therefore, it will be necessary to introduce the non-connected version of Fact \ref{LinDimCon} (this will be done in section 5).\\
We will refer to the following Lemma, whose proof can be found in \cite{deloro2023simple}, with "Connected linearization in dimension $1$". 
\begin{lemma}
    Let $\mathfrak{g}$ be a definable Lie ring acting faithfully on $V$, an irreducible connected $\mathfrak{g}$-module of dimension $1$. Then, $\operatorname{dim}(\mathfrak{g})$ is $1$, $\mathfrak{g}\simeq K^+$ and $V\simeq K^{+}$ for a definable field $K$ of dimension $1$.
\end{lemma}
We classify connected Lie rings of dimension $2$.
\begin{theorem}\label{Dim2}
    Let $\mathfrak{g}$ a definable connected Lie ring of dimension $2$. Then, $\mathfrak{g}$ is soluble of class $2$. Moreover, if $\mathfrak{g}$ is a soluble non-nilpotent Lie ring, then:
    \begin{itemize}
        \item There is a Cartan subring $\mathfrak{c}\leq \mathfrak{g}$ with $\mathfrak{c}\oplus \mathfrak{g}'=\mathfrak{g}$;
        \item $Z(\mathfrak{g})$ is finite and $\mathfrak{g}/Z(\mathfrak{g})\simeq \mathfrak{g}\mathfrak{a_1}(K)$ for a definable field $K$;
        \item There is $h\in \mathfrak{c}$ such that $ad(h)_{|\mathfrak{g}'}=Id_{|\mathfrak{g}'}$.
    \end{itemize}
\end{theorem}
\begin{proof}
We verify that $\mathfrak{g}$ is soluble.\\
Assume, for a contradiction, that it is not soluble. Then, the center must be finite and $\mathfrak{g}/Z(\mathfrak{g})$ must be definably simple. Assume not and let $\mathfrak{h}$ be a definable ideal of dimension $1$. Then, either $[\mathfrak{h},\mathfrak{g}]$ has finite index in $\mathfrak{h}$ or it is finite. Since $[\mathfrak{h},\mathfrak{g}]$ is connected by Lemma \ref{Definabilityder}, either $\mathfrak{h}$ is virtually connected or $[\mathfrak{h},\mathfrak{g}]=0$ \hbox{i.e.} $\mathfrak{h}\leq Z(\mathfrak{g})$. In the second case, $\mathfrak{g}$ is nilpotent since $\mathfrak{g}/Z(\mathfrak{g})$ has dimension $1$ and it is abelian by Lemma \ref{Dim1}. In the first case, $\mathfrak{g}$ has the DCC (since any definable proper infinite subgroup either contains $[\mathfrak{h},\mathfrak{g}]$ or it is isogenic to $\mathfrak{g}/[\mathfrak{g},\mathfrak{h}]$ that is connected). In this case, the solubility follows as in \cite{deloro2023simple}.\\
    We may assume that $\mathfrak{g}/Z(\mathfrak{g})$ is definably simple. By Theorem \ref{DCC}, $\mathfrak{g}$ is simple with $DCC$. Then the proof follows, again, as in \cite{deloro2023simple}.\\
For the second part, let $\mathfrak{h}$ be a definable ideal of dimension $1$ of $\mathfrak{g}$. Then, by Lemma \ref{Definabilityder}, $[\mathfrak{g},\mathfrak{h}]$ is a definable connected ideal in $\mathfrak{g}$ contained in $\mathfrak{h}$. Either it is $0$, and so $\mathfrak{h}$ is contained in $Z(\mathfrak{g})$, or $\mathfrak{h}$ is connected. In the first case, $\mathfrak{g}/\mathfrak{h}$ is abelian by Lemma \ref{Dim1} and so $\mathfrak{g}$ is nilpotent. In the latter, we apply the connected linearization in dimension 1 to derive the results.
\end{proof}
 \subsection{Dimension 3}
 In this section, we analyze connected Lie rings of dimension $3$, always under the hypothesis that the characteristic is greater than or equal to $5$. In particular, following the steps of the proof in \cite{deloro2023simple}, we prove the following.
 \begin{theorem}
   Let $\mathfrak{g}$ be a connected Lie ring of dimension $3$. If $\mathfrak{g}$ is not soluble, then:
   \begin{itemize}
       \item either $\mathfrak{g}\simeq \operatorname{sl}_2(K)$;
       \item or $\mathfrak{g}$ is a Borel of dimension $1$ and $(\mathfrak{g},+)\simeq K^3$ for a field $K$ of dimension $1$.
   \end{itemize}
 \end{theorem}
 
 We start verifying that we can reduce to the study of simple connected Lie rings of dimension $3$. Indeed, either $\mathfrak{g}$ is soluble or $\mathfrak{g}/Z(\mathfrak{g})$ is simple of dimension $3$.
 \begin{lemma}\label{gDim3Con}
     Let $\mathfrak{g}$ be a definable connected Lie ring of dimension $3$. Then, either $\mathfrak{g}$ is soluble or $Z(\mathfrak{g})$ is finite and $\mathfrak{g}/Z(\mathfrak{g})$ is simple with the DCC.
\end{lemma}
\begin{proof}
    Let $\mathfrak{g}$ be as in the hypothesis and assume, for a contradiction, that $\mathfrak{g}$ has an infinite proper definable ideal $\mathfrak{h}$. Then, $[\mathfrak{g},\mathfrak{h}]$ is a connected definable ideal. If it is infinite, $\mathfrak{g}/[\mathfrak{g},\mathfrak{h}]$ and $[\mathfrak{g},\mathfrak{h}]$ are soluble by Lemma \ref{Dim2} and Lemma \ref{Dim1}. If it is $\{0\}$, then $Z(\mathfrak{g})$ is infinite and again $\mathfrak{g}/Z(\mathfrak{g})$ is soluble. Therefore, in both cases, $\mathfrak{g}$ is soluble. Consequently, we may assume that $\mathfrak{g}$ has no infinite proper definable ideals. Then, $Z(\mathfrak{g})$ is finite and contains all the proper definable ideals. This implies that $\mathfrak{g}/Z(\mathfrak{g})$ is definably simple and so simple by Lemma \ref{DefSimImpSim}. Finally, it has the DCC by Corollary \ref{DCC}. 
\end{proof}
The proof for the not bad case is the same as in \cite{deloro2023simple}. Therefore, we analyze only the bad case.
\begin{lemma}
    Given $\mathfrak{g}$ a bad Lie ring of dimension $3$ and $\mathfrak{b}$ a nilpotent Borel. Then, $\dim(\mathfrak{b})$ is $1$ and $(\mathfrak{g},+)\simeq K^3$ for a definable field $K$ of dimension $1$.
\end{lemma}
\begin{proof}
    Let $\mathfrak{b}$ be a nilpotent connected Lie subring of dimension $2$.\\
    By \cite{deloro2023simple}, $\mathfrak{b}$ centralises every $1$-dimensional subquotient $\mathfrak{b}$-module $X=Y/Z$. This implies that $\mathfrak{b}$ has dimension $1$. If not, the action of $\mathfrak{b}$ on $(\mathfrak{g}/\mathfrak{b},+)$ would be trivial and so $\mathfrak{b}$ would be an ideal in $\mathfrak{g}$, a contradiction.\\
So, we may assume that $\mathfrak{b}$ is a Borel of dimension $1$. In particular, it is abelian by Lemma \ref{Dim1}. Taken any $b\in \mathfrak{b}$, the centraliser $C_{\mathfrak{g}}(b)$ has dimension $1$ and $[b,\mathfrak{g}]:=B_b$ has dimension $2$. Since $\mathfrak{b}$ is abelian, $\mathfrak{b}$ acts on $(B_b,+)$: given $b'\in \mathfrak{b}$ and $g\in\mathfrak{g}$, $[b',[b,g]]=[b,[b',g]]\in B_b$.\\
$B_b$ is $\mathfrak{b}$-minimal. Assume not and let $D$ be a sub $\mathfrak{b}$-module in $B_b$ of dimension $1$. By previous proof, $D\leq C_{\mathfrak{g}}(\mathfrak{b})$ that has dimension $2$, a contradiction. Therefore, we may assume that $B_b$ is minimal.\\
The action of $\mathfrak{b}$ cannot be trivial, if not there exists $b\in \mathfrak{b}$ such that the centraliser is of dimension $3$, a contradiction to triviality of $Z(G)$.\\
Since the action is not trivial and $B_b$ is minimal, then $B_b\simeq R^+$ for field $R$ of dimension $2$ and $\mathfrak{b}/\widetilde{C}_{\mathfrak{b}}(B_b)$ embeds in $R^+$. Take $A=[\mathfrak{b},[\mathfrak{b},[...,[\mathfrak{b},x]...]]$ of maximal dimension. Then, for every $b\in \mathfrak{b}$, $[b,A]$ has same dimension as $A$ and $[b',A]=[\mathfrak{b},A]=[b,A]$ for any $b,b'\in \mathfrak{b}$. Let $\phi:\mathfrak{b}\to R$ sending $b$ in $\phi(b)$ such that $[b,x]=\phi(b)\cdot x$. Then, for any $b\in \mathfrak{b}$, $\phi(b)\phi(b')^{-1}(A)$ is equal to $A$. This implies that $N_{R}(A)$ is an infinite subring. If its dimension is $2$, then $A$ is an ideal and so, being infinite, equal to $B_b$. By Claim 4.2.1 of \cite{rosengarten1991aleph}, we may conclude that $B_b$ is a Lie subring, a contradiction.\\
Therefore, $N_R(A)$ is a Lie subring and both $A$ and $K$ have dimension $1$. Moreover, for any $k\in N_R(A)$, $k^{-1}\cdot A\leq A$ iff $A\leq k\cdot A$, by connectivity. Therefore, $K$ is a definable subfield of $R$. Since $\mathfrak{b}\cap B_b$ is $\{0\}$, we may conclude that $(\mathfrak{g},+)\simeq K^3$ for a definable field $K$ of dimension $1$.
\end{proof} 
\subsection{Dimension 4}
We analyse the structure of a connected simple definable Lie ring $\mathfrak{g}$ of dimension $4$. Again, the proof follows \cite{deloro2023simple}, which verifies that there are no simple connected Lie rings of Morley rank $4$. As before, we distinguish two cases: the bad Lie rings and the not bad ones. By Lemma \ref{DCC}, $\mathfrak{g}$ has the DCC and the non-existence of not bad Lie rings of dimension four follows easily by the proof in \cite{deloro2023simple}.\\
For the bad case, Deloro and Ntsiri use, again, the Cherlin-Macintyre property of fields of finite Morley rank that does not hold for fields of finite dimension. On the other hand, we do not have an example of simple Lie ring of dimension $4$ (it has been proved that there are no simple Lie algebras of dimension $4$ over $\mathbb{R}$ \cite{popovych2003realizations}). Therefore, it is not clear if the result for Lie rings of finite Morley rank can be extended or not to the finite-dimensional context. Nevertheless, we reach the following conclusion.
\begin{lemma}\label{Dim4}
    Let $\mathfrak{g}$ be a simple connected definable bad Lie ring of dimension $4$. Then, any nilpotent Borel has dimension $1$ and $(\mathfrak{g},+)$ is isomorphic to $K^4$.
\end{lemma}
\begin{proof}
    Step 4.4.1 and 4.4.2 of \cite{deloro2023simple} can be applied straightforwardly. Consequently, we may assume that the Borel $\mathfrak{b}$ has dimension $1$. Since $\mathfrak{b}$ is a Borel, $C_{\mathfrak{g}}(b)^0=\mathfrak{b}$ for every $b\in \mathfrak{b}$. This implies that $\dim(B_b=[b,\mathfrak{g}])$ is $3$ for every $b\in \mathfrak{b}$. A consequence is that $B_b=B_h$ for every $b,h\in \mathfrak{b}$. Indeed $\dim(\mathfrak{g}/B_b)=1$ and $\mathfrak{g}/B_b$ is $\mathfrak{b}$ invariant. Then, it acts trivially and so $B_h\leq B_b$. The other inclusion follows by symmetry. $\mathfrak{b}\cap B_b$ is finite. Assume not, then $C^2_{\mathfrak{g}}(b)$ would be a proper definable connected Lie ring properly containing $\mathfrak{b}$, a contradiction. Moreover, $B_b$ is irreducible as $\mathfrak{b}$-module. If it has a module of dimension $1$, this is centralised by $\mathfrak{b}$, implying that $\dim(C_{\mathfrak{g}}(b))=2$, a contradiction. If it has a submodule $V$ of dimension $2$, $\mathfrak{b}$ acts trivially on $B_b/V$ and so $[\mathfrak{b},B_b]\leq V$ contradicting $C^2_{\mathfrak{g}}(b)^0=C^0_{\mathfrak{g}}(b)$ for every $b\in \mathfrak{b}$. Therefore, by Lemma \ref{LinDimCon}, there exists a definable field $R$ of dimension $3$ such that $B_b\simeq R^+$ and $\mathfrak{b}$ embeds in $R^+\cdot Id_R$. Proceeding as in \cite{deloro2023simple}, we may find $K$ a definable field of finite dimension such that $K\leq R$ and $R\not=K$. This implies that $\dim(R)<\dim(K)$ and $K$ is an $R$-vector space of finite dimension. Necessarily $\dim(K)=1$ and, since $\mathfrak{b}$ embeds in $K^3$, we have the conclusion.
\end{proof}
\section{Almost Lie ring theory}
In this section, we introduce a series of important instruments for the study of finite-dimensional non-virtually connected Lie rings. There are analogous notions for groups (for the definitions, see, for example, \cite{hempel}). Some of the results are proved in the setting of hereditarily $\widetilde{\mathfrak{M}}_c$-Lie rings, which includes finite-dimensional and simple (in the model-theoretic sense) Lie rings.
\begin{defn}
    Let $\mathfrak{g}$ a definable Lie ring. Then, $\mathfrak{g}$ is \emph{hereditarily $\widetilde{\mathfrak{M}}_c$} if, for any definable subgroup $N$ in $\mathfrak{g}$, there exist $n,d<\omega$ such that cannot exists $g_1,...,g_n\in \mathfrak{g}$ with $|C_{\mathfrak{g}}(g_1,...,g_i/N):C_{\mathfrak{g}}(g_1,...,g_i,g_{i+1})|\geq d$ for every $i\leq n-1$.
\end{defn}
By Lemma \ref{boundedind}, any finite-dimensional definable Lie ring is hereditarily $\widetilde{\mathfrak{M}}_c$-Lie ring.
\subsection{Almost centraliser}
We define the almost centraliser of the action of a Lie ring on a module.
\begin{defn}
    Given $\mathfrak{g}$ a Lie ring and $V$ a $\mathfrak{g}$-module. We define:
\begin{itemize}
    \item the \emph{almost centraliser of the action in $\mathfrak{g}$}, denoted $\widetilde{C}_{\mathfrak{g}}(V)$, as 
    $$\widetilde{C}_{\mathfrak{g}}(V)=\{g\in \mathfrak{g}:\ |V:C_V(g)|\text{ is finite}\}$$
    where $v\in C_V(g)$ iff $gv=0$;
    \item the \emph{almost centraliser of the action in $V$}, denoted $\widetilde{C}_V(\mathfrak{g})$ as 
    $$\widetilde{C}_V(\mathfrak{g})=\{v\in V:\ |\mathfrak{g}:C_{\mathfrak{g}}(v)|\text{ is finite}\}$$
    where $g\in C_{\mathfrak{g}}(v)$ iff $gv=0$.
\end{itemize}
\end{defn}
If we assume $V=(\mathfrak{g},+)$ with the adjoint action, we can define the almost centraliser of a subgroup.
\begin{defn}
    Let $\mathfrak{g}$ be a Lie ring and $H,K$ two subgroups. We define the \emph{almost centraliser of $K$ in $H$} as 
    $$\widetilde{C}_H(K)=\{h\in H:\ [h,K]\apprle K\}.$$
    This is a subgroup of $H$ for any $H,K$. The almost centraliser of $\mathfrak{g}$ in $\mathfrak{g}$ is called the \emph{almost center} and it is denoted by $\widetilde{Z}(G)$.
\end{defn}
In hereditarily $\widetilde{\mathfrak{M}}_c$-Lie rings, the almost centralisers are definable.
\begin{lemma}\label{DefZ}
    Let $\mathfrak{g}$ be a Lie ring and $V$ a $\mathfrak{g}$-module. Then, $\widetilde{C}_{\mathfrak{g}}(V)$ is an ideal of $\mathfrak{g}$ and $\widetilde{C}_V(\mathfrak{g})$ is a $\mathfrak{g}-$invariant subgroup of $V$. Moreover, in a finite-dimensional theory, if both $\mathfrak{g}$ and $V$ are definable, then the almost centralisers are also definable. The same holds in an hereditarily $\widetilde{\mathfrak{M}}_c$-Lie ring.
\end{lemma}
\begin{proof}
    Let $g,g'\in \widetilde{C}_{\mathfrak{g}}(V)$. Then, $(g+g')(V)\leq gV+g'V$ is clearly finite. Moreover, taken $g\in \mathfrak{g}, g'\in \widetilde{C}_{\mathfrak{g}}(V)$ then $[g,g'](v)=gg'(v)-g'g(v)$ and so $[g,g'](V)\leq gg'(V)-g'g(V)\leq gg'(V)+g'(V)=(g-Id)(g'(V))$ that is finite. This verifies that $\widetilde{C}_{\mathfrak{g}}(V)$ is an ideal in $\mathfrak{g}$.\\
Clearly $C_V(\mathfrak{g})$ is a subgroup. Taken $gv$ then 
$$\mathfrak{g}gv=-[\mathfrak{g},g](v)+g\mathfrak{g}v$$ that is finite since $\mathfrak{g}v$ is finite. The definability follows by Lemma \ref{boundedind}.\\
Finally, in the $\widetilde{\mathfrak{M}}_c$ case, the definability follows immediately from the definition.
\end{proof}
An important property of $\widetilde{\mathfrak{M}}_c$-Lie rings is the symmetry of the almost centraliser \hbox{i.e.} given $H,K,N$ definable subgroups of an hereditarily $\widetilde{\mathfrak{M}}_c$-Lie ring, then $H\apprle \widetilde{C}_{\mathfrak{g}}(K/N)$ iff $K\apprle \widetilde{C}_{\mathfrak{g}}(H/N)$.
The proof, that is the same as \cite{hempel}, uses the following fact, due to B. Neumann.
\begin{fact}\label{UniTraSub}
    Let $G$ be a group. Then $G$ cannot be covered by finitely many translates of subgroups of infinite index.
\end{fact}
We prove the symmetry of the almost centraliser.
\begin{lemma}\label{sym}
    Let $H,K,N$ definable subgroups in $\mathfrak{g}$ a definable hereditarily $\widetilde{\mathfrak{M}}_c$-Lie ring. Then $K\apprle \widetilde{C}_{\mathfrak{g}}(H/N)$ iff $H\apprle \widetilde{C}_{\mathfrak{g}}(K/N)$.
\end{lemma}
\begin{proof}
    Assume that $K\apprle \widetilde{C}_{\mathfrak{g}}(H/N)$ and, by contrary, that $H\not\apprle \widetilde{C}_{\mathfrak{g}}(K/N)$. Then, there exists an infinite family $\{h_i\}_{i<\omega}$ such that $h_i-h_j\not\in  \widetilde{C}_{\mathfrak{g}}(K/N)$. By Fact \ref{UniTraSub}, $K$ cannot be covered by a finite set of translates of finite unions of centralisers of $h_i-h_j$. This implies that the partial type 
    $$\{[h_i-h_j,k_i-k_j]\not\in N\}_{i,j<\omega}\cap \{k_i,k_j\in H\}$$
    is finitely satisfable. Up to work in a sufficiently saturated elementary extension, we can find a tuple that realises this type. Thus, $\{k_i\}_{i<\omega}$ is such that $C_{\mathfrak{g}}(k_i-k_j/N)$ does not almost contain $H$, a contradiction. 
\end{proof}
We introduce the notions of almost nilpotent Lie ring.
\begin{defn}
    Let $\mathfrak{g}$ a Lie ring. The \emph{almost central series} of $\mathfrak{g}$ is given by $\{\widetilde{Z}^n(\mathfrak{g})\}_{n<\omega}$ defined by induction as: 
    \begin{align*}
        &\widetilde{Z}^1(\mathfrak{g})=\widetilde{Z}(\mathfrak{g});\\
        &\widetilde{Z}^{n+1}(\mathfrak{g})/\widetilde{Z}^n(\mathfrak{g})=\widetilde{Z}(\mathfrak{g}/\widetilde{Z}^n(\mathfrak{g})).
    \end{align*}
    $\widetilde{Z}^n(\mathfrak{g})$ is called the \emph{$n$-almost center} of $\mathfrak{g}$.\\
    A Lie ring $\mathfrak{g}$ is \emph{almost nilpotent} if there exists $n<\omega$ such that $\mathfrak{g}=\widetilde{Z}^n(\mathfrak{g})$. $\mathfrak{g}$ is \emph{almost abelian} if $\mathfrak{g}=\widetilde{Z}(\mathfrak{g})$.
\end{defn}
We also define almost soluble Lie rings. 
\begin{defn}
    A Lie ring $\mathfrak{g}$ is \emph{almost soluble} if there exists a series 
    $$\mathfrak{g}=\mathfrak{g}_0\geq ...\geq \mathfrak{g}_n=\{0\}$$
    of Lie subrings such that $\mathfrak{g}_i$ is an ideal in $\mathfrak{g}_{i-1}$ and $\mathfrak{g}_{i-1}/\mathfrak{g}_i$ is almost abelian. The series $(\mathfrak{g}_i)_{i\leq n}$ is called an \emph{almost abelian series} of $\mathfrak{g}$. An almost soluble ideal in $\mathfrak{g}$ is said to have \emph{ideal almost abelian series} if it admits an almost abelian series $(\mathfrak{g}_i)_{i\leq n}$ such that any $\mathfrak{g}_i$ is an ideal in $\mathfrak{g}$.
\end{defn}
Applying iteratively Lemma \ref{DefZ}, we obtain the following corollary.
\begin{corollary}
    Given $\mathfrak{g}$ a Lie ring defined in a finite-dimensional theory, then the $n$-almost center of $\mathfrak{g}$ is a definable ideal for any $n<\omega$.
\end{corollary}
The following observation is quite easy.
\begin{observation}
    Let $\mathfrak{g}$ be a Lie ring such that $\widetilde{Z}^n(\mathfrak{g})$ is almost contained in $\widetilde{Z}^{n-1}(\mathfrak{g})$. Then, $\widetilde{Z}^n(\mathfrak{g})=\widetilde{Z}^{n+1}(\mathfrak{g})$.
\end{observation}
\begin{proof}
    Given $g\in \widetilde{Z}^{n+1}(\mathfrak{g})$ then $C_{\mathfrak{g}/\widetilde{Z}^n(\mathfrak{g})}(g)$ is of finite index in $\mathfrak{g}/\widetilde{Z}^n(\mathfrak{g})$. Since $\widetilde{Z}^n(\mathfrak{g})$ is almost contained in $\widetilde{Z}^{n-1}(\mathfrak{g})$, $[g,\mathfrak{g}]$ is almost contained in $\widetilde{Z}^{n-1}(\mathfrak{g})$ and so $g\in \widetilde{Z}^n(\mathfrak{g})$.
\end{proof}
In particular, we derive the following corollary.
\begin{corollary}\label{g/Z(g)}
    Let $\mathfrak{g}$ be a Lie ring such that $\widetilde{Z}(\mathfrak{g})$ is finite. Then, $\mathfrak{g}/\widetilde{Z}(\mathfrak{g})$ has no almost center.
\end{corollary}
In the last section, we will prove that an almost nilpotent Lie subring of a hereditarily $\widetilde{\mathfrak{M}}_c$-Lie ring is almost contained in a definable nilpotent Lie ring and, similarly, an almost soluble Lie ring is almost contained in a soluble definable Lie subring. These will be the main ingredients for the proof of the solubility of the Radical ideal and the nilpotency of the Fitting ideal.
\subsection{Almost ideals and absolute simplicity}
In the analysis of finite-dimensional Lie rings and hereditarily $\widetilde{\mathfrak{M}}_c$-Lie rings, an important role is played by almost ideals.
\begin{defn}
Let $(\mathfrak{g},V)$ a module and $W\leq V$ a subgroup. The Lie subring
$$\widetilde{Stab}_{\mathfrak{g}}(W)=\{g\in \mathfrak{g}:\ g\cdot W\apprle W\}$$
is called the \emph{almost stabiliser} of $W$. $W$ is \emph{almost invariant} if $\widetilde{Stab}_{\mathfrak{g}}(W)=\mathfrak{g}$. A definable module $(\mathfrak{g},V)$ is \emph{absolutely minimal} if there are no definable infinite almost invariant subgroups of infinite index. \\
    Let $\mathfrak{g}$ be a Lie ring and $H$ be a subgroup of $\mathfrak{g}$. The \emph{almost normaliser} of $H$, denoted by $\widetilde{N}_{\mathfrak{g}}(H)$, is the almost stabiliser of $H$ for the adjoint action of $\mathfrak{g}$ on $(\mathfrak{g},+)$. An \emph{almost ideal} is a subgroup of $\mathfrak{g}$ almost invariant for the adjoint action of $\mathfrak{g}$.
\end{defn}
In the hereditarily $\widetilde{\mathfrak{M}}_c$ case, we have the following equivalent definition of almost ideality.
\begin{lemma}\label{ChaAlmIde}
    Let $\mathfrak{g}$ be a definable hereditarily $\widetilde{\mathfrak{M}}_c$-Lie ring and $H$ a definable subgroup. Then, $H$ is an almost ideal iff there exists $H_1$ definable subgroup in $H$ of finite index such that, for every $h\in H_1$, $[h,\mathfrak{g}]\apprle H$.
\end{lemma}
\begin{proof}
    By definition, for every $g\in \mathfrak{g}$, $[g,H]\apprle H$. Therefore, $\mathfrak{g}\leq \widetilde{C}_{\mathfrak{g}}(H/H)$. By Lemma \ref{sym}, $\widetilde{C}_{\mathfrak{g}}(\mathfrak{g}/H)\apprge H$. $\widetilde{C}_{\mathfrak{g}}(\mathfrak{g}/H)$ is a subgroup in $\mathfrak{g}$ and $\widetilde{C}_{\mathfrak{g}}(\mathfrak{g}/H)\cap H$ is an almost invariant subgroup of finite index in $H$ with the desired property. The vice-versa follows from a similar proof.
\end{proof}
The following easy lemma will be very important in the proof of the existence of definable envelopes. 
\begin{lemma}\label{almide1}
    Let $\mathfrak{g}$ be a Lie ring and $K\leq H$ subgroups. Then, $$N_{\mathfrak{g}}(\widetilde{C}_{\mathfrak{g}}(H/K))\geq \widetilde{N}_{\mathfrak{g}}(H)\cap \widetilde{N}_{\mathfrak{g}}(K).$$
\end{lemma}
\begin{proof}
    Let $g\in \widetilde{N}_{\mathfrak{g}}(H)\cap \widetilde{N}_{\mathfrak{g}}(K)$ and $c\in \widetilde{C}_{\mathfrak{g}}(H/K)$. Then 
    $[g,c]\in \widetilde{C}_{\mathfrak{g}}(H/K)$ indeed 
    $$[[g,c],H]\leq [g,[c,H]]+[c,[g,H]]\apprle [g,K]+[c,H]$$
    by assumption. By hypothesis, this is almost contained in $K+K\leq K$. By arbitrariety of $c\in \widetilde{C}_{\mathfrak{g}}(H/K)$, this completes the proof.
\end{proof}
An immediate corollary of Lemma \ref{almide1} is the following.
\begin{corollary}\label{C(H/K)AlmId}
    Let $\mathfrak{g}$ be a Lie ring and $K\leq H$ almost ideals. Then, $\widetilde{C}_{\mathfrak{g}}(H/K)$ is an ideal in $\mathfrak{g}$.
\end{corollary}
Another corollary that will be useful in the last section is the almost ideality of the almost center of an almost ideal.
\begin{lemma}
    Let $\mathfrak{g}$ be a Lie ring and $\mathfrak{h}$ an almost ideal in $\mathfrak{g}$. Then, for any $n<\omega$, $\widetilde{Z}_n(\mathfrak{h})$ is an almost ideal in $\mathfrak{g}$.
\end{lemma}
\begin{proof}
    By induction on $n<\omega$. By definition, $\widetilde{Z}(\mathfrak{h})=\widetilde{C}_{\mathfrak{g}}(\mathfrak{h})\cap \mathfrak{h}$. Since the first is an ideal by Corollary \ref{C(H/K)AlmId} and the second is an almost ideal by assumptions, the conclusion follows.\\
    Similarly $\widetilde{Z}_n(\mathfrak{h})=\widetilde{C}_{\mathfrak{g}}(\mathfrak{h}/\widetilde{Z}_{n-1}(\mathfrak{h}))\cap \mathfrak{h}$ and $\widetilde{C}_{\mathfrak{g}}(\mathfrak{h}/\widetilde{Z}_{n-1}(\mathfrak{h}))$ is an ideal by Corollary \ref{C(H/K)AlmId} and induction hypothesis.
\end{proof}
The following result is the Lie ring version of Proposition 3.27 of \cite{hempel}.
\begin{lemma}\label{fin[]Alm}
    Let $\mathfrak{g}$ be a definable hereditarily $\widetilde{\mathfrak{M}}_c$-Lie ring, $K$ a definable almost ideal and $H$ a definable subgroup of $\mathfrak{g}$. Then,
    $$[\widetilde{C}_{\mathfrak{g}}(H/K),\widetilde{C}_{H}(\widetilde{C}_{\mathfrak{g}}(H/K)/K)]$$
    is almost contained in $K$. Moreover, if $H$ is an almost ideal, it is normalised by $H$.
\end{lemma}
\begin{proof}
Let $\mathfrak{b}=\widetilde{C}_{\mathfrak{g}}(H/K)$ and $W=\widetilde{C}_{H}(\widetilde{C}_{\mathfrak{g}}(H/K)/K)$.\\
 By definition, for every $w\in W$, $[\mathfrak{b},w]+K$ is finite in $(\mathfrak{g}/K,+)$. By compactness, there exists $u$ with $[\mathfrak{b},u]+K/K$ of maximal index equal to $n$. Let $\{b_1=0,...,b_n\}$ be a transversal of $C_{\mathfrak{b}}(u/K)$ in $\mathfrak{b}$. Define the subgroup $C:=\bigcap_{i\not=j=1}^n C_W(b_i-b_j/K)$ of finite index in $W$. Finally, let $\{w_1,...,w_m\}$ be a transversal of $C$ in $W$ and let $F=\sum_{i=1}^m [\mathfrak{b}, w_i]+K$. Any $w_i$ has orbit of maximal cardinality since $[b_j,w_i]+K\not=[b_k,w_i]+K$ by hypothesis. If we verify that 
 $$\{[b, w]:\ b\in \mathfrak{b},w\in W\}+K\subseteq \mathfrak{g}/K$$
 is contained in $F$, the theorem is proven, being $F$ a finite extension of $K$. Take $v\in W$ such that $v=c+w_i$ for $c\in C$ and $w_i\not=0$. Then $[b_j,c+w_i]=[b_j, w_i]+K$ by definition. Since any $[b_j, w_i]+K\not=[b_k, w_i]+K$ for $j\not=k\leq n$ and by maximality of $n$, $[\mathfrak{b}, v]\leq F$.
 Given $c\in C$, then $[\mathfrak{g},c]\leq [\mathfrak{g},c-u]+[\mathfrak{g},u]\leq F+F=F$. Therefore, $[\mathfrak{g},H]\leq F$ that is a finite extension of $K$.\\
 For the second part, assume that $H$ is an almost ideal. Then, it is sufficient to prove that, for $b\in \mathfrak{b}$, $w\in W$ and $h\in H$, $[h,[b,w]]\in [\widetilde{C}_{\mathfrak{g}}(H/K),\widetilde{C}_{H}(\widetilde{C}_{\mathfrak{g}}(H/K)/K)]$. By Jacobi identity, $[h,[b,w]]=[w,[h,b]]+[b,[w,h]]$. The second term is clearly in $A$. For the second, it is sufficient to observe that $[[h,h'],\widetilde{C}_{\mathfrak{g}}(H/K)]$ for $h'\in \widetilde{C}_{H}(\widetilde{C}_{\mathfrak{g}}(H/K))$ is equal to $[h',[h,\widetilde{C}_{\mathfrak{g}}(H/K)]]+[h,[h',\widetilde{C}_{\mathfrak{g}}(H/K)]]\apprle [h,K]+[h',\widetilde{C}_{\mathfrak{g}}(H/K)]\apprle K+K=K$. 
\end{proof}
Another interesting property is the following.
\begin{corollary}\label{IntAlmIde}
    Let $\mathfrak{g}$ be a hereditarily $\widetilde{\mathfrak{M}}_c$-Lie ring and $H,H_1$ two definable almost ideals with finite intersection. Then, $H\apprle \widetilde{C}_{\mathfrak{g}}(H_1)$ and vice-versa.
\end{corollary}
\begin{proof}
    To prove that $H$ is almost contained in $\widetilde{C}_{\mathfrak{g}}(H_1)$, it is sufficient to show that there exists a definable subgroup of finite index $K$ in $H$ such that, for every $k\in K$, $[k,H_1]\apprle H\cap H_1$. Take $K=\widetilde{C}_{H}(\mathfrak{g}/H):=\{h\in H:\ [\mathfrak{g},h]\apprle H\}$. $K$ is of finite index in $H$ by Lemma \ref{sym} and since $H$ is an almost ideal. For any $k\in K$, $[k,H_1]\leq [k,\mathfrak{g}]\apprle H$. On the other hand, by definition of almost invariant subgroup, $[h,H_1]\apprle H_1$. This proves the corollary.
\end{proof}
Returning to finite-dimensional Lie rings, we have the following Lemma.
\begin{lemma}\label{DefBrack}
    Let $\mathfrak{g}$ be a finite-dimensional definable Lie ring and $\mathfrak{h},\mathfrak{k}$ two definable Lie subrings that normalise each other. Then, $[\mathfrak{h},\mathfrak{k}]$ is a definable Lie subring normalised by $\mathfrak{h}$ and $\mathfrak{k}$.
\end{lemma}
\begin{proof}
    Let $H=\sum_{i=1}^n [k_i,\mathfrak{h}]+\sum_{j=1}^m [\mathfrak{k},h_j]$ be a finite sum of subgroups of maximal dimension among all the sums of this form. By definition, $H\leq [\mathfrak{k},\mathfrak{h}]\leq \mathfrak{h}\cap\mathfrak{k}$ since the two Lie subrings normalize each other. Then, $\widetilde{N}_{\mathfrak{g}}(H)$ contains $\mathfrak{h}+\mathfrak{k}$. Indeed, for $g\in \mathfrak{h}\cup \mathfrak{k},$ $[g,H]+H\leq [g,\mathfrak{h}]+H\apprle H$ by maximality of the dimension. Moreover, for every $h\in \mathfrak{h}$, $[h,\mathfrak{k}]\apprle H$ by maximality of the dimension and $[k,\mathfrak{h}]\apprle H$ for any $k\in \mathfrak{k}$. This implies that $\widetilde{C}_{\mathfrak{g}}(\mathfrak{h}/H)\geq \mathfrak{k}$ and vice-versa. Therefore, $[\widetilde{C}_{\mathfrak{k}}(\mathfrak{h}/H),\widetilde{C}_{\mathfrak{h}}(\mathfrak{k}/H)]=[\mathfrak{k},\mathfrak{h}]\apprle H\leq [\mathfrak{h},\mathfrak{k}]$ by the same proof of Lemma \ref{fin[]Alm}. Therefore, $[\mathfrak{h},\mathfrak{k}]$ is a finite extension of $H$ and so $[\mathfrak{h},\mathfrak{k}]$ is definable and clearly normalised both by $\mathfrak{h}$ and $\mathfrak{k}$.
\end{proof}
A corollary of Lemma \ref{DefBrack} is that the commutator between two definable ideals in $\mathfrak{g}$ is a definable ideal.
\begin{corollary}\label{DerivedSeries}
    Given $\mathfrak{g}$ a definable Lie ring of finite dimension and $I,J$ two definable ideals in $\mathfrak{g}$. Then $[I,J]$ is a definable ideal. In particular, for every $n<\omega$, $\mathfrak{g}^{(n)}$ is a definable ideal.
\end{corollary}
The notion of absolutely simple Lie rings substitutes the notion of definably simple in non-virtually-connected finite-dimensional Lie rings.
\begin{defn}
    Let $\mathfrak{g}$ be a definable Lie ring of finite dimension. Then, $\mathfrak{g}$ is \emph{absolutely simple} if for any definable Lie subring $\mathfrak{h}$ of finite index in $\mathfrak{g}$, $\mathfrak{h}$ has no proper definable infinite ideal of infinite index. 
\end{defn}
The following is an important characterization of absolutely simple Lie rings in the hereditarily $\widetilde{\mathfrak{M}}_c$-case.
\begin{lemma}\label{AlmSim}
Let $\mathfrak{g}$ be an hereditarily $\widetilde{\mathfrak{M}}_c$-Lie ring. $\mathfrak{g}$ is absolutely simple iff it has no definable Lie subring of finite index with an infinite definable almost ideal of infinite index.
\end{lemma}
\begin{proof}
Clearly, the second implies the first. Assume, for a contradiction, that $\mathfrak{g}$ is a is an hereditarily $\widetilde{\mathfrak{M}}_c$-Lie ring such that any definable Lie subring of finite index in $\mathfrak{g}$ does not contain an infinite ideal of infinite index, but there exists a definable Lie subring $\mathfrak{h}$ of finite index containing an infinite almost ideal $H$ of infinite index.\\
By Lemma \ref{ChaAlmIde}, there exists a definable almost ideal $H_1$ of finite index in $H$ such that $\widetilde{C}_{H_1}(\mathfrak{h}/H_1)=\{h\in H_1:\ [\mathfrak{h},h]\apprle H_1\}=H_1$. Let $\mathfrak{h}_1=\widetilde{Z}(\mathfrak{h}/H_1)=\{h\in \mathfrak{h}:\ [h,\mathfrak{h}]\apprle H_1\}$. By Lemma \ref{C(H/K)AlmId}, $\mathfrak{h}_1$ is an ideal in $\mathfrak{h}$ containing $H_1$. If it is of infinite index, the Lemma is proved. Therefore, assume that it is of finite index and work in $\mathfrak{h}_1$. By Lemma \ref{fin[]Alm}, $[\mathfrak{h}_1,\mathfrak{h}_1]+H_1$ is a finite extension of $H_1$ and therefore definable, infinite, and of infinite index in $\mathfrak{h}_1$. Since it contains the commutator, it is an ideal in $\mathfrak{h}_1$, contradicting absolute simplicity.
\end{proof}
For finite-dimensional Lie rings, we obtain another equivalent definition of absolute simplicity. We need the following lemma.
\begin{lemma}\label{IdeFromNDIdeal}
    Let $\mathfrak{g}$ be a definable Lie ring of finite dimension. Let $I$ be an ideal in $\mathfrak{g}$ (not necessarily definable). Then, we have three possibilities:
    \begin{itemize}
        \item $I$ is almost contained in $\widetilde{Z}(\mathfrak{g})$;
        \item $I$ is of finite index in $\mathfrak{g}$ and definable;
        \item $\mathfrak{g}$ contains a definable infinite almost ideal of infinite index.
    \end{itemize}
\end{lemma}
\begin{proof}
    If $I$ is finite, it is clearly contained in $\widetilde{Z}(\mathfrak{g})$. Let $I$ be infinite and assume $I\not\leq \widetilde{Z}(\mathfrak{g})$. Then, for every $i\in I$, $[i,\mathfrak{g}]\leq I$. Let $J=\sum_{j=1}^n [i_j,\mathfrak{g}]\leq I$ a sum of maximal dimension with $i_j\in I$ for every $j$. If this is finite, it means that $I\leq \widetilde{Z}(\mathfrak{g})$, a contradiction. We prove that $J$ is an almost ideal. Given $g\in \mathfrak{g}$, $[g,J]=\sum_{i=1}^n [g,[\mathfrak{g},i_j]]$. For every $j$, $[g,[\mathfrak{g},i_j]]\leq [[g,\mathfrak{g}],i_j]+[\mathfrak{g},[g,i_j]]\leq [\mathfrak{g},i_j]+[\mathfrak{g},[g,i_j]]$. Since $[g,i_j]\in I$ and by maximality of the dimension of $J$, $J+[g,J]\apprle J$ \hbox{i.e.} $J$ is an almost ideal of $\mathfrak{g}$. If it is of finite index in $\mathfrak{g}$, $I$ is a finite extension of $J$ and therefore definable.
\end{proof}
In conclusion, we derive the following corollary.
\begin{corollary}\label{NoInfIde}
    Let $\mathfrak{g}$ be a definable Lie ring of finite dimension not almost abelian. $\mathfrak{g}$ is absolutely simple iff every definable Lie subring of finite index has no infinite ideals (also not definable) of infinite index.
\end{corollary}
\begin{proof}
    Clearly, the absolute simplicity is implied by the other assumption. On the other hand, assume that $\mathfrak{g}$ has an infinite ideal of infinite index. By Lemma \ref{IdeFromNDIdeal}, there exists a definable infinite almost ideal of infinite index. This contradicts Lemma \ref{AlmSim}. 
\end{proof}
\subsection{Linearising in the non virtually connected case}
We state the extension of Lemma \ref{LinDimCon} for non-virtually-connected modules. We need some preliminary notions, as essentially infiniteness and unboundness.
\begin{defn}
    Let $A$ be an abelian group and $\phi,\phi'$ two endomorphisms of $A$. Then, $\phi$ and $\phi'$ are \emph{equivalent}, denoted by $\phi\sim\phi'$, if $im(\phi-\phi')$ is finite. This is clearly an equivalence relation on $\operatorname{End}(A)$.\\
    A subset $X$ of $\operatorname{End}(A)$ is \emph{essentially infinite} if $X/\sim $ is infinite. The essentially unboundedness is defined similarly.
\end{defn}
We define absolutely minimal modules.
\begin{defn}
    Let $A$ be a definable abelian group, $B\leq A$ and $X\subseteq \operatorname{End}(A)$. Then, $B$ is \emph{almost $X$-invariant} if $xB\apprle B$ for any $x\in X$. $A$ is \emph{absolutely $X$-minimal} if $A$ has no infinite almost $X$-invariant definable subgroups $B$ of infinite index.
\end{defn}
We state Theorem B from \cite{InvittiDim}.
\begin{theorem}\label{End}
    Let $A$ be an abelian definable group and $\Gamma,\Delta$ two invariant rings of definable endomorphisms such that:
    \begin{itemize}
        \item $C(\Gamma)=\Delta$ and $C(\Delta)=\Gamma$;
        \item $A$ is absolutely $(\Gamma,\Delta)$-minimal;
        \item $\Gamma$ is essentially unbounded and $\Delta$ is essentially infinite or vice-versa;
    \end{itemize}
    Then, there exists a definable infinite field $K$ such that $A/A_0$ is a finite-dimensional $K$-vector space with $A_0$ a finite subgroup in $A$. Moreover, $\Gamma/{\sim}$ and $\Delta/{\sim}$ act $K$-linearly on $A/A_0$.
\end{theorem}
We derive the following corollary of Lemma \ref{End} in case $\Gamma=\mathfrak{g}$ and $A=V$ where $(\mathfrak{g},V)$ is a definable module.
\begin{corollary}\label{LinNVC}
    Let $(\mathfrak{g},V)$ be a definable $\mathfrak{g}$-module in a finite-dimensional theory. Assume that:
    \begin{itemize}
        \item $\mathfrak{g}/\widetilde{C}_{\mathfrak{g}}(V)$ is infinite;
        \item $C_{\operatorname{End}(V)}(\mathfrak{g})/\sim$ is infinite; 
        \item $V$ is absolutely $\mathfrak{g}$-minimal.
    \end{itemize} 
    Then, there exists a finite $\mathfrak{g}$-module $A$ in $V$ such that $V/A$ is a finite-dimensional vector space over a definable field $K$ of finite dimension and $\mathfrak{g}/\widetilde{C}_{\mathfrak{g}}(V)$ embeds in $K^+\cdot Id$.
\end{corollary}
In particular, if in the previous hypothesis, $\dim(V)=1$, then $\mathfrak{g}/\sim$ is isomorphic to a subgroup of finite index in $K^+$. This result is called "linearization in dimension $1$".
\begin{lemma}\label{LinDim1}
    Let $\mathfrak{g}$ be a Lie ring of finite dimension acting on a $\mathfrak{g}$-module of dimension $1$ and assume that $Z(\mathfrak{g})\cap \widetilde{C}_{\mathfrak{g}}(A)$ is not of finite index in $Z(\mathfrak{g})$. Then, there exists a finite subgroup $A_0$ in $A$ and a definable field $K$ of dimension $1$ such that $A/A_0\simeq K^+$ and $\mathfrak{g}/\widetilde{C}_{\mathfrak{g}}(A)$ is isomorphic to a subgroup of finite index in $K^+$.
\end{lemma}
The opposite case is when $\widetilde{C}_{\mathfrak{g}}(V)\sim \mathfrak{g}$. In this case, the module is called \emph{almost trivial}.
\begin{lemma}
    Let $(\mathfrak{g},V)$ be a $\mathfrak{g}$-module definable in a finite-dimensional theory. Then, $\widetilde{C}_{\mathfrak{g}}(V)$ is of finite index in $\mathfrak{g}$ iff $\widetilde{C}_V(\mathfrak{g})$ is of finite index in $V$. If one of them (and so both) holds, the module is \emph{almost trivial}.
\end{lemma}
\begin{proof}
We assume $\widetilde{C}_V(\mathfrak{g)}$ of finite index. We verify that $\widetilde{C}_{\mathfrak{g}}(V)$ is of finite index. The proof of the other direction follows by symmetry.\\
    Define $X=\{(g,v)\in \mathfrak{g}\times V:\ gv=0\}$. We verify that $\dim(X)=\dim(\mathfrak{g})+\dim(V)$. $X$ contains $Y:=\{(g,v)\in \mathfrak{g}\times \widetilde{C}_V(\mathfrak{g}):\ gv=0\}$. By definition, the dimension of every fiber of the second projection is of dimension equal to $\dim(\mathfrak{g})$. Consequently,
    $$\dim(\mathfrak{g})+\dim(V)\geq \dim(X)\geq \dim(\widetilde{C}_V(\mathfrak{g}))+\dim(\mathfrak{g})=\dim(\mathfrak{g})+\dim(V)$$
    since $\widetilde{C}_V(\mathfrak{g})$ has same dimension as $V$.\\
    On the other hand, $X$ is the union of $Y_1:=\{(g,v)\in C_{\mathfrak{g}}(V)\times V:\ gv=0\}$ and $Y_1^c$. Therefore, by union, $\dim(X)=\max\{\dim(Y_1),\dim(Y_1^c)\}$. Working as before with the first projection, we have that $\dim(Y_1)=\dim(\widetilde{C}_{\mathfrak{g}}(V))+\dim(V)$, while, by lower fibration, $\dim(Y_1^c)\leq \dim(\mathfrak{g})+\dim(V)-1$. Therefore $\dim(X)=\max\{\dim(\widetilde{C}_{\mathfrak{g}}(V))+\dim(V),\dim(\mathfrak{g})+\dim(V)-1\}$ and so $\dim(\widetilde{C}_{\mathfrak{g}}(V))=\dim(\mathfrak{g})$.
\end{proof}
We have that $[\widetilde{C}_{\mathfrak{g}}(V),\widetilde{C}_{V}(\mathfrak{g})]$ is finite. The proof follows as in Lemma \ref{fin[]Alm}
\begin{lemma}\label{fin[]}
    Let $(\mathfrak{g},V)$ be a definable module in a finite-dimensional theory. Then $[\widetilde{C}_{\mathfrak{g}}(V),\widetilde{C}_{V}(\mathfrak{g})]$ is finite.
\end{lemma}
If $A$ has finite $n$-torsion for any $n<\omega$, we may reduce to the connected case, thanks to the following corollary of \cite[Lemma 11.1]{InvittiDim}.
\begin{corollary}\label{LinChar0}
    Let $\Gamma$ be a ring of definable endomorphisms of the abelian group $A$. Assume that:
    \begin{itemize}
        \item $\Gamma$ is essentially unbounded;
        \item $A$ is $\Gamma$-minimal (not necessarily almost minimal);
        \item $A[n]$ is finite for any $n\in \mathbb{N}$.
    \end{itemize}
    Then, either there exists $B$ almost $\Gamma$-invariant such that $\Gamma_B/\sim$ is bounded or $A$ is virtually connected.
\end{corollary}
Let $\mathfrak{g}$ be a Lie ring acting on a definable abelian group $A$ not of torsion. Being $A[n]$ a $\mathfrak{g}$-invariant definable subgroup for any $n$, then either it is finite or $A[n]$ is of finite index. Then, given $m=|A/A[n]|$, for any $a\in A$, $a^{nm}=(a^m)^n=0$ since $a^m\in A[n]$. Assume that the action is not almost trivial, then we are in the hypothesis of Corollary \ref{LinChar0}. Assume there exists $B$ such that the ring generated by $\mathfrak{g}$ is not essentially unbounded. Then, $\mathfrak{g}_B/\sim=\mathfrak{g}/\widetilde{C}_{\mathfrak{g}}(B)$ is bounded iff it is finite. This implies that $B\leq \widetilde{C}_{A}(\mathfrak{g})$. By minimality, the action is almost trivial. Therefore, we derive the following corollary.
\begin{corollary}
    Let $\mathfrak{g}$ be a definable Lie ring of finite dimension acting not almost trivially on a definable abelian group $A$ with finite $n$-torsion for every $n<\omega$. Then, $A$ is virtually connected and with finite torsion. Moreover, by Lemma \ref{LinDimCon},there exists a field $K$ such that $A^{0}/T(A^0)$ is a $K$-vector space of linear dimension $n$ and $\mathfrak{g}/\widetilde{C}_{\mathfrak{g}}(A)$ embeds in $\mathfrak{gl}_n(K)$ as Lie ring.
\end{corollary}
These results will be applied in the following section for the analysis of Lie rings of characteristic $0$.
\section{Dimensional Lie rings of characteristic 0}
In this section, we prove that the finite-dimensional version of the "log-CZ"-conjecture holds for Lie rings of finite dimension and characteristic $0$.\\
This is clearly the easiest case, but it is not as immediate as in the finite Morley rank context, since a field of finite dimension is not necessarily algebraically closed.\\
Given $\mathfrak{g}$ a Lie ring definable in a finite-dimensional theory, either $\mathfrak{g}[p]$ is an infinite definable ideal or $\mathfrak{g}[p]\leq \widetilde{Z}(\mathfrak{g})$. In the first case, we may reduce to the analysis of a Lie ring of prime characteristic. In the second, we may quotient by $\widetilde{Z}(\mathfrak{g})$, that is a definable ideal in $\mathfrak{g}$ by Lemma \ref{DefZ}. The Lie ring $\mathfrak{g}/\widetilde{Z}(\mathfrak{g})$ cannot have infinite $p$-torsion for any $p$ prime. If not, let $A$ be a definable ideal in $\mathfrak{g}$ such that $A/\widetilde{Z}( \mathfrak{g})=(\mathfrak{g}/\widetilde{Z}(\mathfrak{g}))[p]$. Then, the homomorphism 
$$p\cdot\_:A\to \widetilde{Z}(\mathfrak{g})$$
has infinite kernel. Therefore, $\mathfrak{g}[p]$ is infinite, a contradiction. Iterating this process, we obtain the following conclusion.
\begin{lemma}\label{Cha0orNil}
    Let $\mathfrak{g}$ be a definable Lie ring of finite dimension with $\mathfrak{g}[p]$ finite for all $p$ prime. Then, either $\mathfrak{g}$ is almost nilpotent (and it is virtually nilpotent by Lemma \ref{NilFromAlmNil}) or there exists $n$ such that $\widetilde{Z}^n(\mathfrak{g})=\widetilde{Z}^{n+1}(\mathfrak{g})$ and $\mathfrak{g}/\widetilde{Z}^{n+1}(\mathfrak{g})$ is of characteristic 0.
\end{lemma}
Therefore, we can concentrate on the study of Lie rings of characteristic 0.\\
We verify that any Lie ring of finite dimension and characteristic $0$ with finite almost center is isogenous to a Lie algebra of finite dimension over a field of characteristic $0$. This implies that it is connected.
\begin{Char0}
  Let $(\mathfrak{g},V)$ be definable module of finite dimension with $char(\mathfrak{g})=0$. Assume that $V$ is minimal and that the action is not almost trivial. Then, $\mathfrak{g}/\widetilde{C}_{\mathfrak{g}}(V)$ embeds in $\operatorname{gl}_n(K)$ for a certain $n\in \omega$ and a definable field $K$ of characteristic $0$. Moreover, it is virtually connected and $(\mathfrak{g}/\widetilde{C}_{\mathfrak{g}}(V))^0$ is a Lie algebra over a definable finite-dimensional field of characteristic $0$.  
\end{Char0}
\begin{proof}
   Since the action is not almost trivial, $\widetilde{C}_V(\mathfrak{g})$ is finite. This contains the torsion of $A$. Therefore, up to factorise by $\widetilde{C}_A(\mathfrak{g})$, we may assume that $A$ is without torsion. Then, by Corollary \ref{LinChar0}, this action can be linearised \hbox{i.e.} there exists a definable field $K$ of finite dimension and characteristic $0$ such that $A/\widetilde{C}_A(\mathfrak{g})\simeq K^n$ and $\mathfrak{g}/\widetilde{C}_{\mathfrak{g}}(V)$ embeds in $\operatorname{gl}_n(K)$ as a homomorphism of Lie rings. Since $\operatorname{gl}_n(K)$ is a vector space over $K$ of characteristic $0$, it has the DCC by Corollary \ref{DCCK^n}. From now on, we may assume the almost centralisers equal to $0$. We define 
   $$R=\{k\in K:\ \forall g\in \mathfrak{g}^0\ k\cdot g\in \mathfrak{g}^0\}$$
   where $\mathfrak{g}$ is seen as a Lie subring of $\mathfrak{gl}_n(K)$.
   $R$ is a subfield of $K$. Indeed, it is clearly closed for sum, opposite, and product. We verify that it is closed for the inverse. By dimensionality, since the multiplication by $r\in R$ is an automorphism of $\mathfrak{gl}_n(K)$ that preserves $\mathfrak{g}^0$, the image $r(\mathfrak{g}^0)$ is a subgroup of finite index in $\mathfrak{g}^0$. Being the latter connected, $r(\mathfrak{g}^0)=\mathfrak{g}^0$. Let $r\in R$, then, by previous proof, for any $g\in \mathfrak{g}^0$, $g=r\cdot g'$. Then $r^{-1}\cdot g=r^{-1}\cdot r\cdot g'=g'\in \mathfrak{g}^0$. Moreover, $R$ is infinite, since it contains $\mathbb{N}$. This implies that $\mathfrak{g}^0$ is a Lie algebra of finite dimension over a definable field $R$ of finite dimension.
\end{proof}
\section{NIP Lie rings}
We extend the results on connected Lie rings of finite dimension to Lie rings definable in a NIP theory of finite dimension. As a consequence, we obtain a characterization of stable definable Lie rings of Lascar rank less than or equal to $4$.\\
The NIP (not the independence property) is a model-theoretic condition introduced by Shelah (for an exhaustive analysis of NIP theories see \cite{simon2015guide}). Subclasses of NIP theories are the stable and the $o$-minimal ones. We recall the definition of NIP theory.
\begin{defn}
    A formula $\phi(x;y)$ is said to have the independence property if there exists an infinite set of $|x|$-tuples $A$ and a set $\{b_I:\ I\subseteq A\}$ of $|y|$-tuples such that 
    $$\phi(a,b_I)\iff a\in I$$
    for all $a\in A$.\\
    A theory $T$ is said to be NIP if no formulas in $T$ have the independence property.
\end{defn}
 NIP theories comprehend structures as Henselian fields of characteristic $0$ with NIP residue field and ordered group \cite{delon1981types}, the fields of $p$-adics \cite{belair1999types}, and ordered abelian groups \cite{gurevich1984theory}. The interest in NIP theories increased in particular after 2000: in a series of papers, Shelah characterizes the basic properties of forking in NIP theories. In particular, we can still define generics linked to forking, called $f$-generics. This work culminates in \cite{shelah2012dependent}, where he proves NIP theories have few types over saturated models. Fundamental objects in the study of NIP theories are Keisler measures, introduced by Keisler in \cite{keisler1987measures}, and analyzed by Hrushovski, Peterzil, and Pillay in \cite{hrushovski2008groups}. A NIP group with a definable Keisler measure is called \emph{definable amenable}. This is linked to the existence of a global $f$-generic by Stonestrom in \cite{stonestrom2023f}.\\
 In NIP theories, there exists a notion of "rank" that is called the \emph{dp-rank} (for definition and property see, for example, \cite{halevi2019dp}). The relation between dp-rank and dimension in NIP finite-dimensional theories is still mysterious. 
\subsection{NIP groups of finite dimension}
In this article, we use only two basic instruments of the study of NIP groups: the Baldwin-Saxl condition and the existence of a type-definable connected component.\\
The Baldwin-Saxl condition is a very important chain condition on uniformly definable subgroups. For a proof, originally made by Shelah, see, for example, Theorem 2.13 of \cite{simon2015guide}.
\begin{lemma}(Baldwin-Saxl Condition)\label{BSC}
    Let $G$ be a NIP group and $\{H_i\}_{i\in I}$ a family of uniformly definable subgroups of $G$. Then, there exists $n<\omega$ such that, for any intersection of $m$ elements $H_{j_1}\cap\cdot\cdot\cdot H_{j_m}$ from $\{H_i\}_{i\in I}$, there exist $i_1,...,i_n\in I$ such that $\bigcap_{k\leq m} H_{j_k}=\bigcap_{k\leq n} H_{i_k} $.
\end{lemma}
From Lemma \ref{boundedind} and Lemma \ref{BSC}, we can derive the icc (descending chain condition on the intersection of uniformly definable subgroups) for groups definable in a NIP finite-dimensional theory.
\begin{lemma}\label{icc}
    Any definable intersection of a family of uniformly definable subgroups of a NIP finite-dimensional group is equal to an intersection of finitely many of them.
\end{lemma}
\begin{proof}
    Let $\{H_i\}_{i\in I}$ be a uniformly definable family of subgroups in $G$, a NIP definable group of finite dimension. By Lemma \ref{BSC}, there exists $n<\omega$ such that, for any finite subset $J$ of $I$, $\bigcap_{j\in J} H_j$ is equal to the intersection of at most $n$ elements. Take $A=\bigcap_{j\leq n} H_{i_j}$ of minimal dimension among all the intersections of $n$ elements from $\{H_i\}_{i\in I}$. Then, for any $\bigcap_{j\in J}H_j$ with $J\subseteq  I$ of cardinality $n$, the intersection with $A$ is of finite index (by minimality of the dimension). By Lemma \ref{boundedind}, we conclude that there exists $B_1=\bigcap_{j\in J}H_j$ such that $A\cap B_1$ is of maximal index in $A$. This, by NIP, is an intersection of $n$ elements in $\{H_i\}_{i\in I}$ that we call $B$. Then $B=B\cap H_i$ and so $B=\bigcap_{i\in I} H_i$. This completes the proof. 
\end{proof}
Another fundamental result for NIP groups is the existence of the type-definable connected component, denoted by $G^{00}$.
\begin{defn}
    Let $G$ be a definable group. we denote by $G^{00}$ the smallest type-definable subgroup of bounded index, if it exists.
\end{defn}
Its existence for NIP groups has been proved in \cite{shelah2008minimal}.
\begin{lemma}\label{G{00}}
    Let $G$ be a group definable in a NIP theory. Then, $G^{00}$ exists.
\end{lemma}
Moreover, since in a definable NIP field $K$, ${K^+}^{00}$ is an ideal of bounded index and so never $0$, $K$ has connected additive group.
\begin{lemma}\label{FieldNIP}
    Let $K$ be a NIP definable field. Then, $(K,+)$ is connected.
\end{lemma}
\subsection{NIP Lie rings of finite dimension}
We characterize NIP Lie rings of small dimension (in particular up to $4$) in the hypothesis that $\mathfrak{g}$ has characteristic different from $2,3$. The general idea is that either the Lie ring is the form we desire, or it is virtually connected. In the latter case, the proofs follow from the analysis of connected Lie rings of finite dimension.\\
We need the following easy Lemma.
\begin{lemma}\label{ObsVc}
    Let $G$ be a group and $H$ be a normal subgroup of $G$. Then,
    \begin{itemize}
        \item if $G/A$ is connected for a finite normal subgroup $A$ in $G$, then $G$ is virtually connected;
        \item if $H$ and $G/H$ are virtually connected, $G$ is virtually connected.
    \end{itemize}
\end{lemma}
\begin{proof}
For every subgroup $H$ of finite index in $G$, $HA/A$ is a subgroup of finite index in $G/A$ and so coincides with $G/A$ \hbox{i.e.} $HA=G$. This implies that $|G:H|\leq |A|$ and so $G$ is virtually connected.\\
For the second, observe that we can assume that $H$ and $G/H$ are connected. Indeed, if $H$ is virtually connected, $H^0$ is of finite index in $H$, and $G/H^0$ is virtually connected by the previous point. Define $G_1/H=(G/H^0)^0$. Then, $G_1$ is a subgroup of finite index in $G$. Therefore, to prove the Lemma, it is sufficient to verify that $G_1$ is connected. Let $G_2$ be a definable subgroup of finite index in $G_1$. Then, $G_2\cap H^0$ is of finite index in $H^0$. Being $H^0$ connected, $G_2\geq H^0$. This implies that $G_2/H^0$ is a definable subgroup of finite index in $G_1/H^0$ and so it coincides with $G_1$. This shows that $G_1$ is connected.
\end{proof}
The following Lemma is an easy application of the icc condition. 
\begin{lemma}\label{C(h)}
    Let $\mathfrak{g}$ be a definable NIP finite-dimensional Lie ring. Then, there exists a definable Lie subring $\mathfrak{h}$ of finite index in $\mathfrak{g}$ such that $\widetilde{Z}(\mathfrak{g})\cap \mathfrak{h}=Z(\mathfrak{h})=\widetilde{Z}(\mathfrak{h})$.
\end{lemma}
\begin{proof}
Let $\mathfrak{h}$ be the intersection of all the centralisers of elements in $\widetilde{Z}(\mathfrak{g})$ \hbox{i.e.} $\mathfrak{h}=C_{\mathfrak{g}}(\widetilde{Z}(\mathfrak{g}))$. This is clearly an ideal being $\widetilde{Z}(\mathfrak{g})$ an ideal.\\
By Lemma \ref{icc}, the intersection of the uniformly definable family of subgroups $\{C_{\mathfrak{g}}(z)\}_{z\in \widetilde{Z}({\mathfrak{g}})}$ coincides with the intersection of finitely many of them. Since all of them are subgroups of finite index, $\mathfrak{h}$ is of finite index. Finally $h\in \mathfrak{h}$ is in $\widetilde{Z}(\mathfrak{g})$ iff $h\in Z(\mathfrak{h})$ by definition.
\end{proof}
Let $\mathfrak{g}$ be a definable Lie ring. We denote by $\mathfrak{g}^{00}$ the type-connected component of $(\mathfrak{g},+)$, if it exists. We verify that, when $\mathfrak{g}^{00}$ exists, it is an ideal of bounded index. In particular, this holds for Lie rings definable in NIP theories.
\begin{lemma}\label{g^00}
    Let $\mathfrak{g}$ be a definable Lie ring. Then, if $\mathfrak{g}^{00}$ exists, it is an ideal of bounded index of $\mathfrak{g}$.
\end{lemma}
\begin{proof}
    To verify that $\mathfrak{g}^{00}$ is an ideal, it is sufficient to prove that for any type-definable subgroup $H$ of bounded index and $g\in \mathfrak{g}$, there exists a type-definable subgroup of bounded index $H_1$ such that $[g,H_1]\leq H$. Assume that $H$ is a subgroup of bounded index type-definable over a set of parameters $M$. Then, the subgroup 
    $$H_1:=\{g'\in \mathfrak{g}:\ [g,g']\in H\}$$
    is a type-definable subset over $M\cup\{g\}$ of bounded index. This completes the proof.
\end{proof}
In the NIP context, we have the following version of Lemma \ref{LinDim1}. 
\begin{lemma}\label{LieNip}
    Let $(\mathfrak{g},V)$ be a $\mathfrak{g}$-module definable in a finite-dimensional NIP theory. Assume that $\dim(V)=1$ and $Z(\mathfrak{g})\cap \widetilde{C}_{\mathfrak{g}}(v)$ is not of finite index in $Z(\mathfrak{g})$. Then, there exists a finite subgroup $V_0$ in $V$ and a definable field $K$ of dimension $1$ such that $V/V_0\simeq K^+$ and $\mathfrak{g}/\widetilde{C}_{\mathfrak{g}}(V)$ is isomorphic to $K^+$. This implies, by Lemma \ref{virtconn}, that $V$ and $(\mathfrak{g}/\widetilde{C}_{\mathfrak{g}}(V),+)$ are virtually connected.\\
    Moreover, for every $k\not=0\ (mod\ p)$, there exists $g\in \mathfrak{g}$ such that $g\cdot a=ka$ for a subgroup of finite index in $A$.
\end{lemma}
\begin{proof}
    The first part follows easily by Lemma \ref{FieldNIP} and \ref{LinNVC}.\\
    Since $\mathfrak{g}/C_{\mathfrak{g}}(V)=K^+$, there exists an element acting as the multiplication by $k$ on $V/V_0$ \hbox{i.e.} $g\cdot v=kv+V_0$ for any $v\in V$. Then, the map
    $$g\cdot\_-k\cdot \_ : V\to V_0$$
    has kernel of finite index. This completes the proof.
\end{proof}
\subsubsection{Dimension 1}
We prove the virtual abelianity of definably minimal NIP Lie rings of finite dimension and characteristic $p\not=2,3$.
\begin{theorem}\label{NIPdim1}
    Let $\mathfrak{g}$ be a NIP finite-dimensional definably minimal Lie ring of characteristic $p\not=2,3$. Then, $\mathfrak{g}$ is virtually abelian.
\end{theorem}
\begin{proof}
    If $\widetilde{Z}(\mathfrak{g})$ is infinite, being definable by Lemma \ref{DefZ}, it must be a Lie subring of finite index. By Lemma \ref{C(h)}, there exists $\mathfrak{h}$ of finite index such that $\widetilde{Z}(\mathfrak{g})\cap \mathfrak{h}=Z(\mathfrak{h})$. $Z(\mathfrak{h})$ is an abelian definable Lie ring of finite index.\\
    Therefore, we may assume $\widetilde{Z}(\mathfrak{g})$ finite and, by Corollary \ref{g/Z(g)}, $\mathfrak{g}/\widetilde{Z}(\mathfrak{g})$ has trivial almost center. Up to work here, we may assume $\widetilde{Z}(\mathfrak{g})=0$. By definable minimality of $\mathfrak{g}$, any $g\in \mathfrak{g}$ has finite centraliser and $[g,\mathfrak{g}]$ is of finite index in $\mathfrak{g}$.
    By Lemma \ref{boundedind}, this family of uniformly definable subgroups has index bounded by a certain $n<\omega$. $\mathfrak{g}^{00}$ is an ideal of bounded index in $\mathfrak{g}$ by Lemma \ref{g^00}. Therefore, for any $g$ in $\mathfrak{g}^{00}$, the group homomorphism 
    $$\ad_g: \mathfrak{g}\to\mathfrak{g}$$ 
    has image contained in $g^{00}$. Since the kernel is finite, the image is a definable subgroup of finite index. This implies that $\mathfrak{g}^{00}$ is definable of finite index \hbox{i.e.} $\mathfrak{g}$ is virtually connected. Taking the connected component and proceeding as in the proof of Lemma \ref{Dim1}, we obtain the conclusion.
\end{proof}
\subsubsection{Dimension 2}
We characterize definable NIP Lie rings of dimension $2$.
\begin{lemma}\label{Dim2NIP}
    Let $\mathfrak{g}$ be a definable NIP Lie ring of dimension $2$. Then $\mathfrak{g}$ is virtually soluble. If $\mathfrak{g}$ is not virtually nilpotent, $\mathfrak{g}$ is virtually connected, and we can apply Lemma \ref{Dim2}.
\end{lemma}
\begin{proof}
     If $\mathfrak{g}$ is definably minimal, it is virtually abelian by Lemma \ref{NIPdim1} and the proof follows. Therefore, we may assume that $\mathfrak{g}$ has a definable Lie subring $\mathfrak{b}$ of dimension $1$. This must be virtually abelian by Lemma \ref{NIPdim1} and consequently $\mathfrak{g}$ contains a definable abelian Lie subring $\mathfrak{h}$ of dimension $1$. Take the action of $\mathfrak{h}$ on $(\mathfrak{g}/\mathfrak{h},+)$. If this action is almost trivial, then there exists a Lie sub ring $\mathfrak{h}_1$ of finite index in $\mathfrak{h}$ (given by the almost centraliser of the action in $\mathfrak{g}$) such that for every  $h\in \mathfrak{h}_1$, $[h,\mathfrak{g}]$ is almost contained in $\mathfrak{h}$. Since $\mathfrak{h}_1$ is of finite index in $\mathfrak{h}$, $[h,\mathfrak{g}]\apprle \mathfrak{h}\iff [h,\mathfrak{g}]\apprle \mathfrak{h}_1$ since $\mathfrak{h}_1\sim\mathfrak{h}$. Therefore,
     $$\{C_{\mathfrak{g}}(h/\mathfrak{h}_1):=\{g\in \mathfrak{g}: [g,h]\in \mathfrak{h}_1\}\}_{h\in \mathfrak{h}_1}$$
     is a family of uniformly definable subgroups of finite index in $\mathfrak{g}$. By icc, 
     $$\mathfrak{g}_1=N_{\mathfrak{g}}(\mathfrak{h}_1)=\bigcap_{h\in\mathfrak{h}_1}C_{\mathfrak{g}}(h/\mathfrak{h}_1)$$
     is of finite index in $\mathfrak{g}$. This is clearly a Lie subring and $\mathfrak{h}_1$ is an ideal in $\mathfrak{g}_1$. Therefore, by Lemma \ref{NIPdim1}, $\mathfrak{g}_1/\mathfrak{h}_1$ and $\mathfrak{h}_1$ are abelian. This implies that $\mathfrak{g}$ is virtually soluble.\\
    If the action of $\mathfrak{h}$ on $\mathfrak{g/h}$ is not almost trivial, $\mathfrak{h}$ and $\mathfrak{g/h}$ are virtually connected by Lemma \ref{LieNip}. Therefore, $\mathfrak{g}$ is virtually connected by Lemma \ref{virtconn} and the proof follows as in Theorem \ref{Dim2}.\\
    For the second part, let $\mathfrak{h}$ be a definable ideal in $\mathfrak{g}$ of dimension $1$. Then, $\mathfrak{h}$ is virtually abelian by Lemma \ref{NIPdim1} and $\widetilde{Z}(\mathfrak{h})$ is of finite index and an ideal in $\mathfrak{g}$. By icc, $Z(\widetilde{Z}(\mathfrak{h}))=\bigcap_{i=1}^n C_{\widetilde{Z}(\mathfrak{h})}(h_i)$ and so it is an abelian ideal of $\mathfrak{g}$ of finite index in $\mathfrak{g}$. Assume $\mathfrak{h}=Z(\widetilde{Z}(\mathfrak{g}))$ then the action of $\mathfrak{g}/\mathfrak{h}$ on $(\mathfrak{h},+)$ can be almost trivial or not. In the first case, we obtain, proceeding as before, that $\mathfrak{g}$ is virtually nilpotent.\\
    If the action is not trivial, then $\mathfrak{h}$ is virtually connected as $\mathfrak{g}/\mathfrak{h}$. This implies that $\mathfrak{g}$ is virtually connected by Lemma \ref{ObsVc} and we can proceed as in \ref{Dim2}.
\end{proof}
\subsection{Absolute simplicity and virtual connectivity}
In this subsection, we verify that a NIP definable absolutely simple Lie ring of finite dimension is virtually connected. This result will be applied to the characterization of NIP Lie ring of dimension $3$ and $4$. We will use some results from section 5.\\
For the following Lemma, in which we prove that the virtually connectivity follows by the existence of an infinite definable connected subgroup, we need a technical definition.
\begin{defn}
    Let $\overline{g}\in \mathfrak{g}^{<\omega}:=\{(g_1,...,g_n):\ \forall i<n\ g_i\in \mathfrak{g}\wedge n<\omega\}$ and $H$ a subgroup of $\mathfrak{g}$. We define $[\overline{g},H]$ the subgroup $[g_1,[g_2,[...,[g_n,H]],...,]$.
\end{defn}
\begin{lemma}\label{ConNip}
    Let $\mathfrak{g}$ be an absolutely simple Lie ring of finite dimension. If $\mathfrak{g}$ has a definable infinite connected subgroup, $\mathfrak{g}$ is virtually connected.
\end{lemma}
\begin{proof}
    Let $H$ be a definable infinite connected subgroup. Then, $[\overline{g},H]$ is connected for every $\overline{g}\in \mathfrak{g}^{<\omega}$. The subgroup $B=\sum_{\overline{g}\in \mathfrak{g}^{<\omega}} [\overline{g},H]$ is definable and connected. Since $[g,B]\leq B$, $B$ is an ideal in $\mathfrak{g}$. Since $\mathfrak{g}$ is absolutely simple, $B$ is of finite index in $\mathfrak{g}$ and so $\mathfrak{g}$ is virtually connected.
\end{proof}
We divide the proof into two parts: in the first, we verify that a soluble non-nilpotent definable NIP Lie ring of finite dimension has an infinite connected subgroup. This, by Lemma \ref{ConNip}, implies that not bad NIP Lie rings are virtually connected. The second part is devoted to the analysis of bad NIP Lie rings.
\begin{defn}
A \emph{Borel} is a definable virtually soluble Lie subring $\mathfrak{b}\leq \mathfrak{g}$ not almost contained in any definable virtually soluble Lie subring of strictly greater dimension. A Lie ring $\mathfrak{g}$ \emph{bad} if it has a virtually nilpotent Borel, not bad otherwise.
\end{defn}
Given a definable NIP Lie ring of finite dimension, there always exists a Borel. Indeed, taken $\mathfrak{h}$ a definably minimal Lie ring, $\mathfrak{h}$ is virtually abelian by Lemma \ref{NIPdim1}. Any virtually soluble Lie subring of maximal dimension is a Borel.\\
We verify that a soluble non-nilpotent NIP Lie ring of finite dimension has an infinite connected definable subgroup.
\begin{lemma}\label{VirSolImpVirCon}
Let $\mathfrak{h}$ be a definable NIP Lie ring of finite index. Assume that $\mathfrak{h}$ is virtually soluble non-virtually nilpotent. Then, $\mathfrak{h}$ has an infinite connected definable subgroup.
\end{lemma}
\begin{proof}
Let $\mathfrak{h}$ be a soluble non-nilpotent NIP finite-dimensional Lie ring. We prove the theorem by induction on the dimension of $\mathfrak{h}$. For $\dim(\mathfrak{h})=2$, the conclusion follows by Lemma \ref{Dim2NIP}.\\
We may assume that the center of $Z(\mathfrak{h})$ is infinite. Let $\mathfrak{h}^{(n)}$ be the last term of the abelian descending series, that is definable by Lemma \ref{DerivedSeries}. Let $J$ be a definable ideal of minimal dimension in $\mathfrak{h}^{(n)}$. We verify that $J$ must be central. Assume not and let $\mathfrak{h}^{(i)}$ be the last term of the abelian descending series such that $\mathfrak{h}^{(i)}/\widetilde{C}_{\mathfrak{h}}(J)$ is not finite. Then, $\mathfrak{h}^{(i)}/C_{\mathfrak{h}}(J)$ is an almost abelian Lie ring that acts not almost trivially on $J$. By icc, there exists an abelian ideal $I_1$ in $\mathfrak{h}^{(i)}/\widetilde{C}_{\mathfrak{h}}(J)$ that acts not almost trivially. Let $J_1$ be a minimal $I_1$-invariant subgroup of $J$. Then, the action cannot be almost trivial. If not, by minimality, the subgroup $\widetilde{C}_{J}(I_1)$ must be of finite index in $J$ \hbox{i.e.} $I_1$ acts almost trivially, a contradiction. Therefore, the action on $J_1$ of $I_1$ is minimal and not almost trivial and, by Lemma \ref{LieNip}, $J_1$ must be virtually connected, proving the Lemma. Therefore, $J$ is almost central in $\mathfrak{h}$ and, up to pass to a Lie subring of finite index, $J\leq Z(\mathfrak{h})$. $\mathfrak{h}/Z(\mathfrak{h})$ is soluble not nilpotent, if not $\mathfrak{h}$ would be nilpotent, a contradiction. Therefore, by induction, it has a connected definable infinite subgroup $A/Z(\mathfrak{h})$.  
 Take, for any $h\in \mathfrak{h}$, the group homomorphism 
$$\ad_h:A/Z(\mathfrak{h})\mapsto \mathfrak{h}.$$
Since $A/Z(\mathfrak{h})$ is connected, the image is connected. If it is infinite, the conclusion follows. Therefore, we may assume that $[h,A]=0$ for any $h\in \mathfrak{h}$. This implies that $A\leq Z(\mathfrak{h})$, a contradiction.
\end{proof}
A corollary of Lemma \ref{ConNip} and Lemma \ref{VirSolImpVirCon} is that not bad absolutely simple Lie rings are virtually connected.
\begin{corollary}\label{VirConNotBadLie}
    Let $\mathfrak{g}$ be a not bad NIP Lie ring of finite dimension. Then, $\mathfrak{g}$ is virtually connected.
\end{corollary}
We verify that also NIP bad Lie rings are virtually connected. This yields to the following conclusion.
\begin{lemma}\label{abssimimpvirtcon}
   Let $\mathfrak{g}$ be an absolutely simple NIP Lie ring of finite dimension. Then, $\mathfrak{g}$ is virtually connected.
\end{lemma}
\begin{proof}
   By Corollary \ref{VirConNotBadLie}, it is sufficient to prove the bad case. Therefore, assume that $\mathfrak{g}$ is a bad Lie ring \hbox{i.e.} it has a nilpotent Borel $\mathfrak{h}$. Take the action of $\mathfrak{h}$ on $\mathfrak{g}/\mathfrak{h}$. This action cannot be almost trivial, since it would imply, by icc, that $\mathfrak{h}$ is an ideal of a definable Lie subring of finite index in $\mathfrak{g}$, contradicting the absolute simplicity. Let $A/\mathfrak{h}$ be a subgroup that is of minimal dimension among all the definable subgroups $A_i/\mathfrak{h}$ that are invariant for a definable Lie subring of finite index in $\mathfrak{h}$. Let $\mathfrak{h}_1$ the Lie ring such that $A/\mathfrak{h}$ is $\mathfrak{h}_1$-invariant. The action cannot be almost trivial. Assume not, then there exists a definable Lie subring $\mathfrak{h}_2$ of finite index in $\mathfrak{h}$ such that $N=N_{\mathfrak{g}}(\mathfrak{h}_2)/\mathfrak{h}_2$ is infinite. Taking a definably minimal Lie subring $\mathfrak{l}$ in $N/\mathfrak{h}_2$, then $\mathfrak{l}$ is a virtually soluble definable Lie ring containing $\mathfrak{h}_2$ and of dimension strictly greater, a contradiction to the hypothesis. By Corollary \ref{LinAlmNil}, $A/\mathfrak{h}$ is virtually connected and isogenic to a vector space of finite dimension over a NIP field. We verify that $A\leq C_{\mathfrak{g}}(Z_{\mathfrak{h}_1})$. For any $h\in Z(\mathfrak{h})$,  we define the group homomorphisms 
   $$\ad_h:A/\mathfrak{h}\to A.$$
   Since $A/\mathfrak{h}$ is virtually connected, also the image is virtually connected. If the image is infinite, $\mathfrak{g}$ is virtually connected by Lemma \ref{ConNip}. Therefore, we may assume that, for every $h\in Z(\mathfrak{h})$, $[h,A]$ is finite. By icc, the subgroup $C_{\mathfrak{g}}(Z(\mathfrak{h}))=\bigcap_{h\in Z(\mathfrak{h})}(C_{\mathfrak{g}}(h))$ is definable and given by the intersection of finitely many centraliser. By previous observation, this implies that $A$ is almost contained in $C_{\mathfrak{g}}(Z(\mathfrak{h}))$. Up to take a submodule of finite index in $A$, we may assume that $A\leq C_{\mathfrak{g}}(Z(\mathfrak{h}))$. Let $h\in Z_2(\mathfrak{h})$. Then, the group homomorphism 
   $$\ad_h:A/\mathfrak{h}\to A/Z(\mathfrak{h})$$
   is well-defined and has virtually connected image.
   For any $h'\in \mathfrak{h}$, also the image of the group homomorphism 
   $$[h',[h,\_]]:A/\mathfrak{h}\to A$$
   is well-defined and virtually connected. By Lemma \ref{ConNip}, either it is finite or $\mathfrak{g}$ is virtually connected. Up to pass to a submodule of finite index, we may assume that, for any $h\in Z_2(\mathfrak{h})$ and $h'\in \mathfrak{h}$, $[h',[h,A]]=0$. Since $[h',h]\in Z(\mathfrak{h})\leq C_{\mathfrak{g}}(A)$, we have that $[h,[h',A]]=0$ for any $h\in Z_2(\mathfrak{h})$ and $h'\in \mathfrak{h}$. This implies that $[\mathfrak{h},A]\leq C_{\mathfrak{g}}(Z_2(\mathfrak{h}))$. The latter is a $\mathfrak{h}$-module. Indeed, for $h\in \mathfrak{h}$, $g\in C_{\mathfrak{g}}(Z_2(\mathfrak{h}))$,
   $$[[h,g],Z_2(\mathfrak{h})]\leq [h,[g,Z_2(\mathfrak{h})]]+[g,[h,Z_2(\mathfrak{h})]\leq 0+[g,Z(\mathfrak{h})]=0.$$
   On the other hand, $C_{\mathfrak{g}}(Z_2(\mathfrak{h}))+\mathfrak{h}$ is strictly greater than $\mathfrak{h}$, if not $[\mathfrak{h},A]\leq \mathfrak{h}$ contradicting the non triviality of the action of $\mathfrak{h}$ on $A/\mathfrak{h}$. Therefore, $A/\mathfrak{h}=C_{A}(Z_2(\mathfrak{h}))+\mathfrak{h}/\mathfrak{h}$. Taken $h\in Z_2(\mathfrak{h})-Z(\mathfrak{h})$, the subgroup $[h,A]\leq [h,C_{\mathfrak{g}}(Z_2(\mathfrak{h}))+\mathfrak{h}]\leq Z(\mathfrak{h})$. Moreover, it is connected. This implies that the image, up to pass to a Lie subring of finite index, must be $0$. This implies $[h,\mathfrak{h}]=0$ \hbox{i.e.} $Z_2(\mathfrak{h})=Z(\mathfrak{h})$. This contradicts the nilpotency of $\mathfrak{h}$ and the conclusion follows.


\end{proof}
Another result in this direction is the following.
\begin{lemma}\label{ConnFinAlmCen}
    Let $\mathfrak{g}$ be a definable soluble NIP Lie ring with finite almost center. Then, $\mathfrak{g}$ is virtually connected and has the DCC.
\end{lemma}
\begin{proof}
We proceed by induction on the dimension. In dimension $2$, the conclusion follows by Lemma \ref{Dim2NIP}. Assume that the Lemma holds for every NIP definable soluble Lie ring of dimension $n$ and finite almost center, and let $\mathfrak{g}$ be a NIP soluble Lie ring of dimension $n+1$ with finite almost center. Up to work in $\mathfrak{g}/\widetilde{Z}(\mathfrak{g})$, we may assume $\widetilde{Z}(\mathfrak{g})=0$. By finite dimension, there exists a definable Lie subring of finite index $\mathfrak{g}_1$ and be a definable $\mathfrak{g}_1$-minimal ideal $\mathfrak{h}$ that is minimal also for any definable Lie subring of finite index in $\mathfrak{g}_1$. Without loss of generality, we may assume $\mathfrak{g}=\mathfrak{g}_1$. The action of $\mathfrak{g}$ on this ideal cannot be almost trivial, since it would imply that the almost center is infinite. Since $\mathfrak{g}$ is soluble and any $\mathfrak{g}^{(n)}$ is definable for any $n\in \mathbb{N}$ by Lemma \ref{Definabilityder}, there must exists $\mathfrak{g}^{(i)}$ such that $\mathfrak{g}^{(i+1)}$ is almost contained in $\widetilde{C}_{\mathfrak{g}}(\mathfrak{h})$ and $\mathfrak{g}^{(i)}$ is not. Therefore, we may apply Lemma \ref{LieLieRin} to linearize the action. By Lemma \ref{LieNip} and icc, $\mathfrak{h}$ is virtually connected and $\mathfrak{h}^0\simeq K^+$ for a definable NIP field $K$. By Lemma \ref{DCCfield}, it has DCC. Assume $\mathfrak{h}$ connected, then $\widetilde{C}_{\mathfrak{g}}(\mathfrak{h})=C_{\mathfrak{g}}(\mathfrak{h})$ and, by previous proof, $\mathfrak{g}/C_{\mathfrak{g}}(\mathfrak{h})$ acts by multiplication for an element in $K$. 
$\mathfrak{g}/\mathfrak{h}$ is still a soluble NIP definable Lie ring of finite dimension. Since $\dim(\mathfrak{g}/\mathfrak{h})<\dim(\mathfrak{g})$, if $\widetilde{Z}(\mathfrak{g}/\mathfrak{h})$ is finite, then $\mathfrak{g}/\mathfrak{h}$ is virtually connected by induction hypothesis. In this case, the conclusion follows from Lemma \ref{virtconn}. Therefore, we may assume that $\widetilde{Z}(\mathfrak{g}/\mathfrak{h})$ is infinite. Passing to a definable Lie subring of finite index in $\mathfrak{g}$, $Z(\mathfrak{g}/\mathfrak{h})$ is infinite. Define $\mathfrak{h}_1$ such that $\mathfrak{h}_1/\mathfrak{h}$ is a definably minimal subgroup in $Z(\mathfrak{g}/\mathfrak{h})$. We prove that $\mathfrak{h}_1$ has the DCC. Since $\mathfrak{h}_1/\mathfrak{h}$ is definably minimal, it is sufficient to prove that $\mathfrak{h}_1/\mathfrak{h}$ is virtually connected. Assume, for a contradiction, that this is not the case. For $g\in C_{\mathfrak{g}}(\mathfrak{h})$, the homomorphism
$$\ad_g:\mathfrak{h}_1/\mathfrak{h}\to \mathfrak{h}$$
is well-defined. It must be trivial since, if not, $\mathfrak{h}_1/\mathfrak{h}$ is isogenous to a subgroup of $\mathfrak{h}$ that has the DCC. Consequently, $\mathfrak{h}_1/\mathfrak{h}$ is virtually connected. Therefore, we may assume that $C_{\mathfrak{g}}(\mathfrak{h})\leq C_{\mathfrak{g}}(\mathfrak{h}_1)$.\\
Let $g\in \mathfrak{g}-C_{\mathfrak{g}}(\mathfrak{h})$. Then, the kernel of the group homomorphism
$$\ad_g:\mathfrak{h}_1\to \mathfrak{h}$$
is infinite by dimensionality. The kernel cannot be equal to $\mathfrak{h}_1$ by assumption on $\mathfrak{g}$. If the kernel is contained in $\mathfrak{h}$, then $\mathfrak{h}_1/\mathfrak{h}$ is virtually connected. Therefore, we may assume that there exists $g\in \mathfrak{g}-C_{\mathfrak{g}}(\mathfrak{h})$ such that $\dim(\mathfrak{h}_1\cap C_{\mathfrak{g}}(g))<\dim(\mathfrak{h}_1)$ and $\mathfrak{h}_1\cap C_{\mathfrak{g}}(g)+\mathfrak{h}\sim \mathfrak{h}_1$. Let $g'\in \mathfrak{g}-C_{\mathfrak{g}}(\mathfrak{h})$. The group homomorphism
$$\ad_{g'}:\mathfrak{h}_1\cap C_{\mathfrak{g}}(g)\to \mathfrak{h}$$
cannot be trivial since $g\not\in C_{\mathfrak{g}}(\mathfrak{h})$. Moreover, in case $\ad_{g'}$ is not trivial the kernel is not almost contained in $\mathfrak{h}\cap \ker([g,\_])$, again $\mathfrak{h}_1$ has the DCC. So we may assume $\mathfrak{h}_1\cap \ker([g,\_])\cap \ker([g',\_])$ is infinite with dimension strictly less that $\dim(\mathfrak{h}_1\cap \ker([g,\_]))$ and $\mathfrak{h}_1\cap \ker([g,\_])\cap \ker([g',\_])+\mathfrak{h}\sim\mathfrak{h}_1$. By finite dimensionality, this process must end after finitely many steps. This means that there exists an infinite definable subgroup $H_1$ such that $H_1+\mathfrak{h}\sim \mathfrak{h}_1$ and $H_1/\mathfrak{h}$ is virtually connected. Consequently, $\mathfrak{h}_1$ has the DCC.\\
Therefore, we may assume that $\mathfrak{h}_1$ has the DCC. The definable Lie ring $\mathfrak{g}/\mathfrak{h}_1$ has infinite almost center, if not the conclusion follows by Lemma \ref{ObsVc}. Iterating the previous proof, we conclude that there exists a Lie subring $\mathfrak{h}_2$ of $\mathfrak{g}$ with the DCC such that $\dim(\mathfrak{h}_2)>\dim(\mathfrak{h}_1)$. The iteration must eventually stop by finite dimensionality. Therefore, $\mathfrak{g}$ is connected. The DCC follows by Theorem \ref{DCCseries}.
\end{proof} 
The condition $\widetilde{Z}(\mathfrak{g})$ finite is necessary: let $K$ be a definable NIP field of dimension $1$ and let $\mathfrak{ga}_1(K)$ be a definable soluble Lie ring of dimension $2$. Let $A$ be a definable NIP abelian not virtually connected group. Then, $A\oplus \mathfrak{ga}_1(K)$ with bracket given by $[a+g,a+g']=[g,g']$ is a soluble Lie ring of dimension $3$ that is not virtually connected.\\
We may conclude this section with the characterization of definable NIP Lie rings of dimension $3$ and $4$.
\begin{lemma}
   Let $\mathfrak{g}$ be a definable NIP Lie ring of finite dimension. Then,
   \begin{itemize}
       \item If $\dim(\mathfrak{g})=3$, either $\widetilde{Z}(\mathfrak{g})$ is infinite or $\mathfrak{g}$ is virtually connected;
       \item if $\dim(\mathfrak{g})=4$, either $\widetilde{Z}(\mathfrak{g})$ is infinite or $\mathfrak{g}$ is virtually connected.
   \end{itemize}
\end{lemma}
\begin{proof}
   By Lemma \ref{ConnFinAlmCen} and Lemma \ref{abssimimpvirtcon}, the first statement follows. For the second one, again by Lemma \ref{ConnFinAlmCen} and Lemma \ref{abssimimpvirtcon}, it sufficient to study the following two cases:
   \begin{itemize}
       \item $\mathfrak{g}$ has a definable ideal $I$ such that $\dim(\mathfrak{g}/I)=3$ and $\mathfrak{g}/I$ is absolutely simple;
       \item $\mathfrak{g}$ has a definable ideal $I$ such that $\dim(I)=3$ and $I$ is absolutely simple.
   \end{itemize} 
   In the first case, up to pass to a definable Lie subring of finite index, we may assume $I$ abelian and $\mathfrak{g}/I$ is connected with the DCC. Taking the action of a definably minimal connected Lie subring $\mathfrak{h}$ of $\mathfrak{g}/I$ on $I$, either it is almost trivial or $I$ and $\mathfrak{h}$ are connected. In the latter, $\mathfrak{g}$ is virtually connected by Lemma \ref{virtconn}. Assume the first, then $\widetilde{C}_{\mathfrak{g}}(I)\geq \mathfrak{h}$. Therefore, $\widetilde{C}_{\mathfrak{g}}(I)$ is an ideal of dimension at least $2$. Since $\widetilde{C}_{\mathfrak{g}}(I)/I$ is an ideal in $\mathfrak{g}/I$, it must be of finite index by absolute simplicity. Consequently, there exists a Lie subring of finite index $\mathfrak{g}_1$ in $\mathfrak{g}$ such that $I\leq Z(\mathfrak{g}_1)\leq \widetilde{Z}(\mathfrak{g})$ \hbox{i.e.} $\widetilde{Z}(\mathfrak{g})$ is infinite.\\
   Assume that $I$ is an absolutely simple ideal in $\mathfrak{g}$. By Lemma \ref{abssimimpvirtcon}, $I$ is virtually connected. Moreover, by Lemma \ref{DCC}, it has the DCC. Passing to a Lie subring of finite index, we may assume that $\mathfrak{g}/I$ is abelian. Then, for every $g\in \mathfrak{g}$, the group homomorphism 
   $$\ad_g:\mathfrak{g}\mapsto I$$
   is well-defined. Since the almost center can be assumed equal to $0$, the kernel cannot be all $\mathfrak{g}$. Moreover, the kernel is infinite by dimensionality. Assume that the kernel is contained in $I$, then $\mathfrak{g}/I$ would be virtually connected. This implies that $\mathfrak{g}$ would be virtually connected by Lemma \ref{virtconn}. Therefore, we may assume that $\ker([g,_])+I=\mathfrak{g}$. Define, for $g'\in \mathfrak{g}$, the group homomorphism 
   $$\ad_{g'}:\ker(\ad_g)\mapsto I.$$
   This homomorphism cannot have finite image for any $g\in \mathfrak{g}$, since it would imply that $\ker(\ad_g)$ is contained in the almost center of $\mathfrak{g}$, which is infinite. Again, the kernel cannot be contained in $I$, if not $\mathfrak{g}/I$ would be virtually connected. By dimensionality, $\ker([g,_])\cap \ker([g',\_])+I=\mathfrak{g}$. Iterating the process, we deduce that any finite intersection of kernels is infinite. By icc, $Z(\mathfrak{g})=\bigcap_{g\in \mathfrak{g}} \ker([g,\_])$ is the intersection of a finite number of subgroups. This implies that $Z(\mathfrak{g})$ is infinite, a contradiction.
\end{proof}

\subsection{Stable Lie rings of finite dimension}
In this subsection we verify that stable fields of finite dimension are algebraically closed. This is used to refine the characterization of stable Lie rings of finite dimension.\\
Stability is a model theoretic condition introduced, again, by Shelah stronger than NIP. The study of stable theories and in particular stable groups immediately capture the interest of model theorist, who have seen in stability the natural extension of finite Morley rank. In 2006, Sela showed that free groups are stable \cite{sela2013diophantine}, mining the possibility of a geometric approach to stable groups. The definition of stable theories is the following.
\begin{defn}
    For a given infinite cardinal $\kappa $, $T$ is \emph{$\kappa$-stable} if for every set $A$ of cardinality $\kappa$ in a model of $T$, the set $S(A)$ of complete types over $A$ also has cardinality $\kappa$. $T$ is \emph{stable} if it is $\kappa$-stable for some infinite cardinal $\kappa$.
\end{defn}
A very important notion for the analysis of stable theories is forking.
\begin{defn}
\begin{itemize}
    \item A formula $\phi(x,a)$ \emph{divides} over $B$ if there is a sequence $(a_i)_{i<\omega}$ and $k\in \omega$ such that $tp(a_i/B)=tp(a/B)$ for every $i<\omega$ and $\{\phi(x,a_i)\}_{i<\omega}$ is $\kappa$-inconsistent.
\item A formula $\phi(x,a)$ \emph{forks} over $B$ if it implies a finite disjunction of formulas dividing over $B$.
\end{itemize} 
\end{defn}
For stable theories, we can define by induction a rank called the Lascar rank, denoted by $U$.
\begin{defn}
    Let $p$ be a type in $S(A)$ for $A\leq \mathfrak{M}$ with $\mathfrak{M}$ a model of $T$ stable. Then, we define $U(p)\geq \alpha$, for $\alpha$ an ordinal, by induction:
    \begin{itemize}
        \item $U(p)\geq 0$ iff $p$ is finitely satisfable;
        \item $U(p)\geq \alpha+1$ if there exists $B\supseteq A$ and $q\in S(B)$ such that $q$ forks over $A$, $q$ extends $p$ and $U(q)\geq \alpha$;
        \item $U(p)\geq \alpha$ for $\alpha$ a limit ordinal if $U(p)\geq \lambda$ for any $\lambda<\alpha$.
    \end{itemize}
\end{defn}
One of the most important results in stability theory is that a superstable field is algebraically closed, as proved by Shelah and Cherlin in \cite{cherlin1980superstable}. Substantially the same proof verifies that a stable field of finite dimension is algebraically closed.
\begin{lemma}
   Let $K$ be a stable field of finite dimension. Then, $K$ is algebraically closed.
\end{lemma}
\begin{proof}
   We recall that a finite-dimensional NIP field is perfect and $K^+$ is connected. By Theorem 35 and Lemma 34 of \cite{cherlin1980superstable}, also $K^{\times}$ is connected. This implies that $K$ has no Artin-Schreier extensions nor Kummer extensions. Assume, for a contradiction, that there exists a stable not algebraically closed finite-dimensional field $K$. Being perfect, it has an algebraic Galois extension. Let $(F,K)$ be a couple of stable fields of finite dimension such that $F$ is a proper algebraic extension of $K$ and assume that $|F:K|$ is minimal among these couples. If we prove that this extension splits and $q$ is prime, we reach a contradiction by Lemma 19 of \cite{cherlin1980superstable}. Assume $q$ is not prime and let $r$ be a proper divisor of $q$. Then, the fixed field $F_1$ of an element of order $r$ in $\operatorname{Gal}(F/K)$ is a proper subfield of $F$ containing $K$, contradicting the minimality. The index of the splitting field of $x^q-1$ divides $q-1$ and, by minimality, it must be equal to $F$. This completes the proof.
\end{proof}
From the conclusions of Lemma \ref{NIPdim1}, \ref{Dim2NIP}, and \ref{abssimimpvirtcon}, we may assume we are working with connected Lie subrings. Then, by Lemma \ref{DCC}, they have the DCC. Applying the same proof of \cite{deloro2023simple}, we derive the following theorem.
\begin{theorem}
  Let $\mathfrak{g}$ be a stable Lie ring of finite dimension. Then, 
  \begin{itemize}
      \item If $\dim(\mathfrak{g})=1$, $\mathfrak{g}$ is virtually abelian;
      \item If $\dim(\mathfrak{g})=2$, $\mathfrak{g}$ is virtually soluble. If it is not virtually nilpotent, it is virtually connected and the connected component is isomorphic to $\mathfrak{ga}_1(K)$ for an algebraically closed field $K$;
      \item If $\dim(\mathfrak{g})=3$ and $\mathfrak{g}$ is not virtually soluble, it is virtually connected and the connected component is isomorphic to $\operatorname{sl}_2(K)$ for an algebraically closed field $K$;
      \item If $\dim(\mathfrak{g})=4$, $\mathfrak{g}$ is not absolutely simple.
  \end{itemize}
\end{theorem}
\section{Action of a Lie ring}
This section is devoted to the analysis of the action of a Lie ring over a definable module. The idea is to extend the corresponding result for Lie rings of finite Morley rank (see \cite{deloro2025soluble}). We will need a preliminary result on the action of an abelian Lie ring on a definable module. This characterization is fundamental in the analysis of the action of a Lie ring with an abelian definable ideal. 
We also characterize the action of an almost abelian Lie ring.
\subsection{Action of a Lie ring with an abelian ideal}
We prove that the action on an abelian Lie ring on a minimal (not necessarily absolutely minimal) module can be linearized. This is very similar to \cite[Theorem 12.3]{InvittiDim}.
\begin{lemma}\label{MinLinAbe}
Let $(\mathfrak{g},W)$ be a definable $\mathfrak{g}$-module of finite dimension.
   Assume that $\mathfrak{g}$ is abelian, $W$ is $\mathfrak{g}$-minimal, and that the action is not almost trivial. Then, there exists a definable field $K$ such that $W/W_0\simeq K^+$ for a finite $\mathfrak{g}$-submodule $W_0$ and $\mathfrak{g}/\mathfrak{g}_0$ embeds in $K^+\cdot Id$ for a definable ideal $\mathfrak{g}_0$ of infinite index in $\mathfrak{g}$. This implies that there are no proper $\mathfrak{g}$-submodules in $W/W_0$.
\end{lemma}
\begin{proof}
    We work in the module $(\mathfrak{g}/\widetilde{C}_{\mathfrak{g}}(W),W/\widetilde{C}_{\mathfrak{g}}(W))$. This action is well-defined since, by abelianity, if $g\in C_{\mathfrak{g}}(W)$, then $\mathrm{im}(g)$ is a finite $\mathfrak{g}$-invariant subgroup, and therefore contained in $\widetilde{C}_W{\mathfrak{g}}$. Let $g\in \mathfrak{g}$, then both $\operatorname{ker}(g)$ and $\mathrm{im}(g)$ are $\mathfrak{g}$-invariant. By minimality, they are either finite or of finite index in $W$. In the first case, $g$ acts as the $0$-endomorphism since any finite $\mathfrak{g}$-invariant subgroup is contained in $\widetilde{C}_{W}(\mathfrak{g})=\{0\}$. Therefore, all the non-trivial endomorphisms given by elements of $\mathfrak{g}$ are monomorphisms. The same holds for all the endomorphisms generated by $\mathfrak{g}$. Then, proceeding as in \cite{InvittiDim}, we have that $\langle \mathfrak{g}\rangle$ acts by automorphisms and $W/\widetilde{C}_{W}(\mathfrak{g})\simeq K^+$ for $K$ the field of fractions of $\langle \mathfrak{g}\rangle$.\\
  For the second point, any $\mathfrak{g}$-submodule $W_1$ is a $K$-subvector space. Since $\operatorname{lindim}(W/W_0)=1$, either $W_1=W_0$ or $W_1=W$.
\end{proof}
We verify that, given a definable module $(\mathfrak{g},A)$ such that $A$ is minimal for any definable ideal of finite index in $\mathfrak{g}$, then $A$ is absolutely $\mathfrak{g}$-minimal.
\begin{lemma}\label{LinAlmAbe2}
    Let $\mathfrak{g}$ be a Lie ring acting on $A$, both definable in a finite-dimensional theory. Assume that $A$ is minimal for any definable ideal of finite index in $\mathfrak{g}$. Then, $A$ is absolutely $\mathfrak{g}$-minimal.
\end{lemma}
\begin{proof}
 Assume, for a contradiction, that $A$ is not absolutely $\mathfrak{g}$-minimal. Then, there exists an infinite almost $\mathfrak{g}$-invariant subgroup $B\leq A$ of infinite index. By definition, $\{g\in \mathfrak{g}:\ g(B)\apprle B\}=\mathfrak{g}$. We denote by $M$ the subgroup $M=\{m\in A:\ \mathfrak{g}(m)\apprle B\}$. This is a $\mathfrak{g}$-module. Indeed, given $m\in M$ and $g\in \mathfrak{g}$, then $\mathfrak{g}(g(m))=[\mathfrak{g},g](m)+g(\mathfrak{g}(m))\apprle B+g(B)\apprle B$ by hypothesis on $B$. We verify that $M\apprge B$. Let 
$$X=\{(g,b)\in \mathfrak{g}\times B:\ g(b)\in B\}.$$
The dimension of every fiber of the first projection is $\dim(B)$, by previous observation. Therefore, by fibration, $\dim(X)=\dim(\mathfrak{g})+\dim(B)$. Let $\pi_2:X\to B$ be the second projection. Then, proceeding as in Lemma \ref{DCC}, $\dim(M)=\dim(\{b\in B:\ \dim(\pi_2^{-1}(b))=\dim(\mathfrak{g})\})=\dim(B)$ (the latter is a definable subgroup by Lemma \ref{boundedind}). Consequently, $M\apprge B$ and so $M$ is infinite. By minimality of $A$, we may conclude that $M\sim A$. Therefore, for almost all $a\in A$, $\mathfrak{g}(a)\apprle B$. Taking 
$$Y=\{(g,b)\in \mathfrak{g}\times A:\ g(a)\in B\}$$
and proceeding as before, we may conclude that the Lie ring
$$\mathfrak{h}=\{g\in \mathfrak{g}:\ g(A)\apprle B\}$$
is of finite index in $\mathfrak{g}$. Moreover, it is an ideal. Let $g\in \mathfrak{g}$ and $h\in \mathfrak{h}$, then $[g,h](A)=g(h(A))-h(g(A))\apprle g(B)-h(A)\apprle B+B\apprle B$ since $B$ is almost $\mathfrak{g}$-invariant. Consequently, $A$ is $\mathfrak{h}$-invariant. Let $a\in A$ be such that $|\mathfrak{h}(a)+B/B|=N$ is maximal (it is finite by assumption). Let $\{h_1,...,h_N\}$ be representatives of $C_{\mathfrak{h}}(a/B)$ in $\mathfrak{h}$. Since $h(A)\apprle B$ for any $h\in \mathfrak{h}$, the subgroups $C_A(h_i/B)$ are all of finite index in $A$. Let $C=\bigcap_{i=1}^N C_A(h_i/B)$. This is a subgroup of finite index in $A$, and let $F=\{f_1,...,f_M\}$ be a system of representatives of $C$ in $A$. Given $a'\in A$ such that $a'=c+f$ with $f\not=0$, then, $h_i(a')=h_i(c+f)=h_i(c)+h_i(f)=h_i(f)$ for any $i\leq N$. Since $h_i(f)+B\not=h_j(f)+B$ for any $i\not=j$, the image of $a$ is contained in $\{h_i(f_j)+B\}_{i\leq N,j\leq M}$. Taken $c\in C$, $\mathfrak{h}(c)\leq \mathfrak{h}(c-a)+\mathfrak{h}(a)\leq F+B/B$ that is a finite extension of $B$. Therefore, 
$\mathfrak{h}(A)$ is a finite extension of $B$. Moreover, it is a $\mathfrak{h}$-module. By $\mathfrak{h}$-minimality, $\dim(B)=\dim([\mathfrak{h},A])=\dim(A)$, contradicting the assumption on $B$.

\end{proof}
We prove the extension of \cite[Theorem A]{deloro2025soluble} to finite-dimensional theories.
\begin{theorem}\label{LieLieRin}
    Let $\mathfrak{g}$ be a Lie ring acting on $V$, an absolutely minimal $\mathfrak{g}$-module, both definable in a finite-dimensional theory. Assume that $\mathfrak{g}$ has a definable abelian ideal $\mathfrak{a}$ such that the action of $\mathfrak{a}$ on $V$ is not almost trivial. Then, there exists a definable ideal $\mathfrak{g}_1$ in $\mathfrak{g}$ of finite index in $\mathfrak{a}$ such that $\mathfrak{a}/\mathfrak{g}_1\leq Z(\mathfrak{g}/\mathfrak{g}_1)$. Moreover, there exists a definable field $K$ such that $\mathfrak{a}/\mathfrak{g}_1$ embeds in $K\cdot Id$ and $V/V_1$ is a $K$-vector field of finite dimension, with $V_1$ a finite $\mathfrak{g}$-module, on which $\mathfrak{g}/\mathfrak{g}_1$ acts $K$-linearly. 
\end{theorem}
\begin{proof}
    Since $\widetilde{C}_{V}(\mathfrak{a})$ is $\mathfrak{g}$-invariant and the action of $\mathfrak{a}$ is not almost trivial by assumption, $\widetilde{C}_{V}(\mathfrak{a})$ is finite by $\mathfrak{g}$-minimality of $V$. By abelianity, $\widetilde{C}_{\mathfrak{a}}(V)$ acts trivially on $V/\widetilde{C}_V(\mathfrak{a})$. Indeed, for any $a\in \widetilde{C}_{\mathfrak{a}}(V)$, $a(V)$ is a finite $\mathfrak{a}$-invariant group and so contained in $\widetilde{C}_{V}(\mathfrak{a})$. $\widetilde{C}_{\mathfrak{a}}(V)$ is an ideal in $\mathfrak{g}$: let $a\in \widetilde{C}_{\mathfrak{a}}(V)$ and $g\in \mathfrak{g}$, then $[g,a](V)\leq (a(g(V))+(g(a(V))$ that is finite since sum of finite subgroups.\\
   Taking $V=V/\widetilde{C}_V(\mathfrak{a})$ and $\mathfrak{g}/\widetilde{C}_{\mathfrak{a}}(V)$, we may assume that $\mathfrak{a}:=\mathfrak{a}/\widetilde{C}_{\mathfrak{a}}(V)$ is infinite and that $\widetilde{C}_V(\mathfrak{a})=\{0\}$. Consequently, there are no finite $\mathfrak{a}$-invariant subgroups in $V$. In particular, since any image or kernel of an element in $\mathfrak{a}$ is $\mathfrak{a}$-invariant, any $a\in \mathfrak{a}$ cannot have finite image or finite kernel.\\
    We verify that $\mathfrak{a}$ is central in $\mathfrak{g}$. Assume, for a contradiction, that $\mathfrak{a}$ is not central. Let $W$ be a minimal $\mathfrak{a}$-invariant subgroup of $V$. Then, by Lemma \ref{MinLinAbe}, $W\simeq K^+$ for a definable field $K$ and $W$ has no proper definable $\mathfrak{a}$-invariant subgroups.
    \begin{claim}
        In the previous hypothesis, we have that:
        \begin{itemize}
            \item No ideal of $\mathfrak{g}$ contained in $\mathfrak{a}$ almost centralises $W$;
            \item There exists $h\in \mathfrak{g}$ such that $[h,\mathfrak{a}]$ does not almost centralize $W$.
        \end{itemize}
    \end{claim}
    \begin{claimproof}
        Assume, for a contradiction, that $\mathfrak{b}$ is an ideal of $\mathfrak{g}$ contained in $\mathfrak{a}$ that almost centralises $W$. Then, $\widetilde{C}_{V}(\mathfrak{b})$ is an invariant $\mathfrak{g}$-module. Indeed, given $g\in \mathfrak{g}$ and $v\in \widetilde{C}_V(\mathfrak{b})$,
        $$\mathfrak{b}(gv)=[\mathfrak{b},g](v)+g(\mathfrak{b}v)$$
        that is finite by hypothesis. Since $V$ is absolutely $\mathfrak{g}$-minimal without finite $\mathfrak{a}$-submodules, this submodule must be of finite index. Therefore, the action of $\mathfrak{b}$ on $V$ is almost trivial, $\mathfrak{b}$ is finite and $\mathfrak{b}\cap \mathfrak{a}=\{0\}$. This implies that $\mathfrak{b}=0$.\\
        For the second point, suppose, for a contradiction, that $[h,\mathfrak{a}]\leq \widetilde{C}_{\mathfrak{a}}(W)$ for any $h\in \mathfrak{g}$. Then, $[h,\widetilde{C}_{\mathfrak{a}}(W)]\leq [h,\mathfrak{a}]\leq \widetilde{C}_{\mathfrak{a}}(W)$. Therefore, $\widetilde{C}_{\mathfrak{a}}(W)$ is an ideal that almost centralises $W$ and contained in $\mathfrak{a}$. By preceding point, it must be $\{0\}$.
    \end{claimproof}\\
    We define $S_i=\sum_{j=0}^i \binom{i}{j} h^jW$.
    \begin{claim}
        In the previous hypothesis
        \begin{itemize}
            \item Each $S_i$ is an $\mathfrak{a}$-module;
            \item For $w\in W-\{0\}$ and $i\geq 0$, $h^iw\in S_i$ iff $S_i=S_{i+1}$;
            \item There is $k\leq \operatorname{dim}(V)$ such that $\sum_{m\geq 0} h^mW=S_k=\oplus_{i=0}^{k-1} h^iW$.
        \end{itemize}
    \end{claim}
    \begin{claimproof}
        For the first point, observe that
        $$aS_i=\sum_{j=0}^i \binom{i}{j} ah^jW=\sum_{j=0}^i \binom{i}{j} [a,h^j]W+h^jaW\leq W+\sum_{j=0}^i \binom{i}{j} h^jW\leq S_i.$$
        For the second point, define 
        $$W_1=\{w\in W :\ h^iw\in S_i\}.$$
        $W_1$ is a non-trivial, definable and $\mathfrak{a}$-invariant subgroup of $W$. By Lemma \ref{MinLinAbe}, either $W_1=0$ or $W_1=W$. Therefore $h^iW\leq S_i$ and $S_{i+1}=S_i$.\\
        The third point follows from the fact that if the intersection is finite, it is $0$ being a $\mathfrak{a}$-invariant subgroup in $W$.
    \end{claimproof}\\
    Let $k\in \mathbb{N}$ be maximal such that $S_k\not=S_{k-1}$. This exists by finite-dimensionality and the third point of the previous lemma.
    \begin{claim}
    In the previous hypothesis,
    \begin{itemize}
        \item Taking $X=S_k/S_{k-1}$, then $X\simeq W$ as $\mathfrak{a}$-module;
        \item $K=C_{\operatorname{DefEnd}(X)}(\mathfrak{a})$ is a definable infinite field;
        \item $h$ induces an homomorphism $\eta\in \operatorname{DefEnd}(X)$ on $X$ and for any $a\in \mathfrak{a}$, $\eta a=a\eta+ka'$.
    \end{itemize}
    \end{claim}
    \begin{claimproof}
        The homorphism taking $w\in W$ to $h^{k-1}w+S_{k-1}$ is clearly an $\mathfrak{a}$-invariant homorphism. If $h^{k-1}w\in S_{k-1}$, then $S_k=S_{k-1}$ by the second point of Claim $2$, a contradiction.\\
        Therefore, being $W\simeq X$, $X$ is $\mathfrak{a}$-minimal and again isomorphic to the additive group of an infinite field.\\
        For the second point, $K$ acts $\mathfrak{a}$-linearly on $X$ and so any image or kernel is equal to $X$ or $\{0\}$. This implies that $K$ acts on $X$ by automorphisms. Since the inverse of an isomorphism commuting with $\mathfrak{a}$ again commutes with $\mathfrak{a}$, $K$ is an infinite field.\\
        Let $x=h^{k-1}+S_{k-1}\in X$ and $\eta(x)=h^kw+S_{k-1}$. $\eta$ is well defined since, if $h^{k-1}w\in S_{k-1}$, then $w=0$ so $h^kw=0$.\\
        Clearly $h^{k-1}aw=ah^{k-1}w+S_{k-1}$ so $h^{k-1}aw$ represents $a\cdot \eta(x)$ and $\eta(a\cdot x)$ since 
        $$\eta(a\cdot x)=hah^{k-1}w+S_{k-1}=h^kaw=a\cdot \eta(x).$$
        Therefore, $\eta(a\cdot x)=ah^kw+ka'h^{k-1}=a\eta(x)+ka'\cdot x$ and $\eta a=a\eta+ka'$.
     \end{claimproof}
     \begin{claim}
         For $\lambda\in K$, $\delta(\lambda)=[\eta,\lambda]$ is a definable derivation of $K$.
     \end{claim}
     \begin{claimproof}
         For the first point, it is sufficient to prove that $\delta(\lambda)\in K$ \hbox{i.e.} $\delta(\lambda)$ commutes with every element of $\mathfrak{a}$. Given $\lambda\in K$, then 
         $$(\eta\lambda-\lambda\eta)a=\eta\lambda a-\lambda\eta a=\eta a\lambda-\lambda(\eta a)=a\eta\lambda+ka'\lambda-\lambda a\eta-\lambda ka'=a(\eta\lambda-\lambda\eta)$$
     \end{claimproof}\\
     Since a finite-dimensional field has no non-trivial definable derivations by Lemma \ref{DerTri}, $[\eta,\lambda]=0$ for any $\lambda\in C_{\operatorname{DefEnd}}(\mathfrak{a})\geq \mathfrak{a}$. Therefore, $ka'$ acts trivially on $X$ and also $a'$ acts trivially on $X$. Consequently, it acts trivially on $W$. So $[h,\mathfrak{a}]$ acts almost trivially on $W$, a contradiction. This implies that $\mathfrak{a}$ is central in $\mathfrak{g}$.\\
     Therefore, any image or kernel of an element in $\mathfrak{a}$ must be of finite index or finite. The second case is impossible since $\widetilde{C}_V(\mathfrak{a})$ is $0$. Since any non-trivial element of $\mathfrak{a}$ acts as a monomorphism with image of finite index, we conclude that the field of fractions of $\mathfrak{a}$, that is definable by \cite[Proposition 3.6]{wagner2020dimensional}, acts on $V$ by automorphisms. Therefore, $V$ is a vector field of finite dimension over a definable infinite field $K$ and $\mathfrak{g}$ acts $K$-linearly since any element of $K$ commutes with any element in $\mathfrak{g}$. This completes the proof.
\end{proof}

\subsection{Action on an almost abelian Lie ring}
In the following theorem, we linearize the action of an almost abelian Lie ring. This is not only a techincal result, since, working in IP finite dimensional theories, we cannot hope to reduce almost abelian Lie rings to abelian ones. For example, pseudofinite Lie rings of dimension $1$ are almost abelian but not virtually abelian. The following proof is similar to \cite[Theorem 12.4]{InvittiDim}.
\begin{theorem}\label{LinAlmAbe}
    Let $\mathfrak{g}$ be an almost abelian Lie ring acting on the abelian group $A$, both definable in a finite-dimensional theory. Assume that the action is not almost trivial and absolutely minimal and that $\mathfrak{g}'\leq Z(\mathfrak{g})$. Then, $\mathfrak{g}/\widetilde{C}_{\mathfrak{g}}(A)$ acts on $A/\widetilde{C}_A(\mathfrak{g})$ and $\mathfrak{g}/\widetilde{C}_{\mathfrak{g}}(A)$ is abelian. Then, we can linearize the action by Lemma \ref{MinLinAbe}. 
\end{theorem}
\begin{proof}
    Let $\mathcal{C}$ be the family of centralisers in $\mathfrak{g}$ of finite tuples of elements in $\mathfrak{g}$. These are all Lie subrings of finite index by almost abelianity. Let $B$ be a definable subgroup of minimal dimension invariant for some subgroup in $\mathcal{C}$ and let $N_B\in \mathcal{C}$ be a subgroup in $\mathcal{C}$ that stabilises $B$. For any tuple $\overline{g}=(g_1,...,g_n)$ of elements in $\mathfrak{g}$, the subgroup $g_1\cdot\cdot \cdot g_n\cdot B$ is of same dimension as $B$ or finite. Indeed, $g\cdot B$ is a definable subgroup of dimension less than or equal to $B$ normalised by $N_B\cap C_{\mathfrak{g}}(g)\in \mathcal{C}$. By minimality, $gB$ is either finite or of finite index in $B$. We conclude by induction.\\
    Given a sum of translates $\sum_{i=1}^m g_1\cdot\cdot\cdot g_{n_i}(B)$ of maximal dimension, this is an almost $\mathfrak{g}$-invariant subgroup in $A$ and therefore of finite index by absolute minimality. Let $M$ be the intersection of all the centralisers of the tuples $\{g_{1},...,g_{n_i}\}$ with $N_B$. Let $R_M$ be the ring of endomorphisms generated by $M$: it is clearly contained in $\bigcap_{i=1}^m C_{\operatorname{End}(A)}(g_1,...,g_{n_i})$. Therefore, fixed $r\in R_M$, either the image through $r$ of $g_1\cdot\cdot\cdot g_{n_i}(B)$ is finite for every $i$ or it is of finite index for every $g_1\cdot\cdot\cdot g_n\cdot B$. In the first case, $r\in \widetilde{C}_{\mathfrak{g}}(A)$ that is not of finite index in $\mathfrak{g}$. In the second, the image of $r$ is of finite index in $A$ since $rg_1,...g_{n_i}(B)\leq g_1...g_nr(B)$ by hypothesis. Since $A\sim \sum_{i=1}^n g_1...g_{n_i}(B)$, the proof follows. Therefore, there exist unboundedly many elements $\{r_i\}_{i\in I}\subseteq R_M$ such that $r_i-r_j\not\in \widetilde{C}_{\mathfrak{g}}(A)$ for $i\not=j\in I$.\\ 
    We show that there exists a maximal finite subgroup in $A$ invariant for the action of an element in $\mathcal{C}$. Assume not, and let $S=\sum_{i<\omega} S_i$ for $\{S_i\}_{i<\omega}$ an ascending series of finite subgroup invariant for an element of $\mathcal{C}$. This is stabilised by $N\cap M$ with $N=\bigcap_{i<\omega} N_i$ where $N_i\in \mathcal{C}$ stabilises $S_i$. This is clearly a subgroup of bounded index in $\mathfrak{g}$. Then $(N\cap M)+ \widetilde{C}_{\mathfrak{g}}(A)/\widetilde{C}_{\mathfrak{g}}(A)$ is not finite (since $\widetilde{C}_{\mathfrak{g}}(A)$ is not of finite index). Consequently, there exist two elements $r_i,r_j$ that are in $R_{N\cap M}$ such that $(r_i-r_j)(S)=\{0\}$ and $(r_i-r_j)(A)$ is not finite. This contradicts the previous result. In conclusion, there exists a maximal subgroup that necessarily is $\mathfrak{g}$-invariant. Assume not, and take $g\in \mathfrak{g}$ such that $g(S)\not\leq S$. Then, $S+g(S)$ is a finite subgroup strictly containing $S$ and invariant for $N\cap C_{\mathfrak{g}}(g)$, a contradiction. Being $\mathfrak{g}$-invariant and finite, $S\leq \widetilde{C}_A({\mathfrak{g}})$. Consequently, $A/\widetilde{C}_A(\mathfrak{g})$ has no finite subgroups invariant for an element in $\mathcal{C}$. Given $g\in \widetilde{C}_{\mathfrak{g}}(A)$, $g(A)$ is a finite subgroup invariant for $C_{\mathfrak{g}}(g)\in \mathcal{C}$. Therefore, $g(A)$ is contained in $\widetilde{C}_{A}(\mathfrak{g})$ and so $\mathfrak{g}/\widetilde{C}_{\mathfrak{g}}(A)$ acts on $A/\widetilde{C}_{A}(\mathfrak{g})$. Working here, we may assume $\widetilde{C}_{\mathfrak{g}}(A)=0$ and $\widetilde{C}_{A}(\mathfrak{g})=0$. Still $\mathfrak{g'}\leq Z(\mathfrak{g})$ and so it is contained in $M$ (being a finite intersection of centralisers). Given the kernel or the image of any $r\in R_M$ generated by $\{g_1,...,g_m\}$, they are invariant under $C_{\mathfrak{g}}(g_1,...,g_m)\in \mathcal{C}$, and so one of them is finite and the other of finite index. Since $\widetilde{C}_{\mathfrak{g}}(A)=\{0\}$, either $r=0$ or the kernel is $\{0\}$. Therefore, we can construct a definable skew-field $K$, proceeding as in \cite{InvittiDim}, with $\dim(K)=\dim(A)$. We verify that $K$ must be a field. For any $k\in K$, $k$ is generated by elements $\{g_1,...,g_m\}\in M$. Let $N=C_{\mathfrak{g}}(g_1,...,g_n)\cap M$. Then, $B$ is invariant for $N$ and minimal (by minimality of the dimension). All the elements different from $0$ in $R_{M\cap N}$ are monomorphisms with image of finite index in $A$. We can define a skew subfield $K_1$ of $K$ of the same dimension as $\dim(A)=\dim(K)$. Then, $K$ is a vector space over $K_1$ and so $\dim(K)=\operatorname{lindim}(K/K')\cdot \dim(K')$. Since the two dimensions coincide, we may conclude that $\operatorname{lin.dim}(K/K')=1$ and so $K=K'$. On the other hand, $K'$ centralises $k\in K$: $C_{K}(k)$ is a skew subfield and so it contains the skew subfield generated by $M\cap N$. This implies that $k\in Z(K)$ and, by arbitrariety of $k\in K$, $K$ is a field. By linearisation, $M$ acts as $K^+\cdot Id$ on the $K$-vector space of finite dimension $A$. \\
    We verify that $\mathfrak{g}$ acts $K$-linearly on $A$. Indeed $C_{\mathfrak{g}}(g)\cap M$ is of finite index in $M$. Proceeding as before, we obtain that $K$ is equal to the field generated by $R_{C_{\mathfrak{g}}(g)\cap M}$ and so $g$ acts $K$-linearly. Given $g,g'\in \mathfrak{g}$ seen as matrices in $\mathfrak{gl}_n(K)$, $\mathrm{tr}([g,g'])=0$. On the other hand, $[g,g']\in \mathfrak{g}'\leq M$ that acts as a diagonal matrix. Therefore, $[g,g']=0$ and we conclude that $\mathfrak{g}$ is abelian.
\end{proof}
Observe that, in general, given an almost abelian Lie ring $\mathfrak{g}$, we can always find a characteristic Lie subring $\mathfrak{h}$ such that $\mathfrak{h}'\leq Z(\mathfrak{h})$, simply taking $\mathfrak{h}=C_{\mathfrak{g}}(\mathfrak{g}')$. On the other hand, we cannot be sure that $A$ remains absolutely $\mathfrak{h}$-invariant. If we assume that $A$ is $\mathfrak{h}$-minimal for any definable Lie subring $\mathfrak{h}$ of finite index in $\mathfrak{g}$, $A$ is absolutely $C_{\mathfrak{g}}(\mathfrak{g}')$-minimal by Lemma \ref{LinAlmAbe2} and therefore we can apply Theorem \ref{LinAlmAbe}. \\
A corollary of Theorem \ref{LinAlmAbe} and Lemma \ref{LinAlmAbe2} is the characterization of the action of a nilpotent Lie ring in a finite-dimensional theory. This is fundamental in the analysis of the \emph{ almost cohomology groups} in \cite{InvittiCohomology}.
\begin{corollary}\label{LinAlmNil}
    Let $\mathfrak{g}$ be a definable nilpotent Lie ring acting on a definable module $A$. Assume that the action is not almost trivial and that $A$ is $\mathfrak{h}$-minimal for any $\mathfrak{h}$ definable ideal of finite index in $\mathfrak{g}$. Then, there exists a definable ideal $\mathfrak{h}$ of finite index such that $\mathfrak{h}/\widetilde{C}_{\mathfrak{h}}(A)$ is abelian and acts scalarly on the $\mathfrak{h}$-module $A_1/A_2$ with $A_1$ a subgroup of finite index in $A$ and $A_2$ finite.
\end{corollary}
\begin{proof}
By Lemma \ref{LinAlmAbe2}, $A$ is absolutely $\mathfrak{g}$-minimal. Therefore, since the action is not almost trivial, $\widetilde{C}_A(\mathfrak{g})$ is finite. Without loss of generality, we may assume $\widetilde{C}_{A}(\mathfrak{g})=0$.
   Assume that the action of $\widetilde{Z}(\mathfrak{g})$ on $A$ is not almost trivial. Then, $\widetilde{C}_A(\widetilde{Z}(\mathfrak{g}))$ is a $\mathfrak{g}$-invariant subgroup that, by assumption, cannot be of finite index. By hypothesis, it is equal to $0$. Let $B$ be a minimal $Z(\mathfrak{g})$-invariant subgroup and a sum $\sum_{\overline{g}\in \mathfrak{g}^{<\omega}} \overline{g}B$ of maximal dimension. Then, $Z(\mathfrak{g})$ acts on $B$ not almost trivially. Moreover, the almost centraliser $\widetilde{C}_{B}(\widetilde{Z}(\mathfrak{g}))$ is contained in $\widetilde{C}_{A}(\widetilde{Z}(\mathfrak{g}))=0$. Therefore, by Lemma \ref{LinAlmAbe}, $\widetilde{Z}(\mathfrak{g})/\widetilde{C}_{A}(\widetilde{Z}(\mathfrak{g}))$ is abelian. By Lemma \ref{LieLieRin}, $\widetilde{Z}(\mathfrak{g})/\widetilde{C}_{\widetilde{Z}(\mathfrak{g})}(A)$ is central in $\mathfrak{g}/\widetilde{C}_{\mathfrak{g}}(A)$. Moreover, there exists a definable field $K$ such that $\mathfrak{g}/\widetilde{C}_{\mathfrak{g}}(A)$ acts $K$-linearly on the $K$-vector space of finite dimension $A$. Since $\mathfrak{g}$ is nilpotent, then, by Engel's Theorem, there exists a $\mathfrak{g}$-invariant subgroup $A_1$ of $A$ of $K$-dimension $1$. By minimality, we may conclude that $\mathfrak{g}/C_{\mathfrak{g}}(A_1)$ is abelian and acts $K$-scalarly on the $\mathfrak{g}$-module $A_1$ of finite index in $A$.\\
   Assume that $\widetilde{Z}(\mathfrak{g})$ acts almost trivially. Then, $\widetilde{Z}(\mathfrak{g})\cap \widetilde{C}_{\mathfrak{g}}(A)$ is of finite index. We verify that $\widetilde{Z}(\mathfrak{g})\cap \widetilde{C}_{\mathfrak{g}}(A)\leq C_{\mathfrak{g}}(A)$. It is sufficient to prove that, for $g\in \widetilde{Z}(\mathfrak{g})\cap \widetilde{C}_{\mathfrak{g}}(A)$, $g(A)\leq \widetilde{C}_{A}(\mathfrak{g})=0$ by hypothesis. For any $g'\in C_{\mathfrak{g}}(g)$, $g'(g(A))=[g',g](A)+g(g'(A))\leq gA$. Since, by hypothesis, $C_{\mathfrak{g}}(g)$ is of finite index in $\mathfrak{g}$ and $g(A)$ is finite, $g(A)\leq \widetilde{C}_A(\mathfrak{g})$. Therefore, $\widetilde{Z}^2(\mathfrak{g})/C_A(\mathfrak{g})$ is almost abelian. Proceeding as before, we either obtain the linearization or $\widetilde{Z}^2(\mathfrak{g})$ is almost contained in $\widetilde{C}_{\mathfrak{g}}(A)$. Let $g\in \widetilde{C}_{\mathfrak{g}}(A)\cap \widetilde{Z}^2(\mathfrak{g})$. By definition, $[g,\mathfrak{g}]\apprle \widetilde{Z}(\mathfrak{g})\apprle C_{\mathfrak{g}}(A)$. Let $C_g=\{g'\in \mathfrak{g}:\ [g',g]\in C_{\mathfrak{g}}(A)\}$. By previous proof, $C_g$ is a definable subgroup of finite index in $\mathfrak{g}$. For every $g'\in C_g$, $g'(gA)=[g,g'](A)+g(g'(A))=g(A)$ since $[g,g'](A)=0$. Therefore, $g(A)$ is contained in $\widetilde{C}_A(\mathfrak{g})=0$ \hbox{i.e.} $g\in C_{\mathfrak{g}}(A)$. This implies that $\widetilde{Z}^2(\mathfrak{g})$ is almost contained in the centraliser of the action. We iterate the process for any $\widetilde{Z}^n(\mathfrak{g})$ and eventually we conclude.

   \end{proof}
\section{Definable envelopes}
In this section, we prove the existence of (almost) soluble and (almost) nilpotent definable envelopes respectively for (almost) soluble and (almost) nilpotent Lie subrings of an hereditarily $\widetilde{\mathfrak{M}}_c$-Lie ring.
\begin{theorem}
    Let $\mathfrak{g}$ be a definable hereditarily $\widetilde{\mathfrak{M}_c}$-Lie ring and $\mathfrak{h}$ a Lie subring. Then,
\begin{itemize}
    \item if $\mathfrak{h}$ is almost nilpotent, $\mathfrak{h}$ is contained in a definable almost nilpotent Lie ring $\mathfrak{b}$ such that $N_{\mathfrak{g}}(\mathfrak{h})\leq N_{\mathfrak{g}}(\mathfrak{b})$;
    \item if $\mathfrak{h}$ is almost soluble with almost abelian series
    $$\mathfrak{h}=\mathfrak{h}_0\geq \mathfrak{h}_1\geq...\geq \mathfrak{h}_n=0$$
    $\mathfrak{h}$ is contained in a definable almost soluble Lie ring with almost abelian series normalised by $\bigcap_{i\leq n} N_{\mathfrak{g}}(\mathfrak{h}_i)$.
\end{itemize}
\end{theorem}
In particular, the Theorem implies that an almost nilpotent ideal is contained in an almost nilpotent definable ideal and a soluble ideal, having an ideal almost abelian series, is contained in an almost soluble definable ideal with ideal almost abelian series.\\
A consequence of these results is that, in a finite-dimensional definable Lie ring, the Radical and the Fitting ideal are definable and respectively soluble and nilpotent (see Lemma \ref{FitFinDim} and Lemma \ref{RadIdeFinDim}). Moreover, with a proof similar to \cite{hempel}, we verify that the Fitting ideal of an hereditarily $\widetilde{\mathfrak{M}}_c$-Lie ring is definable and nilpotent.
\subsection{Construction of the definable envelopes}
We start from the almost nilpotent case.
\begin{lemma}\label{AlmNilEmb}
    Let $\mathfrak{g}$ be a definable hereditarily $\widetilde{\mathfrak{M}}_c$-Lie ring and $\mathfrak{h}$ an almost nilpotent Lie subring. Then, $\mathfrak{h}$ is contained in a definable almost nilpotent Lie subring $\mathfrak{b}$ of same almost nilpotency class such that $N_{\mathfrak{g}}(\mathfrak{b})\geq N_{\mathfrak{g}}(\mathfrak{h})$.
\end{lemma}
\begin{proof}
    We proceed by induction on the almost nilpotency class of $\mathfrak{h}$. Assume that $\mathfrak{h}$ is almost abelian. Let $C=\bigcap_{i=1}^n C_{\mathfrak{g}}(h_i)$ be a centraliser of minimal $\widetilde{\mathfrak{M}}_c$-dimension among all the finite intersections of centralisers of elements in $\mathfrak{h}$. For minimal $\widetilde{\mathfrak{M}}_c$-dimension we mean that, for every $h\in \mathfrak{h}$, $C_{\mathfrak{g}}(h)\apprge C$. By assumptions, $\mathfrak{h}\apprle \widetilde{C}_{\mathfrak{g}}(C)$. We verify that $C$ almost contains $\mathfrak{h}$ and $[n,\mathfrak{h}]\apprle \mathfrak{h}$ for any $n\in N_{\mathfrak{g}}(\mathfrak{h}$). Being an intersection of finitely many Lie subrings almost containing $\mathfrak{h}$ (by almost abelianity), $C$ almost contains $\mathfrak{h}$. We prove the second statement. Let $n\in N_{\mathfrak{g}}(\mathfrak{h})$, then $[n,h_i]\in \mathfrak{h}$ by definition. Therefore, $C\cap C_{\mathfrak{g}}([n,h_i])$ is of finite index in $C$. To verify that $C$ is almost invariant for $\ad_g$, it is sufficient to show that, for every $i\leq n$, there exists a subgroup $C_i$ of finite index in $C$ such that $[[n,C_i],h_i]=0$. Let $C_i=C\cap C_{\mathfrak{g}}([n,h_i])$, then, by hypothesis, $C_i\sim C$. Moreover, 
    $$[[n,C_i],h_i]\leq [C_i,[n,h_i]]+[[h_i,C_i],n]=0.$$
    This proves that $C$ is almost invariant for $N_{\mathfrak{g}}(\mathfrak{h})$. $\widetilde{C}_{\mathfrak{g}}(C)$ is a Lie subring in $\mathfrak{g}$ such that $N_{\mathfrak{g}}(\widetilde{C}_{\mathfrak{g}}(C))\geq \widetilde{N}_{\mathfrak{g}}(C)\geq N_{\mathfrak{g}}(\mathfrak{h})$ by Lemma \ref{C(H/K)AlmId}. Therefore, also $\mathfrak{b}_1=\widetilde{Z}_{\mathfrak{g}}(\widetilde{C}_{\mathfrak{g}}(C))$ is a Lie ring normalised by $N_{\mathfrak{g}}(\mathfrak{h})$. In addiction, $\mathfrak{b}_1$ is almost abelian. $\mathfrak{b}_1$ almost contains $\mathfrak{h}$. Indeed, $\mathfrak{b}=\widetilde{C}_{\mathfrak{g}}(\widetilde{C}_{\mathfrak{g}}(C))\cap \widetilde{C}_{\mathfrak{g}}(C)$, but the latter contains $\mathfrak{h}$ by previous proof, while the first since $\widetilde{C}_{\mathfrak{g}}(\widetilde{C}_{\mathfrak{g}}(C))\apprge C$ by Lemma \ref{sym}. Finally, let $\mathfrak{b}=\mathfrak{b}_1+\mathfrak{h}$. $\mathfrak{b}$ is definable, being a finite extension of $\mathfrak{b}$, and it is almost nilpotent for the same reason. It is a Lie ring normalised by $N_{\mathfrak{g}}(\mathfrak{h})$ since $\mathfrak{h}\leq N_{\mathfrak{g}}(\mathfrak{h})$ and $N_{\mathfrak{g}}(\mathfrak{h})$ normalises both of them. This completes the proof in the almost abelian case.\\
    By induction hypothesis, assume that the conclusion of the Lemma is true for any Lie ring of almost nilpotency class $n$, and let $\mathfrak{h}$ be an almost nilpotent Lie ring of almost nilpotency class $n+1$. Let $\widetilde{Z}(\mathfrak{h})$ be the almost center of $\mathfrak{h}$ and denote with $C$ the intersection of finitely many centralisers of elements in $\widetilde{Z}(\mathfrak{h})$ of minimal $\widetilde{\mathfrak{M}}_c$-dimension. Then $C$, by definition of almost center, almost contains $\mathfrak{h}$ and its normaliser contains $N_{\mathfrak{g}}(\mathfrak{h})$ since $N_{\mathfrak{g}}(\mathfrak{h})\leq N_{\mathfrak{g}}(\widetilde{Z}(\mathfrak{h}))$. Let $\mathfrak{b}_1:=\widetilde{C}_{\mathfrak{g}}(\widetilde{C}_{\mathfrak{g}}(C))$ and $\mathfrak{b}_2:=\widetilde{Z}(\mathfrak{b}_1)$. We have that $\mathfrak{b}_1\apprge \mathfrak{h}$ and $\mathfrak{b}_2\apprge \widetilde{Z}(\mathfrak{h})$. Indeed, since $\widetilde{C}_{\mathfrak{g}}(C)=\widetilde{C}_{\mathfrak{g}}(C)$, then $\mathfrak{b}_1=\widetilde{C}_{\mathfrak{g}}\widetilde{C}_{\mathfrak{g}}(C)\apprge C\apprge \mathfrak{h}$ by Lemma \ref{sym}. For the second, it is sufficient to prove that $\widetilde{C}_{\mathfrak{g}}(C)$ almost contains $\widetilde{Z}(\mathfrak{h})$ and this follows by construction of $C$. Moreover, $N_{\mathfrak{g}}(\mathfrak{b}_2)\geq N_{\mathfrak{g}}(\widetilde{C}_{\mathfrak{g}}(C))\geq N_{\mathfrak{g}}(\mathfrak{h})$.\\
    We work in $N_{\mathfrak{g}}(\mathfrak{b}_2)/\mathfrak{b}_2$. This contains the Lie ring $\mathfrak{h}+\mathfrak{b}_2/\mathfrak{b}_2$ that has almost nilpotency class $n$. By induction, we obtain $\mathfrak{b}_3$ almost containing $\mathfrak{h}$ such that $\mathfrak{b}_3/\mathfrak{b}_2$ is almost nilpotent and $N_{\mathfrak{g}}(\mathfrak{b}_3)\geq N_{\mathfrak{g}}(\mathfrak{h})$. Let $\mathfrak{b}=\mathfrak{b}_3\cap \mathfrak{b}_1$. $\mathfrak{b}$ is a definable Lie ring such that:
    \begin{itemize}
        \item $\mathfrak{b}$ almost contains $\mathfrak{h}$;
        \item $N_{\mathfrak{g}}(\mathfrak{b})\geq N_{\mathfrak{g}}(\mathfrak{h})$;
        \item $\mathfrak{b}/\mathfrak{b}_2$ is almost nilpotent;
        \item $\mathfrak{b}_2\leq \widetilde{Z}(\mathfrak{b})$.
    \end{itemize}
    Therefore, $\mathfrak{b}$ is a definable almost nilpotent Lie ring that almost contains $\mathfrak{h}$ and such that $N_{\mathfrak{g}}(\mathfrak{h})\leq N_{\mathfrak{g}}(\mathfrak{b})$. Taking $\mathfrak{b}+\mathfrak{h}$, the proof follows.
\end{proof}
An immediate corollary is the following.
\begin{corollary}
    Let $\mathfrak{g}$ be a definable hereditarily $\widetilde{\mathfrak{M}}_c$-Lie ring and $\mathfrak{h}$ an almost nilpotent ideal. Then, $\mathfrak{h}$ is contained in an almost nilpotent definable ideal of the same almost nilpotency class.
\end{corollary}
We prove a similar result for almost soluble Lie subrings. 
\begin{lemma}\label{EnvAlmSol}
   Let $\mathfrak{g}$ be a definable hereditarily $\widetilde{\mathfrak{M}}_c$-Lie ring and $\mathfrak{h}$ an almost soluble Lie subring. Then, $\mathfrak{h}$ is contained in an almost soluble definable Lie subring $\mathfrak{b}$ with same almost soluble class. Moreover, if $\mathfrak{h}=\mathfrak{h}_0\geq \mathfrak{h}_1\geq...\geq \mathfrak{h}_n=\{0\}$ is the almost soluble series of $\mathfrak{h}$, we can construct an almost abelian series for $\mathfrak{b}$
   $$\mathfrak{b}=\mathfrak{b}_0\geq \mathfrak{b}_1\geq...\geq \mathfrak{b}_n=\{0\}$$
   such that $N_{\mathfrak{g}}(\mathfrak{b}_j)\geq \bigcap_{i=1}^jN_{\mathfrak{g}}(\mathfrak{h}_i)$. 
\end{lemma}
\begin{proof}
    By induction on the almost soluble class. If $\mathfrak{h}$ is almost abelian, the proof follows by Lemma \ref{AlmNilEmb}. Therefore, assume that $\mathfrak{h}$ has almost soluble class $n$ and that the conclusion of the theorem holds for Lie subrings of almost soluble class $\leq n+1$. $\mathfrak{h}_1$ has soluble class $n-1$ and, by induction hypothesis, it is contained in an almost soluble ideal $\mathfrak{b}_1$ with almost soluble series with the desired properties. In particular, since $\mathfrak{h}_i$ is an ideal in $\mathfrak{h}$, then $N_{\mathfrak{g}}(\mathfrak{b}_1)\geq \bigcap_{i\leq n}N_{\mathfrak{g}}(\mathfrak{h}_i)\geq \mathfrak{h}$ for any $i\leq n$. Working in $N_{\mathfrak{g}}(\mathfrak{b}_1)/\mathfrak{b}_1$, $\mathfrak{h}+\mathfrak{b}_2/\mathfrak{b}_2$ is almost abelian. By Lemma \ref{AlmNilEmb}, there exists a definable Lie ring $\mathfrak{b}$ such that:
    \begin{itemize}
        \item $\mathfrak{b}$ almost contains $\mathfrak{h}$;
        \item $N_{\mathfrak{g}}(\mathfrak{b}/\mathfrak{b}_2)\geq N_{\mathfrak{g}}(\mathfrak{h}+\mathfrak{b}_2/\mathfrak{b}_2)$;
        \item  $\mathfrak{b}/\mathfrak{b}_2$ is almost abelian.
    \end{itemize}
    Therefore, $\mathfrak{b}$ is definable, almost soluble of almost soluble class as $\mathfrak{h}$ and with almost soluble series with the desired property. The same holds for $\mathfrak{b}+\mathfrak{h}$.
\end{proof}
The following corollary is immediate.
\begin{corollary}
    Let $\mathfrak{g}$ be a definable hereditarily $\widetilde{\mathfrak{M}}_c$-Lie ring and $\mathfrak{h}$ an almost soluble ideal with ideal almost soluble series. Then, $\mathfrak{h}$ is contained in a definable almost soluble ideal of the same soluble class and with ideal almost soluble series. This, in particular, holds for soluble ideals. 
\end{corollary}
Moreover, similarly to the group case in \cite{hempel}, almost nilpotent definable ideals contain a nilpotent definable ideal of finite index, and almost soluble definable ideals (with soluble series given by ideals in $\mathfrak{g}$) contain a definable soluble ideal of finite index.
\begin{lemma}\label{NilFromAlmNil}
    Let $\mathfrak{g}$ be a definable hereditarily $\widetilde{\mathfrak{M}}_c$-Lie ring and $\mathfrak{h}$ an almost nilpotent definable Lie subring. Then, there exists $\mathfrak{b}$ a definable nilpotent Lie subring of finite index in $\mathfrak{h}$ such that $N_{\mathfrak{g}}(\mathfrak{b})\geq N_{\mathfrak{g}}(\mathfrak{h})$.
\end{lemma}
\begin{proof}
    Take the ascending almost central series for $\mathfrak{h}$
    $$\mathfrak{h}=\widetilde{Z}_n(\mathfrak{h})\geq \widetilde{Z}_{n-1}(\mathfrak{h})\geq...\geq \widetilde{Z}(\mathfrak{h})\geq \{0\}.$$
    Clearly all the $n$-almost center are normalized by $N_{\mathfrak{g}}(\mathfrak{h})$ by Lemma \ref{C(H/K)AlmId}. Moreover, the action of $\mathfrak{h}$ on $\widetilde{Z}_i(\mathfrak{h})/\widetilde{Z}_{i-1}(\mathfrak{h})$ is almost trivial and, by Lemma \ref{fin[]}, 
    $$A_i:=[\widetilde{C}_{\mathfrak{h}}(\widetilde{Z}_i(\mathfrak{h})/\widetilde{Z}_{i-1}(\mathfrak{h})),\widetilde{Z}_i(\mathfrak{h})/\widetilde{Z}_{i-1}(\mathfrak{h})]/\widetilde{Z}_{i-1}(\mathfrak{h})$$
    is a finite Lie subring normalized by $N_{\mathfrak{g}}(\mathfrak{h})$. Define $C_i:=\widetilde{C}_{\mathfrak{h}}(\widetilde{Z}_i(\mathfrak{h})/\widetilde{Z}_{i-1}(\mathfrak{h}))$ a Lie subring normalized by $N_{\mathfrak{g}}(\mathfrak{h})$, by Lemma \ref{C(H/K)AlmId}, and of finite index in $\mathfrak{h}$ by Lemma \ref{sym}. Since $A_i$ is a Lie subring normalized by $N_{\mathfrak{g}}(\mathfrak{h})$, $D_i:=C_{\mathfrak{g}}(A_i)$ is an ideal of finite index in $N_{\mathfrak{g}}(\mathfrak{h})$. Defined $\mathfrak{b}:=\bigcap_{i=1}^n C_i\cap D_i$, $\mathfrak{b}$ is a Lie subring of finite index in $\mathfrak{h}$ normalized by $N_{\mathfrak{g}}(\mathfrak{h})$. Moreover, this is nilpotent since $A_i/\widetilde{Z}_{i-1}(\mathfrak{h})\leq C_{\mathfrak{h}}(\mathfrak{b}+\widetilde{Z}_{i-1}(\mathfrak{h})/\widetilde{Z}_{i-1}(\mathfrak{h}))$ and $\widetilde{Z}_i(\mathfrak{h})/A_i\leq C_{\mathfrak{h}}(\mathfrak{b}+A_i/A_i)$. Since $\mathfrak{b}$ is contained in $\mathfrak{h}$, $\mathfrak{b}$ is nilpotent. 
\end{proof}
We can derive the following corollary.
\begin{corollary}\label{EnvNil}
    Let $\mathfrak{g}$ be a definable hereditarily $\widetilde{\mathfrak{M}}_c$-Lie ring and $\mathfrak{h}$ a nilpotent ideal. Then, there exists a definable nilpotent ideal $\mathfrak{b}$ containing $\mathfrak{h}$.
\end{corollary}
\begin{proof}
    By Lemma \ref{AlmNilEmb}, there exists an almost nilpotent definable ideal $\mathfrak{b}$ that contains $\mathfrak{h}$. By Lemma \ref{NilFromAlmNil}, there exists a definable nilpotent ideal $\mathfrak{b}_2$ that almost contains $\mathfrak{h}$. Taken $\mathfrak{h}+\mathfrak{b}_2$, this is a definable nilpotent ideal containing $\mathfrak{h}$.
\end{proof}
For almost soluble Lie rings, we have a similar result.
\begin{lemma}\label{SolFromAlmSol}
    Let $\mathfrak{g}$ be a definable hereditarily $\widetilde{\mathfrak{M}}_c$-Lie subring and $\mathfrak{h}$ a definable almost soluble Lie ring with almost soluble series normalized by $H\geq \mathfrak{h}$. Then, $\mathfrak{h}$ contains a definable Lie subring $\mathfrak{b}$ of finite index in $\mathfrak{h}$, soluble and normalized by $H$.
\end{lemma}
\begin{proof}
    Let $\mathfrak{h}=\mathfrak{h}_0\geq \mathfrak{h}_1\geq ...\geq \mathfrak{h}_n\geq \{0\}$ be an almost abelian series normalised by $H$. By Lemma \ref{fin[]}, $[\mathfrak{h}_{i},\mathfrak{h}_i]$ is almost contained in $\mathfrak{h}_{i+1}$. Let $A_i:=[\mathfrak{h}_{i},\mathfrak{h}_i]$, then $A_i/\mathfrak{h}_i$ is a finite Lie ring in $\mathfrak{g}/\mathfrak{h}_{i+1}$ normalised by $H$. Therefore, $C_i:=C_{\mathfrak{g}}(A_{i}/\mathfrak{h}_{i+1})$ is of finite index in $H\geq \mathfrak{h}$. Define $\mathfrak{b}:=\mathfrak{h}\cap \bigcap_{i=1}^n C_i$. This is clearly a Lie subring of finite index in $\mathfrak{h}$ normalised by $H$. We verify that $\mathfrak{b}^{2m}\leq \mathfrak{h}_n$. By induction, it is clear for $m=0$. Therefore, let $\mathfrak{b}^{2m+2}=[[\mathfrak{b}^{2m},\mathfrak{b}^{2m}],[\mathfrak{b}^{2m},\mathfrak{b}^{2m}]]$ that, by induction, is contained in  $[[\mathfrak{h}_m,\mathfrak{h}_m],\mathfrak{b}]\leq \mathfrak{h}_{m+1}$. This completes the proof.
\end{proof}
Finally, we have the following corollary.
\begin{corollary}\label{EnvSol}
    Let $\mathfrak{g}$ be a definable hereditarily $\widetilde{\mathfrak{M}}_c$-Lie ring and $\mathfrak{h}$ a soluble ideal. Then, $\mathfrak{h}$ is contained in a definable soluble ideal.
\end{corollary}
\begin{proof}
    By Lemma \ref{EnvAlmSol}, $\mathfrak{h}$ is contained in an almost soluble ideal $\mathfrak{b}_1$ with ideal almost soluble series. By Lemma \ref{SolFromAlmSol}, there exists a soluble definable ideal $\mathfrak{b}_2$ of finite index in $\mathfrak{b}_1$. Therefore, $\mathfrak{b}:=\mathfrak{b}_2+\mathfrak{h}$ is a finite extension of $\mathfrak{b}_2$ and so definable and soluble.
\end{proof}
\subsection{Fitting and radical ideal in dimensional theories}
We define the Fitting and the Radical ideal of a Lie ring.
\begin{defn}
    Let $\mathfrak{g}$ be a Lie ring. The \emph{Fitting ideal} $F(\mathfrak{g})$ is the sum of all the nilpotent ideals in $\mathfrak{g}$. The \emph{Radical ideal} is the sum of all the soluble ideals in $\mathfrak{g}$.
\end{defn}
We prove that for every definable Lie ring $\mathfrak{g}$ of finite dimension, $F(\mathfrak{g})$ is nilpotent and definable and that $R(\mathfrak{g})$ is soluble and definable.\\
We start verifying that, in a finite-dimensional context, any locally nilpotent Lie subring is almost nilpotent.
\begin{lemma}\label{AlmNilFinDim}
Let $\mathfrak{g}$ be a definable Lie ring of finite dimension and $\mathfrak{h}$ a locally nilpotent Lie subring. Then, $\mathfrak{h}$ is almost nilpotent. In particular, the Fitting is almost nilpotent.
\end{lemma}
\begin{proof}
If $\mathfrak{h}$ is finite, it is nilpotent. Therefore, we may assume that $\mathfrak{h}$ is infinite.\\
Let $\mathfrak{g}$ and $\mathfrak{h}$ be a counterexample of the Lemma with $\mathfrak{g}$ of minimal dimension. In particular, there is no definable Lie ring $\mathfrak{g}_1$ containing $\mathfrak{h}$ such that $\dim(\mathfrak{g}_1)<\dim(\mathfrak{g})$. $\widetilde{Z}(\mathfrak{g})$ must be finite. Assume not, then $\mathfrak{g}/\widetilde{Z}(\mathfrak{g})$ is of smaller dimension and $\mathfrak{h}+\widetilde{Z}(\mathfrak{g})/\widetilde{Z}(\mathfrak{g})$ is still locally nilpotent. By induction hypothesis, $\mathfrak{h}/\widetilde{Z}(\mathfrak{g})$ is almost nilpotent, and the same holds for $\mathfrak{h}$. Up to work in $\mathfrak{g}/\widetilde{Z}(\mathfrak{g})$, we may assume $\mathfrak{g}$ centerless.\\
 Take $C$ the centraliser of a finite tuple of elements in $\mathfrak{h}$ of minimal dimension. If $\dim(C)=0$, let $C_1$ be a minimal finite centraliser of a finite tuple in $\mathfrak{h}$. By local nilpotency, $C_1\cap C_{\mathfrak{g}}(h)$ is not empty for every $h\in \mathfrak{h}$. Therefore, by minimality, $C_1\leq C_{\mathfrak{g}}(h)$ for all $h\in \mathfrak{h}$ \hbox{i.e.} the definable Lie subring $C_{\mathfrak{g}}(C_1)$ contains $\mathfrak{h}$. By minimality of the dimension, $C_{\mathfrak{g}}(C_1)$ is of finite index in $\mathfrak{g}$. Consequently, $ \widetilde{Z}(G)\geq C_1\not=0$, a contradiction.\\
    Therefore, we may assume that $C$ is infinite. Then, $C\cap C_{\mathfrak{g}}(h)$ has same dimension as $C$ for any $h\in\mathfrak{h}$. Then, $C\apprle C_{\mathfrak{g}}(h)$ for any $h\in \mathfrak{h}$. Consequently, $\mathfrak{h}\leq \widetilde{C}_{\mathfrak{g}}(C)$. This is clearly a definable Lie subring and, as before, it must be of finite index in $\mathfrak{g}$. Therefore, $C_{\mathfrak{g}}(C)\apprge \mathfrak{g}$ and so, by Lemma \ref{sym}, $C\apprle \widetilde{Z}(\mathfrak{g})$. Being $C$ infinite, we obtain a contradiction.
\end{proof}
From the almost nilpotency, we deduce the nilpotency of the Fitting ideal.
\begin{theorem}\label{FitFinDim}
    Let $\mathfrak{g}$ be a definable Lie ring of finite dimension. Then, $F(\mathfrak{g})$ is nilpotent and definable. 
\end{theorem}
\begin{proof}
    $F(\mathfrak{g})$ is almost nilpotent by Lemma \ref{AlmNilFinDim}. By Lemma \ref{AlmNilEmb}, there exists a definable ideal $\mathfrak{h}$ containing $\mathfrak{g}$ and almost nilpotent. By Lemma \ref{NilFromAlmNil}, there exists a definable nilpotent ideal $\mathfrak{b}$ that almost contains $F(\mathfrak{g})$. Therefore, $F(\mathfrak{g})$ is a finite extension of a definable ideal and so definable. Moreover, it is nilpotent. Let $\overline{g}$ be a transversal for $\mathfrak{b}$ in $F(\mathfrak{g})$. For every finite tuple $\overline{g}$ in $F(\mathfrak{g})$, the ideal generated by $\overline{g}$ in $\mathfrak{h}$ is still nilpotent. This implies that $\mathfrak{b}+\mathfrak{b}_1$, with $\mathfrak{b}_1$ the ideal generated by $\overline{g}$, is nilpotent and contains $F(\mathfrak{g})$. Therefore, $F(\mathfrak{g})$ is nilpotent. 
\end{proof}
We prove that the Radical ideal is soluble and definable.
\begin{theorem}\label{RadIdeFinDim}
    Let $\mathfrak{g}$ be a definable Lie ring of finite dimension. Then, $R(\mathfrak{g})$ is soluble and definable. 
\end{theorem}
\begin{proof}
    Let $\mathfrak{b}$ be a definable soluble ideal in $\mathfrak{g}$ of maximal dimension. Then, every soluble ideal $\mathfrak{b}_1$ in $\mathfrak{g}$ is almost contained in a definable soluble ideal $\mathfrak{b}_2$. Since $\mathfrak{b}+\mathfrak{b}_2$ is still a soluble definable ideal, then $\mathfrak{b}_2\apprle \mathfrak{b}$. Consequently, every soluble ideal in $\mathfrak{g}$ is almost contained in $\mathfrak{b}$. We work in $\mathfrak{g}/\mathfrak{b}$. Since $\mathfrak{b}$ is soluble, given a soluble ideal $\mathfrak{h}/\mathfrak{b}$ in $\mathfrak{g}/\mathfrak{b}$, then $\mathfrak{h}$ is soluble. By maximality of the dimension, any soluble ideal in $\mathfrak{h}/\mathfrak{b}$ is finite and so contained in the almost center \hbox{i.e.} $R(\mathfrak{g})\leq \widetilde{Z}(\mathfrak{g}/\mathfrak{b})$. Let $\mathfrak{s}=\widetilde{Z}(\mathfrak{g}/\mathfrak{b}_1)$. $\mathfrak{s}$ is an almost soluble ideal in $\mathfrak{g}$ with ideal almost series. This implies, by Lemma \ref{EnvSol}, that $\mathfrak{s}$ is almost contained in a definable soluble ideal $\mathfrak{c}$. By construction, $\dim(\mathfrak{c})>\dim(\mathfrak{b})$, contradicting the maximality of $\mathfrak{b}$. Therefore, $\widetilde{Z}(\mathfrak{g}/\mathfrak{b})$ contains $R(\mathfrak{g})$ and it is a finite extension of $\mathfrak{b}$. Therefore, it is definable. Proceeding as in Lemma \ref{FitFinDim}, the conclusion follows.
\end{proof}
We prove that, given an absolutely simple Lie ring $\mathfrak{g}$, then $\mathfrak{g}/R(\mathfrak{g})$ is semi-simple in the following sense.
\begin{defn}
    Let $\mathfrak{g}$ be a Lie ring. $\mathfrak{g}$ is \emph{semi-simple} if it is not abelian and it has no proper abelian ideals.
\end{defn}
\begin{corollary}
    Let $\mathfrak{g}$ be an absolutely simple Lie ring. Then, $\mathfrak{g}/\widetilde{Z}(\mathfrak{g})$ is absolutely simple and semi-simple.
\end{corollary}
\begin{proof}
By absolute simplicity, $\widetilde{Z}(\mathfrak{g})$ is finite. Therefore, by Lemma \ref{g/Z(g)}, $\mathfrak{g}/\widetilde{Z}(\mathfrak{g})$ has $0$ almost center.
 Assume for a contradiction that $\mathfrak{g}/\widetilde{Z}(\mathfrak{g})$ is not semi-simple. Then, $R(\mathfrak{g}/\widetilde{Z}(\mathfrak{g}))$ is not $0$. Being a definable soluble ideal in $\mathfrak{g}$, by Lemma \ref{RadIdeFinDim}, $R(\mathfrak{g}/\widetilde{Z}(\mathfrak{g}))$ must be of finite index in $\mathfrak{g}$. Moreover, $R(\mathfrak{g}/\widetilde{Z}(\mathfrak{g}))^n$ is still a definable ideal for any $n$ and so either $0$ or of finite index. Take the biggest $m$ such that $R(\mathfrak{g}/\widetilde{Z}(\mathfrak{g}))^m$ is of finite index. Then, $R(\mathfrak{g}/\widetilde{Z}(\mathfrak{g}))^{m+1}=0$ and so $\mathfrak{g}$ is virtually abelian, a contradiction to absolute simplicity. The absolute simplicity of $\mathfrak{g}/\widetilde{Z}(\mathfrak{g})$ is clear.
\end{proof}
\subsection{Fitting in hereditarily $\widetilde{\mathfrak{M}}_c$-Lie rings}
We prove that the Fitting ideal of a hereditarily $\widetilde{\mathfrak{M}}_c$-Lie ring is nilpotent. This proof follows the steps of the proof of the nilpotency of the Fitting for $\widetilde{\mathfrak{M}}_c$-groups, which can be found in \cite{hempel}. 
\begin{lemma}\label{sol}
Let $\mathfrak{g}$ be a definable hereditarily $\widetilde{\mathfrak{M}}_c$-Lie ring. Then, $F(\mathfrak{g})$ is soluble. 
\end{lemma}
\begin{proof}
We will denote $F(\mathfrak{g})$ by $I$.\\
    We take $g_1,...,g_n\in \mathfrak{g}$ with $n<\omega$ maximal such that $C_{\mathfrak{g}}(g_1,...,g_i)$ is of infinite index in $C_{\mathfrak{g}}(g_1,...,g_{i-1})$ for any $i\geq 1$. We prove by inverse induction on $i$ that $C_I(g_1,...,g_i)$ is soluble. For $i=0$, it follows from the fact that $C_I(g_1,...,g_n)$ is a definable almost abelian Lie ring. Therefore, by Lemma \ref{fin[]}, the derived ideal is finite and nilpotent, since $I:=F(\mathfrak{g})$ is locally nilpotent. Therefore $C_I(g_1,...,g_n)$ is soluble.\\
    By induction hypothesis, for any $g\in G$ such that $C_{\mathfrak{g}}(g)$ is of infinite index in $C_{\mathfrak{g}}(g_1,...,g_{i-1})$, $C_{I}(g_1,...,g_i,g)$ is soluble. Let $H=C_I(g_1,...,g_{i-1})$ and, replacing $\mathfrak{g}$ with $C_\mathfrak{g}(g_1,...,g_{i-1})$, we may assume that $C_I(g)$ is soluble for any $g\in G-\widetilde{Z}(G)$. $H$ cannot be contained in the almost center of $G$ (if not the proof follows as before). Therefore, we can find a nilpotent Lie subring $H_0$ such that $H_0/\widetilde{Z}(\mathfrak{g})$ is not trivial. Therefore, $C_{H_0}(H_0/\widetilde{Z}(\mathfrak{g}))$ strictly contains $\widetilde{Z}(\mathfrak{g})\cap H_0$. Take $h_1$ in the difference and repeat the process. By the $\widetilde{\mathfrak{M}}_c$-condition, we may stop after finitely many steps. So we may find $h\in I-\widetilde{Z}(\mathfrak{g})$ such that $C_I(h/\widetilde{Z}(\mathfrak{g}))$ has index at most $d$. Since $h $ is not in the almost center, the centraliser $C_{\mathfrak{g}}(h)$ is soluble.\\
    Define $N=[\widetilde{Z}(\mathfrak{g}),\widetilde{C}_{\mathfrak{g}}(\widetilde{Z}(\mathfrak{g}))]$ a finite ideal in $\mathfrak{g}$ and $\mathfrak{g}_1=C_{\mathfrak{g}}(N)$. This is clearly an ideal in $\mathfrak{g}$ of finite index. We can assume that $I\leq \mathfrak{g}_1$ since $I/\mathfrak{g_1}\cap I$ is finite and so nilpotent.\\
   Define the map 
    $$[h,\_]: x\in C_I(h/N)\mapsto [h,x]\in N.$$
    The kernel is $C_I(h)$ and an ideal since $N$ is central. Since it is soluble and $N$ is nilpotent, $C_I(h/N)$ is soluble. Take the map
    $$[h,\_]+N:x\in C_I(h/\widetilde{C}_I(\mathfrak{g}))\mapsto [h,x]+N\in \widetilde{C}_I(\mathfrak{g})/N.$$
    The kernel is $C_I(h/N)$ that is soluble and an ideal since $\widetilde{C}_I(\mathfrak{g})/N$ is central by assumption. Therefore $\mathfrak{s}:=C_I(h/\widetilde{C}_I(\mathfrak{g}))$ is soluble and of finite index in $H$. Take $a_1,...,a_n$ a left transversal of $\mathfrak{s}$ in $F(\mathfrak{g})$ and let $J$ be the nilpotent ideal generated by $a_1,...,a_n$. Then, $J+\mathfrak{s}$ contains $H$ and is soluble. Therefore, also $H$ is soluble. By induction, $F(\mathfrak{g})$ is soluble. 
\end{proof}
\begin{lemma}\label{nilAct}
    Let $L$ and $A$ be quotients of definable Lie subrings of $\mathfrak{g}$ such that $A$ is abelian and $L$ acts on $A$ by the adjoint action. Suppose there exists $H$ abelian Lie subring of $L$ such that:
    \begin{itemize}
        \item There are elements $\{h_i\}_{i<l}$ in $H$ and natural numbers $(m_i:i<l)$ such that $ad_{h_i}^{m_i}(A)\text{ is finite}$;
        \item For any $h\in H$, the index of $C_A(h_1,...,h_{l-1},h)$ is finite in $C_A(h_1,...,h_{l-1})$.
    \end{itemize}
    Then, there is a definable Lie sub ring $K$ of $L$ which almost contains $H$ and a natural number $m$ such that $\widetilde{C}^m_A(K)$ has finite index in $A$.
\end{lemma}
\begin{proof}
    Let $\overline{h}$ be the tuple given by the statement. As $\mathfrak{g}$ is a hereditarily $\widetilde{\mathfrak{M}}_c$-Lie ring, there is a natural number $d$ such that every descending chain of centralisers of index greater than $d$ has length at most $n$. Hence the group
    $$K:=\{g\in C_L(\overline{h}):\ |C_A(\overline{h}):C_A(g,\overline{h})| \text{ is finite}\}$$
 is a definable Lie ring containing $H$ (by abelianity). Let $m$ be equal to $1+\sum_{i=1}^l (m_i-1)$. Fix an arbitrary tuple $\overline{n}=(n_0,,...,n_{m-1})\in l^{\times m}$. By the pigeon-hole principle, there exists at least one $i\in l$ such that there are $m$ repetitions in the tuple. By the assumption on $H$, 
    $$[h_{n_0},[h_{n_1},...[h_{n_{m-1}},A]...]$$
    is finite. This implies that $C_A(h_{n_0})\apprge [h_{n_1},...[h_{n_{m-1}},A]...]$. This holds for every $h_i$ and so $[h_{n_1},...[h_{n_{m-1}},A]...]\apprle C_A(\overline{h})$. Since for any $k\in K$, $C_{\mathfrak{g}}(k)\apprge C_{\mathfrak{g}}(\overline{h})$, 
    $$[h_{n_1},...[h_{n_{m-1}},A]...]\apprle C_A(k)$$
    and so $[k,[h_{n_1},...[h_{n_{m-1}},A]...]$ is finite. By commutation, and repeating the previous process, we have that, for any $k_1,...,k_{n_{m-1}}$, $[k_{n_1},...[k_{n_{m-1}},A]...]$ is finite. Therefore $K\leq \widetilde{C}_{\mathfrak{g}}([k_{n_2},...[k_{n_{m-1}},A]...])$ and so, by symmetry, $\widetilde{Z}(K)\apprge k_{n_2},...[k_{n_{m-1}},A]...]$. Iterating again the proof, we obtain that $A\leq\widetilde{C}_A^m(K)$.
\end{proof}
Finally, we can verify that the Fitting ideal is nilpotent also in hereditarily $\widetilde{\mathfrak{M}}_c$-Lie rings.
\begin{theorem}
    Let $\mathfrak{g}$ be an hereditarily $\widetilde{\mathfrak{M}_c}$-Lie ring. Then, $F(\mathfrak{g})$ is nilpotent. 
\end{theorem}
\begin{proof}
    The Fitting ideal is soluble by Lemma \ref{sol}. So there exists $m<\omega$ such that $F(\mathfrak{g})^{(m)}=\{0\}$. We will prove by induction on $n$ that $F(\mathfrak{g})^{(n)}$ is nilpotent. Observe that, proceeding as in Theorem \ref{FitFinDim}, is sufficient to verify that $F(\mathfrak{g})^{(n)}$ is almost nilpotent.\\
    Assume, by induction, that $F(\mathfrak{g})^{(n+1)}$ is nilpotent. By Lemma \ref{EnvNil}, there exists an ideal $N$ that is nilpotent and contains $F(\mathfrak{g})^{n+1}$. Since it is nilpotent, there exists a chain $\{N_i\}_{i\leq n}$ of ideals in $\mathfrak{g}$ such that $[N,N_i]\leq N_{i+1}$. Take the action of $\mathfrak{g}/N$ on $N_i/N_{i+1}$ by $[\_,\_]$ (this is well-defined since $N$ acts trivially on $N_i/N_{i+1}$). For any $g\in F(\mathfrak{g})^{n}$, there exists a nilpotent ideal $H_g$ that contains $g$. Since $H_g+N$ is again nilpotent, there exists $m_g$ such that $\operatorname{ad}^{m_g}_g(N)=0$. 
    By Lemma \ref{nilAct}, there exists a definable Lie subring $K_i$ containing $F(\mathfrak{g})^n$ and such that $\widetilde{C}_{\mathfrak{g}}(K_i/N_{i+1})\apprge N_i$. Defined $K=\bigcap_{i=1}^n K_i$, $K$ still almost contains $F(\mathfrak{g})^n$ and $N_i\apprle \widetilde{C}_{\mathfrak{g}}(K/N_{i+1})$ for every $i$. Finally, $F(\mathfrak{g})^n/N$ is an abelian ideal in $\mathfrak{g}/N$ and so it is contained in an almost abelian ideal $H$. Taking $\mathfrak{b}=K\cap H$, this is a definable Lie ring that almost contains $F(\mathfrak{g})^n$. It is almost nilpotent since $\widetilde{Z}_{\sum_{i=1}^n m_i}(\mathfrak{b})$ almost contains $N\cap \mathfrak{b}$ by previous proof and $\mathfrak{b}/\widetilde{Z}_{\sum_{i=1}^n m_i}(\mathfrak{b})$ is almost abelian by definition of $H$. This implies that $\mathfrak{b}$ is almost nilpotent, and the same holds for $F(\mathfrak{g})^n$. This completes the proof.
\end{proof}
\subsection{Almost radical and almost Fitting}
We can also define the almost version of the Fitting.
\begin{defn}
    Let $\mathfrak{g}$ be an hereditarily $\widetilde{\mathfrak{M}}_c$-Lie ring. The \emph{almost Fitting} of $\mathfrak{g}$, denoted by $\widetilde{F}(\mathfrak{g})$, is the sum of all the almost nilpotent ideals.
\end{defn}
This is clearly well-defined. Observe that it is locally almost nilpotent \hbox{i.e.} any Lie ring generated by finitely many elements is almost nilpotent. Indeed, the sum of two almost nilpotent ideals in a hereditarily $\widetilde{\mathfrak{M}}_c$-Lie ring is still almost nilpotent. 
\begin{lemma}
    Let $\mathfrak{g}$ be an hereditarily $\widetilde{\mathfrak{M}}_c$-Lie ring and $\mathfrak{h},\mathfrak{h}_1$ two almost nilpotent ideals. Then, $\mathfrak{h}+\mathfrak{h}_1$ is almost nilpotent. 
\end{lemma}
\begin{proof}
    By Lemma \ref{AlmNilEmb}, both $\mathfrak{h}$ and $\mathfrak{h}_1$ are respectively contained in $\mathfrak{b}$ and $\mathfrak{b}_1$ almost nilpotent definable ideals. These are virtually nilpotent by Lemma \ref{NilFromAlmNil}. Let $\mathfrak{c},\mathfrak{c}_1$ be the two nilpotent ideals of finite index in $\mathfrak{b}$ and $\mathfrak{b}_1$ respectively. Then, $\mathfrak{c}+\mathfrak{c}_1$ is of finite index in $\mathfrak{b}+\mathfrak{b}_1$ and nilpotent. This implies that $\mathfrak{h}+\mathfrak{h}_1$ is virtually nilpotent and so almost nilpotent.
\end{proof}
We also define the almost radical ideal of $\mathfrak{g}$.
\begin{defn}
    Let $\mathfrak{g}$ be a Lie ring. The \emph{almost Radical} ideal, denoted by $\widetilde{R}(\mathfrak{g})$, is the sum of all almost soluble ideals with ideal almost abelian series.
\end{defn}
We have the following result.
\begin{lemma}\label{AlmRadFinRad}
    Let $\mathfrak{g}$ be a finite-dimensional definable Lie ring. Then, $\widetilde{R}(\mathfrak{g})$ is a finite extension of $R(\mathfrak{g)}$ and so definable and virtually soluble.
\end{lemma}
\begin{proof}
By Lemma \ref{EnvAlmSol} and \ref{SolFromAlmSol}, any almost soluble ideal $\mathfrak{h}$ with ideal almost soluble series is almost contained in $R(\mathfrak{g})$. Therefore $\widetilde{R}(\mathfrak{g})/R(\mathfrak{g})\leq \widetilde{Z}(\mathfrak{g}/R(\mathfrak{g}))$. The latter is an almost soluble ideal with ideal almost soluble series. This implies that $\widetilde{R}(\mathfrak{g})=\widetilde{Z}(\mathfrak{g}/R(\mathfrak{g}))$ that is almost contained in $R(\mathfrak{g})$. 
\end{proof}
Therefore, $\mathfrak{g}/\widetilde{R}(\mathfrak{g})$ is an absolutely semi-simple Lie ring in the following sense.
\begin{defn}
    A Lie ring $\mathfrak{g}$ is \emph{absolutely semi-simple} if it does not contain any almost abelian ideal (and therefore also almost soluble ones).
\end{defn}
We introduce definably minimal almost ideals.
\begin{defn}
    Let $\mathfrak{g}$ be a definable Lie ring and $\mathfrak{h}$ be a definable almost ideal in $\mathfrak{g}$. Then, $\mathfrak{h}$ is \emph{definably minimal} if it contains no definable infinite almost ideals of infinite index.
\end{defn}
We have the following property for finite-dimensional absolutely semi-simple Lie rings.
\begin{lemma}
    Let $\mathfrak{g}$ be a definable finite-dimensional absolutely semi-simple Lie ring. Then, $\mathfrak{g}$ contains finitely many, up to commensurability, definably minimal almost ideals.
\end{lemma}
\begin{proof}
    Let $\widetilde{\oplus}_{i=1}^n H_i$ be a sum of maximal dimension of definably minimal almost ideals. This sum is almost direct since the intersection of two definable almost ideals is a definable almost ideal. Therefore, given $H_i$ and $\widetilde{\oplus}_{j=1}^{i-1} H_i$, either the intersection is finite or almost contains $H_i$ that can be erased without changing the dimension. Let $H$ be another definably minimal ideal and assume it is not commensurable to any $H_i$ for $i\leq n$. Then, $H\apprle \widetilde{\oplus}_{i=1}^n H_i$ by maximality of the dimension. On the other hand, $ H\cap H_i$ is finite for every $i$ (if not, the two almost invariant subgroups would be equivalent). By Lemma \ref{IntAlmIde}, $H\apprle \widetilde{C}_{\mathfrak{g}}(H_i)$ for every $i$. This implies that $H\apprle \bigcap_{i=1}^n \widetilde{C}_{\mathfrak{g}}(H_i)$. Clearly $\bigcap_{i=1}^n\widetilde{C}_{\mathfrak{g}}(H_i)=\widetilde{C}_{\mathfrak{g}}(\widetilde{\oplus}_{i=1}^n H_i)$. Therefore, $H\apprle \widetilde{C}_{\mathfrak{g}}(\widetilde{\oplus}_{i=1}^n H_i):=C$. By Lemma \ref{C(H/K)AlmId}, $C$ is an ideal in $\mathfrak{g}$. The same holds for $\widetilde{Z}(C)$. If we prove that this is infinite, we have a contradiction. $\widetilde{Z}(C)=\widetilde{C}_{\mathfrak{g}}\widetilde{C}_{\mathfrak{g}}(\widetilde{\oplus}_{i=1}^n H_i)\cap C$. By Lemma \ref{sym}, $\widetilde{C}_{\mathfrak{g}}\widetilde{C}_{\mathfrak{g}}(\widetilde{\oplus}_{i=1}^n H_i)\apprge \widetilde{\oplus}_{i=1}^n H_i\apprge H$. Since $H$ is also contained in $C$, and $H$ is infinite. This completes the proof.
\end{proof}
Given any definably minimal subgroup $A$ in $\mathfrak{g}$, any finite sum of maximal dimension of the form $\sum_{\overline{g}\in \mathfrak{g}^{<\omega}}[\overline{g},A]$ is clearly an almost ideal. Consequently, there exists at most $n$ definably minimal almost ideals, up to equivalence. Applying Goursat's Lemma, this implies that, up to isogeny, there are at most finitely many definably minimal subgroups.
\section{Questions and possible extensions}
The questions that arise are probably more than the answers given by this paper. Clearly, lots of questions from \cite{deloro2023simple} still make sense in the finite-dimensional context, even if lots of tools of groups of finite Morley rank cannot be applied (or at least it is unclear if they hold). An example is the characterization of abelian groups of finite Morley rank (due to the non-existence, in general, of a connected component). On the other hand, we have seen, for example in Lemma \ref{DefBrack}, that Lie rings of finite dimension behave better than groups of finite dimension.\\
The first topic that needs more attention is the structure of non virtually connected Lie rings of finite dimension (in an IP theory). Already, it is not clear if a definable Lie ring of dimension $1$ should be virtually abelian-by-finite or not. Then, another problem will be to extend the known results in dimension $2,3$ and $4$. The linearisation theorem for non-virtually connected Lie rings and Lemma \ref{LinAlmAbe} will be fundamental, but it is far to be sufficient. Indeed, the pseudofinite Lie rings of dimension $1$ are finite-by-abelian-by-finite, but the dimension $2$ case is still unclear.
\begin{theorem}
    Let $\mathfrak{g}$ be a definable pseudofinite Lie ring of dimension $1$ and characteristic $>3$. Then, $\mathfrak{g}$ is finite-by-abelian-by-finite.
\end{theorem}
\begin{proof}
    Let $\mathfrak{g}=\prod \mathfrak{g}_i/\mathcal{U}$ for $\mathcal{U}$ an ultrafilter. It is sufficient to prove, by Lemma \ref{NilFromAlmNil}, that the almost center is infinite. Assume not. Up to work in $\mathfrak{g}/\widetilde{Z}(\mathfrak{g})$ that is still pseudofinite, we may assume $\widetilde{Z}(\mathfrak{g})=0$. This implies that the cardinality of the centraliser of any element in $\mathfrak{g}$ is bounded by $N\in \mathbb{N}$ \hbox{i.e.} the formula
    $$\forall x\in \mathfrak{g}\ \nexists x_1,...,x_{N+1} \bigwedge_{i\not=j} x_i\not=x_j \wedge [x_i,x]=0.$$
    If the characteristic is $0$, this is impossible. Consequently, we may assume that almost all the $\mathfrak{g}_i$ are finite Lie algebras over $\mathbb{F}_p$, by Los's Theorem. Again by Los's theorem, we may assume that for almost all the finite Lie algebras, the centraliser of any element has cardinality bounded by $N$. By Lemma \ref{FinSubAlg}, the cardinality of almost all the Lie algebras $\mathfrak{g}_i$ is bounded by $N'\in \mathbb{N}$. This implies that $\mathfrak{g}$ is finite, a contradiction. 
\end{proof}
An interesting result for pseudofinite absolutely simple Lie rings of finite dimension is the following characterization.
\begin{theorem}
    Let $\mathfrak{g}$ be a pseudofinite definable absolutely simple Lie ring of finite dimension and of characteristic $p>3$. Then, $\mathfrak{g}/\widetilde{Z}(\mathfrak{g})$ has a definable ideal of finite index that is simple and elementary equivalent to an ultraproduct of finite simple Lie algebras. 
\end{theorem}
\begin{proof}
    By Lemma \ref{NoInfIde}, there are no ideals of infinite index in $\mathfrak{g}/\widetilde{Z}(\mathfrak{g})$, and every ideal of finite index is definable. Being the almost center $0$, we may assume that $\mathfrak{g}$ has only definable index ideals of finite index. \\
    Let $\mathfrak{g}=\prod\mathfrak{g}_i/\mathcal{U} $ for finite Lie algebras $\mathfrak{g}_i$. We verify that almost all the Lie algebras $\mathfrak{g}_i$ are semi-simple. Assume not and let $\mathfrak{h}_i$ be a soluble proper ideal in $\mathfrak{g}_i$, that exists for almost all $i$. Define $\mathfrak{h}=\prod\mathfrak{h}_i/\mathcal{U}$. $\mathfrak{h}$ is an abelian ideal in $\mathfrak{g}$ that must be $0$, a contradiction.\\
Assume that almost all the $\mathfrak{g}_i$ are not simple. A minimal ideal $\mathfrak{h}_i$ in $\mathfrak{g}_i$ semi-simple is characteristically simple, and all proper ideals are nilpotent, by \cite{seligman1957characteristic}. $\mathfrak{h}=\prod \mathfrak{h}_i/\mathcal{U}$ is an ideal in $\mathfrak{g}$ and therefore definable of finite index. Up to work here, we may assume that every $\mathfrak{g}_i$ has only nilpotent ideals. Proceeding as in the first part, we may conclude that $\mathfrak{g}_i$ is semi-simple. This implies that the $\mathfrak{g}_i$ must be simple for almost every $i\in I$. Consequently, $\mathfrak{g}$ is simple: given a definable ideal $I$ in $\mathfrak{g}$, $I=\prod I_i/\mathcal{U}$ and any $I_i$ is an ideal in $\mathfrak{g}_i$. Then, by simplicity of $\mathfrak{g}_i$, it is either $0$ or $\mathfrak{g}_i$. Therefore, either $I=0$ or $I=\mathfrak{g}$.
\end{proof}
At present, a result similar to Wilson's theorem \cite{wilson1995simple} for simple pseudofinite groups is unaccessible. An instrument that seems necessary is a classification of finite simple Lie algebras, something we are far from having.\\
Another fundamental question is the classification, if it exists, of bad Lie algebras of dimension $3$ and $4$ in finite characteristic, in particular, show that bad Lie ring is a Lie algebra of finite dimension over a definable field, proving the "logarithmic CZ-conjecture", in dimensions $3$ and $4$.\\
Moving to hereditarily $\widetilde{\mathfrak{M}}_c$-Lie ring, it is still open if the Radical ideal is soluble. This question is also open for groups, and no known techniques seem to be able to attack this problem.


\end{document}